\documentclass[reqno,a4paper,11pt]{amsart}
\usepackage[T1]{fontenc}
\usepackage{mathrsfs}
\usepackage{amssymb}
\usepackage{libertine}
\usepackage[libertine]{newtxmath}
\usepackage[utf8]{inputenc}
\usepackage[english]{babel}
\usepackage{url}
\usepackage[backref=page, colorlinks=true]{hyperref}
\hypersetup{linkcolor=blue, citecolor=blue, filecolor=black, urlcolor=blue}
\usepackage{comment}

\usepackage{cancel}

\renewcommand*{\backrefalt}[4]{   
	\ifcase #1 (Not cited.)        
	\or        (Cited on page~#2.) 
	\else      (Cited on pages~#2.)
	\fi}

\newcommand{\crosses}[1]{%
	\ifcase#1\relax
	\or
	\rslash\or
	\rslash\mskip-5.5mu\rslash\or
	\rslash\mskip-5.5mu\rslash\mskip-5.5mu\rslash%
	\fi
}
\newcommand{\rslash}{\raisebox{.15ex}{/}}
\usepackage[usenames,dvipsnames]{xcolor}
\usepackage{color}
\usepackage{geometry}
\usepackage[all]{xy}
\usepackage{enumitem}

\usepackage{slashed}
\usepackage{bbm}
\usepackage{cancel}

\usepackage{cleveref}
\usepackage{tikz}
\usepackage{tikz-cd}

\usepackage{xargs}  
\usepackage{soul}

\usepackage{amsmath}
\usepackage{graphicx}
\usepackage{amsthm}
\usepackage{latexsym}
\usepackage{amsfonts}
\usepackage{bbm}
\usepackage{mathrsfs}
\usepackage{cases}

\calclayout

\makeatletter
\@namedef{subjclassname@2020}{%
	\textup{2020} Mathematics Subject Classification}
\makeatother

\numberwithin{equation}{section}

\theoremstyle{plain}
\newtheorem{lemma}{Lemma}[section]
\newtheorem{proposition}[lemma]{Proposition}
\newtheorem{proposition/definition}[lemma]{Proposition/Definition}
\newtheorem*{theorem*}{Theorem}
\newtheorem{theorem}[lemma]{Theorem}
\newtheorem{corollary}[lemma]{Corollary}

\theoremstyle{definition}
\newtheorem{definition}[lemma]{Definition}
\newtheorem{remark}[lemma]{Remark}
\newtheorem{example}[lemma]{Example}

\DeclareRobustCommand{\SkipTocEntry}[5]{}

\DeclareMathOperator{\id}{id}
\DeclareMathOperator{\image}{im}
\DeclareMathOperator{\rank}{rank}

\DeclareMathOperator{\Aut}{Aut}

\newcommand{\D}{\mathop{}\!\mathrm{d}}

\newcommand{\hooklongrightarrow}{\lhook\joinrel\longrightarrow}
\newcommand{\pr}{\mathrm{pr}}
\newcommand{\at}[1]{\big|_{#1}}
\newcommand{\At}[1]{\Big|_{#1}}

\newcommand{\Hom}{\operatorname{\mathsf{Hom}}}

\newcommand{\Lie}{\mathscr{L}}
\newcommand{\bla}{\langle \hspace{-2.7pt} \langle}
\newcommand{\bra}{\rangle\hspace{-2.7pt} \rangle}

\newcommand{\calD}{\mathcal{D}}
\newcommand{\calE}{\mathcal{E}}
\newcommand{\calF}{\mathcal{F}}
\newcommand{\calG}{\mathcal{G}}
\newcommand{\calH}{\mathcal{H}}

\newcommand{\calJ}{\mathcal{J}}
\newcommand{\calK}{\mathcal{K}}
\newcommand{\calL}{\mathcal{L}}

\newcommand{\calQ}{\mathcal{Q}}

\newcommand{\calS}{\mathcal{S}}

\newcommand{\scrL}{\mathscr{L}}

\newcommand{\bbD}{\mathbb{D}}

\newcommand{\bbR}{\mathbb{R}}

\newcommand{\bbT}{\mathbb{T}}

\newcommand{\frakX}{\mathfrak{X}}

\newcommand{\rmD}{\mathrm{D}}

\newcommand{\rmd}{\mathrm{d}}

\newcommand{\ldsb}{[\![}
\newcommand{\rdsb}{]\!]}
\newcommand{\ldab}{\langle\!\langle}
\newcommand{\rdab}{\rangle\!\rangle}
\renewcommand{\theta}{\vartheta}

\allowdisplaybreaks

\newcounter{proof:lem:centralizers}

\title{Weak Dual Pairs in Dirac--Jacobi Geometry}

\author[J.~Schnitzer]{Jonas Schnitzer}
\address{Department of Mathematics, University of Freiburg, Ernst-Zermelo-Straße 1, D-79104 Freiburg, Germany.}
\email{\href{mailto:jonas.schnitzer@math.uni-freiburg.de}{jonas.schnitzer@math.uni-freiburg.de}}

\author[A.~G.~Tortorella]{Alfonso G.~Tortorella}
\address{Department of Mathematics, KU Leuven, Celestijnenlaan 200B - 3001 Leuven, Belgium.}
\email{\href{mailto:alfonsogiuseppe.tortorella@kuleuven.be}{alfonsogiuseppe.tortorella@kuleuven.be}}



\begin{document}

\begin{abstract}
	Adopting the omni-Lie algebroid approach to Dirac--Jacobi structures, we propose and investigate a notion of weak dual pairs in Dirac--Jacobi geometry.
	Their main motivating examples arise from the theory of multiplicative precontact structures on Lie groupoids.
	Among other properties of weak dual pairs, we prove two main results.
	1) We show that the property of fitting in a weak dual pair defines an equivalence relation for Dirac--Jacobi manifolds.
	So, in particular, we get the existence of self-dual pairs and this immediately leads to an alternative proof of the normal form theorem around Dirac--Jacobi transversals.
	2) We prove the characteristic leaf correspondence theorem for weak dual pairs paralleling and extending analogous results for symplectic and contact dual pairs.
	Moreover, the same ideas of this proof apply to get a presymplectic leaf correspondence for weak dual pairs in Dirac geometry (not yet present in literature).
\end{abstract}

\maketitle

\tableofcontents

\section{Introduction}
\label{sec:intro}

Descending from Sophus Lie's function groups~\cite{lie1890theorie}, \emph{symplectic dual pairs} (namely dual pairs in Poisson geometry) found their modern formulation in the works of Weinstein on the local structure of Poisson manifolds~\cite{weinstein1983} and Howe in representation theory~\cite{howe1989remarks}.
A symplectic dual pair consists of two Poisson maps defined on the same symplectic manifold and satisfying a certain orthogonality condition.
A prototypical example is formed by source and target maps of a symplectic groupoid.
Additionally, they naturally pop up in connection with (super)integrable Hamiltonian systems, moment maps and reduction of mechanical systems with symmetries (see, e.g.,~\cite{ortega2013momentum} and references therein).
Symplectic dual pairs can be interpreted as a sort of ``generalized morphisms'' of Poisson manifolds.
Indeed, if two Poisson manifolds fit in a dual pair, then their local structure is very closely related: their leaf spaces identify so that corresponding symplectic leaves have the same transversal structure.
Actually, these properties of symplectic dual pairs lead to the introduce and investigate Morita equivalence of Poisson manifolds~\cite{xu1991}.

In~\cite{FM2018} the notion of \emph{(weak) dual pairs} in Dirac geometry has been also introduced and investigated.
Indeed, Dirac geometry~\cite{C1990} generalizes Poisson geometry representing the geometric framework for studying constrained Hamiltonian systems.
A Dirac structure on a manifold $M$ is a Lagrangian subbundle of the generalized tangent bundle $\bbT M=TM\oplus T^\ast M$ whose sections are closed under the Dorfman bracket.
They generalize involutive distributions, Poisson and presymplectic structures, and also arise as the infinitesimal counterpart of presymplectic groupoids.
As for their local structure, every Dirac manifold admits a singular foliation by presymplectic leaves, with each leaf carrying a transverse Poisson structure.

Recently, in~\cite{STV2019}, Blaga, Salazar, Vizman and the second author proposed and studied a new notion of dual pairs in Jacobi geometry, called \emph{contact dual pairs}, conceptually grounded on the close relation between symplectic/Poisson and contact/Jacobi geometry.
Indeed, on the one hand, contact structures are seen as the odd-dimensional analogue of symplectic structures.
While, on the other hand, Jacobi structures (independently introduced by Kirillov~\cite{K1976} and Lichnerowicz~\cite{lichnerowicz1978varietes}) can be seen as the (possibly degenerate) contravariant generalization of contact structures exactly like Poisson structures are the (possible degenerate) contravariant generalization of symplectic structures.
This point of view, that will be made precise in the paper, consists in adopting the bundle approach of~\cite{marle1991jacobi} and seeing contact/Jacobi geometry as symplectic/Poisson geometry on the \emph{gauge algebroid} $DL$ of a line bundle $L\to M$.

This paper introduces the concept of weak dual pair in Dirac--Jacobi and systematically studies its first properties and applications.
A \emph{Dirac--Jacobi structure} on a line bundle $L\to M$ is a Lagrangian subbundle of $\bbD L:=DL\oplus J^1L$, the \emph{omni-Lie algebroid}~\cite{CHEN2010799} of $L$, such that the module of its section is closed under a Dorfman-like bracket.
They encompass, unifying and generalizing, Jacobi and precontact structures as well as locally conformal Dirac structures~\cite{wade2004locally}.
Further, they arise as infinitesimal counterparts of precontact groupoids~\cite{ponte2006integration}.
Their local structure turns out to be richer than the one of Dirac structures.
Indeed, every Dirac--Jacobi manifold admits a singular characteristic foliation whose leaves inherit a precontact or a locally conformal presymplectic (lcps) structure.
Additionally, each characteristic leaf brings a transverse Dirac--Jacobi structure (unique up to isomorphism) whose type depends on the type of the leaf: it is homogeneous Poisson for precontact leaf and Jacobi for lcps leaves~\cite{DirJacBun}.
In recent years, Dirac--Jacobi geometry has emerged as the right conceptual framework for describing and investigating \emph{generalized contact structures}, namely generalized complex structures on odd-dimensional manifolds (cf., e.g., \cite{IW2005,poon2011generalized,schnitzer2020local,GenConBun} and references therein).

In this introduction we only give rough statements of the main results.
They will be made precise and better explained in the body of the paper.
\begin{theorem*}[{\bf A}]
	For Dirac--Jacobi manifolds $(M_i,L_i,\calL_i)$, $i=0,1$, the property of fitting in a weak dual pair 
	\begin{equation}
		\label{eq:intro:wdp}
		\begin{tikzcd}
			(M_0,L_0,\calL_0)&&(M,L,\operatorname{Gr}\varpi)\arrow[ll, swap, "\Phi_0"]\arrow[rr, "\Phi_1"]&&(M_1,L_1,\calL_1^\mathrm{opp})
		\end{tikzcd}
	\end{equation}
	defines an equivalence relation.
	In particular, each Dirac--Jacobi manifold $(M_0,L_0,\calL_0)$  fits in a \emph{self-dual pair}
	\begin{equation*}
		\begin{tikzcd}
			(M_0,L_0,\calL_0)&&(M,L,\operatorname{Gr}\varpi)\arrow[ll, swap, "\Phi_0"]\arrow[rr, "\Phi_1"]&&(M_0,L_0,\calL_0^\mathrm{opp}).
		\end{tikzcd}
	\end{equation*}
\end{theorem*}
\noindent
As direct applications of the existence of self-dual pairs for Dirac--Jacobi manifolds, we get an alternative proof of the Normal Form Theorem around Dirac--Jacobi transversals.

\begin{theorem*}[{\bf B}]
	Dirac--Jacobi manifolds have a normal form around Dirac--Jacobi transversals.
\end{theorem*}
\noindent
Notice that the latter result was proved by the first author in~\cite{NorForJac} with different techniques.
Finally, we prove a Characteristic Leaf Correspondence Theorem for weak dual pairs in Dirac--Jacobi geometry.

\begin{theorem*}[{\bf C}]
	If two Dirac--Jacobi manifolds fit in a weak dual pair like~\eqref{eq:intro:wdp}, with connected fibers, then their characteristic leaf spaces can be identified so that corresponding leaves are of the same type, i.e.~both precontact or both lcps, and have the same transverse structure.
\end{theorem*}
\noindent
This property of weak dual pairs suggests the introduction and investigation of Morita equivalence in the Dirac--Jacobi setting: this suggestive idea will be pursued by the authors in a future work.
Moreover, mutatis mutandis, the same arguments used to prove this last property allow to get a Presymplectic Leaf Correspondence Theorem for weak dual pairs in Dirac geometry (in the sense of~\cite{FM2018}).
\begin{theorem*}[{\bf D}]
	If two Dirac manifolds fit in a weak dual pair (in the sense of~\cite{FM2018}), with connected fibers, then their presymplectic leaf spaces can be identified so that corresponding presymplectic leaves have the same transverse Dirac structure.
\end{theorem*}

The paper is organized as follows.
Section~\ref{sec:DiracJacobiReview} reviews the line bundle approach to Dirac--Jacobi geometry including the Atiyah algebroid, the first jet bundle and the omni-Lie algebroid of a line bundle.
Section~\ref{sec:products} describes in detail products and fiber products in the category of line bundles with regular line bundle morphisms, as well as products of Dirac--Jacobi bundles which will be crucial in defining operations on weak dual pairs.
Section~\ref{sec:weak_dual_pairs} introduces weak dual pairs in Dirac--Jacobi geometry with motivating examples coming from over-precontact groupoids.
Further, Section~\ref{sec:properties_operations} discusses their first properties and the relevant operations, namely composition and pull-back along Dirac--Jacobi transversals.
Section~\ref{sec:self-dual_pair} presents a proof of the existence of self-dual pairs.
As direct applications of the latter we also an alternative proof of the Normal Form Theorem around Dirac--Jacobi transversals.
Finally, in Section~\ref{sec:characteristic_leaf_correspondence}, we discuss in detail the Characteristic Leaf Correspondence Theorem for weak dual pairs in Dirac--Jacobi geometry.
The paper also provides, in Appendix~\ref{sec:presymplectic_leaf_correspondence}, a sketch of the Presymplectic Leaf Correspondence Theorem for weak dual pairs in Dirac geometry (not yet present in literature).

\section{A Review of Dirac--Jacobi Geometry}
\label{sec:DiracJacobiReview}

We review the line bundle approach to precontact, Jacobi and Dirac--Jacobi geometry.
According to this approach, slightly generalizing $\calE^1(M)$ Dirac structures~\cite{Wade2000}, a Dirac--Jacobi structure~\cite{DirJacBun} is a Dirac-like structure on the \emph{omni-Lie algebroid} $\bbD L:=DL\oplus J^1L$ of a line bundle $L\to M$.
So that, in particular, precontact/Jacobi structures are seen as presymplectic/Poisson structures on $DL\Rightarrow M$, the \emph{gauge algebroid} of $L$.
Roughly speaking, the guiding principle is that, as detailed below, functions on a manifolds $M$ are replaced by sections of $L$, vector fields on $M$ are replaced by derivations of $L$, i.e.~sections of  $DL$, the \emph{gauge algebroid} of $L$, differential $1$-forms corresponds to sections of $J^1L$, the \emph{1st jet bundle} of $L$, and so forth.

%
%

\subsection{Derivations and the Atiyah algebroid of a vector bundle}
\label{subsec:Atiyah_algebroid}

Even though derivations of vector bundles are very classical topics in differential geometry, we give the basic definitions and properties, which will be necessary throughout this section.
This section is far from being a complete introduction to this topic.
A more detailed discussion can be found for example in \cite{mackenzie_2005} and
\cite{Rubtsov_1980}. 

Throughout this paper, given a vector bundle morphism $\Phi:E\to F$, covering a smooth map $\varphi:M\to N$, we say that $\Phi$ is \emph{regular} if it is fiberwise invertible, i.e.~$\Phi_x:=\Phi|_{E_x}\colon E_x\to F_{\varphi(x)}$ is a linear isomorphism, for any $x\in M$.
Further, a regular $\Phi$ determines a \emph{pull-back of sections} $\Phi^\ast\colon\Gamma(F)\to\Gamma(E),\ s\mapsto\Phi^\ast s$, where
\begin{equation*}
	(\Phi^\ast s)_x=\Phi_x^{-1}(s_{\varphi(x)}),\ \ \text{for all}\ x\in M.
\end{equation*}

A \emph{derivation} of a vector bundle $E\to M$ is an $\bbR$-linear map $\Delta\colon \Gamma(E)\to\Gamma(E)$ such that there is a (unique) $\sigma_\Delta\in\frakX(M)$, the \emph{symbol} of $\Delta$, satisfying the following Leibniz rule, for all $e\in\Gamma(E)$ and $f\in C^\infty(M)$,
\begin{equation*}
	\Delta(fe)=\sigma_\Delta(f)e+f\square e.
\end{equation*}
The set of derivations of $E$, denoted by $\calD E$, has natural structures of $C^\infty(M)$-module and Lie algebra.

The Lie algebra $\calD E$ can also be understood as the Lie algebra of the infinitesimal vector bundle automorphisms of $E\to M$.
Indeed, a derivation $\Delta\in\calD E$ generates, as its flow, the $1$-parameter group $\Phi_t$ of local vector bundle automorphisms of $E\to M$ which is uniquely determined by
\begin{equation}
	\label{eq:Lie_derivative_formula_0}
	\frac{\rmd}{\rmd t}(\Phi_t^\ast e)=\Phi_t^\ast(\Delta e),\ \text{for all}\ e\in\Gamma(E).
\end{equation}
Notice that, in particular, the $1$-parameter group $\varphi_t$ of local diffeomorphisms of $M$ covered by $\Phi_t$ coincides with the flow of $\sigma_\Delta$.
Conversely, each $1$-parameter group $\Phi_t$ of local vector bundle automorphisms of $E\to M$ arises in this way from a unique derivation $\Delta$ of $E\to M$.

The $C^\infty(M)$-module of derivations of $E$ can also be seen as the module of sections of a vector bundle $DE\to M$.
For any $x\in M$, its fiber $D_xE$ consists of the \emph{derivations of $E$ at $x$}, i.e.~those $\bbR$-linear maps $\delta\colon \Gamma(E)\to E_x$ such that there is a (unique) $\sigma_\delta\in T_xM$, called the \emph{symbol of $\delta$}, satisfying the following Leibniz rule $\delta(fe)=\sigma_\delta(f)e_x+f(x)\delta e$, for all $e\in\Gamma(E)$ and $f\in C^\infty(M)$.
 
The vector bundle $DE$ becomes a \emph{transitive} Lie algebroid, called the \emph{gauge or Atiyah algebroid} of $E$, with the commutator of derivations as Lie bracket, and the symbol map $\sigma\colon DE\to TM,\ \delta\mapsto\sigma_\delta$, as anchor.
Further, the \emph{tautological representation} $\nabla$ of $DE$ in $E$ is defined by $\nabla_\Delta e=\Delta e$, for all $\Delta\in\calD E$, and $ e\in\Gamma(E)$.

The construction of the gauge algebroid is functorial: it determines a functor from the category of vector bundles, with regular vector bundle morphisms, to the category of Lie algebroids.
Indeed, each regular VB morphism $\Phi\colon E\to F$ gives raise to the Lie algebroid morphism $D\Phi\colon DE\to DF,\ \delta\mapsto (D\Phi)\delta$, defined by
\begin{equation*}
	((D\Phi)\delta)s=\Phi(\delta(\Phi^\ast s)),\ \ \text{for all}\ s\in\Gamma(F).
\end{equation*}
Further, $\Delta\in\calD E$ and $\square\in\calD F$ are \emph{$\Phi$-related} if $(D\Phi)\circ\Delta=\square\circ\varphi$ (or equivalently $\Phi^\ast\circ\square=\Delta\circ\Phi^\ast$).

\begin{remark}
	\label{rem:derivations_of_line bundles}
	Any derivation of $E$ is actually a first order differential operator.
	So, $DE$ is naturally a vector subbundle of $(J^1E)^\ast\otimes E$, the vector bundle whose sections are the first order differential operators from $\Gamma(E)$ to itself.
	Above $J^1E\to M$ denotes the first jet bundle of $E$.
	The converse, i.e.~the equality $DE=(J^1E)^\ast\otimes E$, holds if and only if $E\to M$ is a line bundle.
\end{remark}

\subsection{The Line Bundle Approach to Contact and Jacobi Geometry}
Fix a line bundle $L\to M$.
The \emph{der-complex of $L$} (see~\cite{Rubtsov_1980}) is the de Rham complex $(\Omega_D^\bullet(L),\rmd_D)$ of the Atiyah algebroid $DL$ with values in its tautological representation in $L$.
The de Rham differential $\rmd_D$ is also called the \emph{der-differential}, and the elements of $\Omega_D^\bullet(L)=\Gamma(\wedge^\bullet(DL)^\ast\otimes L)$ are referred to as \emph{$L$-valued Atiyah forms}.
Then, there exists a Cartan calculus on $\Omega_D^\bullet(L)$ whose structural operations are the der-differential and, for any $\Delta\in\calD L$, the \emph{contraction} and the \emph{Lie derivative} along $\Delta$
\begin{equation*}
	\rmd_D:\Omega_D^\bullet(L)\to\Omega_D^{\bullet+1}(L),\qquad \iota_\Delta:\Omega_D^\bullet(L)\to\Omega_D^{\bullet-1}(L),\qquad\scrL_\Delta=\colon\Omega_D^\bullet(L)\to\Omega_D^\bullet(L).
\end{equation*}
Moreover, these operations are related through the following identities, for any $\square,\Delta\in\calD L$,
\begin{equation*}
	\label{eq:Cartan_identities}
	\begin{gathered}{}
		[\rmd_D,\iota_\Delta]=\scrL_\Delta,\quad
		[\iota_\square,\scrL_\Delta]=\iota_{[\square,\Delta]},\quad [\scrL_\square,\scrL_\Delta]=\scrL_{[\square,\Delta]},
		\quad
		[d_D,d_D]=[d_D,\scrL_\Delta]=[\iota_\Delta,\iota_\Delta]=0,
	\end{gathered}
\end{equation*}
where $[-,-]$ is the graded commutator.
Note that, as it can be computed directly, $\scrL_\mathbbm{1}=\id_{\Omega_D^\bullet(L)}$ where $\mathbbm{1}\in\calD L$ denotes the identity map, i.e.~$\mathbbm{1}e=e$.
This means nothing else that the der-complex $(\Omega^\bullet_D(L),\rmd_D)$ is acyclic, with a contracting homotopy given by $\iota_{\mathbbm{1}}$, i.e.
\begin{equation}
	\label{eq:contracting_homotopy}
	[\rmd_D,\iota_{\mathbbm{1}}]=\id_{\Omega_D^\bullet(L)}.
\end{equation}

\begin{remark}
	Since $DL=(J^1L)^\ast\otimes L$, the $L$-valued pairing $\langle-,-\rangle\colon DL\otimes J^1L\to L$ induces a VB isomorphism $J^1L\overset{\sim}{\to}(DL)^\ast\otimes L$ and this identifies $\Gamma(J^1L)$ with $\Omega_D^1(L)$.
Further, note also that $\Gamma(L)=\Omega_D^0(L)$.
In view of this, $\Omega_D^0(L)\to\Omega_D^1(L),\ \lambda\mapsto d_D\lambda$, coincides with the $1$-st jet prolongation $\Gamma(L)\to\Gamma(J^1L),\ \lambda\mapsto j^1\lambda$.
\end{remark}

The construction of the der-complex is functorial: it gives a contravariant functor from the category of line bundles, with regular LB morphisms, to the category of dg-modules.
Indeed, each regular LB morphism $\Phi\colon L\to L^\prime$ induces a dg-module morphism $\Phi^\ast\colon (\Omega_D^\bullet(L^\prime),\rmd_D)\to(\Omega_D^\bullet(L),\rmd_D),\ \eta^\prime\mapsto\Phi^\ast\eta^\prime,$ defined by, for all $\eta'\in\Omega_D^k(L^\prime)$, $x\in M$, and $\delta_1,\ldots,\delta_k\in D_xE$,
\begin{equation*}
	(\Phi^\ast\eta')(\delta_1,\ldots,\delta_k)=\Phi_x^{-1}(\eta'((D\Phi)\delta_1,\ldots,(D\Phi)\delta_k)).
\end{equation*}

\subsubsection{Precontact Structures}
\label{sec:precontact_structures}
Let us recall the line bundle approach to precontact structures.
For any manifold $M$, by the relation $\calH=\ker\theta$, hyperplane distributions $\calH$ on $M$are in one-to-one correspondence with no-where zero $1$-forms $\theta$ on $M$ with values in some line bundle $L\to M$ unique up to LB isomorphisms.
Assume now that $\calH$ and $\theta$ are related as above.
Then their curvature form $\omega\in\Gamma(\wedge^2 \calH^\ast\otimes L)$ is defined by $\omega(X,Y)=\theta([X,Y])$, for all $X,Y\in\Gamma(\calH)$.
If $\omega$ is non-degenerate, i.e.~the VB morphism $\omega^\flat:C\to C^\ast\otimes L,\ X\mapsto\omega(X,-)$, is an isomorphism, the $1$-form $\theta$ and the corresponding hyperplane distribution $\calH$ are said to be \emph{contact}.
Now precontact structures arise as the generalization of contact structures when we drop the non-degeneracy condition.
More precisely we have the following.
\begin{definition}
	\label{def:precontact}
	A \emph{precontact structure} $(L,\theta)$ on a manifold $M$ consists of a line bundle $L\to M$ and a $1$-form $\theta\in\Omega^1(M;L)$.
	Then a \emph{precontact manifold} $(M,L,\theta)$ is a manifold equipped with a precontact structure.
\end{definition}

In this paper we will take advantage of the line bundle approach to (pre)contact structures and so we will look at them as (pre)symplectic structures on the gauge algebroid of a line bundle.

\begin{definition}
	\label{def:presymplectic_Atiyah}
	Let $L\to M$ be a line bundle.
	An $L$-valued Atiyah $2$-form $\varpi\in\Omega_D^2(L)$ is said to be \emph{presymplectic} if $\rmd_D\varpi=0$.
	Additionally, $\varpi$ is said to be \emph{symplectic} if it is also non-degenerate, i.e.~the VB morphism $\varpi^\flat:DL\to (DL)^\ast\otimes L=J^1L,\ \delta\mapsto\varpi(\delta,-)$ is an isomorphism.
\end{definition}

\begin{proposition}[{\cite[Proposition~3.3]{DirJacBun}}]
	\label{prop:Atiyah_presymplectic_forms}
	For any line bundle $L\to M$, the relation $\theta\circ\sigma=\iota_{\mathbbm{1}}\varpi$ establishes a 1-1 correspondence between $L$-valued (pre)contact forms $\theta$ and $L$-valued (pre)symplectic Atiyah forms $\varpi$.
\end{proposition}

\subsubsection{Jacobi Structures}
Fix a line bundle $L\to M$.
Denote by $\calD^\bullet L$ the graded space of skew-symmetric multi-derivations of $L$ so that, in particular, $\calD^0L=\Gamma(L)$ and $\calD^1 L=\Gamma(D L)$.
By Remark~\ref{rem:derivations_of_line bundles}, they coincide with the skew-symmetric first order multi-differential operators from $\Gamma(L)$ to itself, i.e.~$\calD^\bullet L=\Gamma(\wedge^\bullet(J^1L)^\ast\otimes L)$.
The Lie algebroid structure of $DL$ and its tautological representation in $L$ extend to a graded Lie bracket $[-,-]$, called \emph{Schouten--Jacobi bracket}~\cite{le2018deformations}, making $(\calD^\bullet L)[1]$ into a graded Lie algebra.

Recall that a \emph{Jacobi structure} on a line bundle $L\to M$ consists of a Lie bracket $\{-,-\}$ on $\Gamma(L)$ that is a derivation of $L$ in each entry.
Now, as it is easy to see, the following relation
$$\{\lambda,\mu\}=\calJ(j^1\lambda,j^1\mu),\quad\text{for all}\ \lambda,\mu\in\Gamma(L)$$establishes a one-to-one correspondence between the Jacobi structures $\{-,-\}$ on $L\to M$ and the \emph{Maurer--Cartan elements} $\calJ$ of $((\calD^\bullet L)[1],[-,-])$, i.e.~$\calJ\in\calD^2L$ satisfying $[\calJ,\calJ]=0$.
Further, a Jacobi structure $\calJ=\{-,-\}$ on $L\to M$ is \emph{non-degenerate} if the VB morphism $J^\sharp\colon J^1L\to DL,\ \eta\mapsto J(\eta,-)$, is invertible.

\begin{proposition}[{\cite[Prop.~3.6]{DirJacBun}}]
	\label{prop:non-degenerate_Jacobi_structure}
	For any line bundle $L\to M$, the relation $\varpi^\flat=(\calJ^\sharp)^{-1}$ establishes a 1-1 correspondence between symplectic Atiyah forms $\varpi\in\Omega^2_D(L)$ and non-degenerate Jacobi structures $\calJ\in\calD^2L$.
\end{proposition}

\subsection{The Omni-Lie Algebroid of a Line Bundle and its Automorphisms}

The omni-Lie algebroid was first introduced in \cite{CHEN2010799} in order to connect Dirac-like subbundles to Lie 
algebroids. Our aim is a little bit different, 
since we are interested in the Dirac analogue in Jacobi geometry, which are not Lie 
algebroids. Nevertheless, we adapt the notion of Chen and Liu from \cite{CHEN2010799} and combine them with notions coming from Dirac geometry. This sections 
follows exactly the same lines as \cite{NorForJac}.

\begin{definition}
	\label{def:omni_Lie_algbd}
	The \emph{omni-Lie algebroid} of a line bundle $L\to M$ is the vector bundle $\mathbb{D}
	L:=DL\oplus J^1 L\to M$ equipped with 
	\begin{enumerate}[label={\arabic*)}]
		\item the \emph{Dorfman-like bracket} $\ldsb-,-\rdsb:\Gamma(\bbD L)\times\Gamma(\bbD L)\to\Gamma(\bbD L)$ defined by
		\begin{align*}
			\ldsb(\Delta_1,\psi_1) ,(\Delta_2,\psi_2 )\rdsb
			=([\Delta_1,\Delta_2],\Lie_{\Delta_1} \psi_2- \iota_{\Delta_2}\D_L\psi_1)
		\end{align*}		 
		\item the \emph{non-degenerate $L$-valued symmetric product} $\ldab-,-\rdab:\bbD L\otimes\bbD L\to L$ defined by
		\begin{align*}
			\bla (\Delta_1,\psi_1) ,(\Delta_2,\psi_2 )\bra := \psi_1(\Delta_2)+\psi_2(\Delta_1)
		\end{align*}
		\item the representation given by the \emph{canonical projection} $\pr_D\colon \mathbb{D}L\to DL,\ (\delta,\alpha)\mapsto\delta$.
	\end{enumerate} 
\end{definition}
We will denote the projection on the second summand by $\pr_J\colon\bbD L\to J^1L,\ (\delta,\alpha)\mapsto\alpha$.
The omni-Lie algebroid $\bbD L$, equipped with the above structural operations, represents the prototypical example of what is called a \emph{Courant-Jacobi algebroid}~\cite{grabowski2013graded,grabowski2002graded,dacosta2004dirac}.
%
Let us now discuss Courant-Jacobi automorphisms of $\bbD L$. 
\begin{definition}
	\label{def:CJ-Aut}
	A \emph{Courant-Jacobi 
		automorphism} $(F,\Phi)$ of the omni-Lie algebroid $\bbD L$ consists of VB automorphisms $F:\bbD L\to\bbD L$ and $\Phi:L\to L$ covering the same diffeomorphism $\varphi:M\to M$ and preserving the structure maps of the omni-Lie algebroid, i.e., for all $u,v\in\Gamma(\bbD L)$,
	\begin{equation*}
		(D\Phi) (\pr_D u)=\pr_D (Fu),\qquad \Phi\ldab u,v\rdab=\ldab Fu,Fv\rdab,\qquad F^\ast \ldsb u,v\rdsb=\ldsb F^\ast u,F^\ast v\rdsb.
	\end{equation*}
\end{definition}

	As it is easy to see, the Courant-Jacobi automorphisms of the omni-Lie algebroid $\bbD L$ form a group under composition that we will denote by $\Aut_{\sf{CJ}}(\bbD L)$.
	Let us now describe the group structure of $\Aut_{\sf{CJ}}(\bbD L)$.
	Each LB automorphism of $L$ gives rise to a Courant--Jacobi automorphism of $\bbD L$ as in the group embedding 
	\begin{equation*}
		\Aut_{\sf{VB}}(L)\hooklongrightarrow\Aut_{\sf{CJ}}(\bbD L),\ \Phi\longmapsto(\bbD\Phi,\Phi).
	\end{equation*}
	where $\bbD \Phi:\bbD L\to \bbD L$ is defined by $(\mathbb{D}\Phi)(\Delta,\alpha):=(D\Phi(\Delta), (D\Phi^{-1})^\ast \alpha)$, for all $(\Delta,\alpha)\in\bbD L$.
	Each closed $L$-valued Atiyah $2$-form $B$ gives rise to a Courant--Jacobi automorphism of $\bbD L$, called \emph{gauge transformation by $B$} (or \emph{$B$-field transformation}), as in the group embedding
	\begin{equation}
		\label{eq:gauge_transformation}
		\Omega^2_{D,\text{cl}}(L)\hooklongrightarrow\Aut_{\sf{CJ}}(\bbD L),\ B\longmapsto(\exp(B),\id_L),
	\end{equation}
	where $\exp(B):\bbD L\to\bbD L$ is defined by $\exp(B)(\Delta,\alpha):=(\Delta,\alpha+\iota_\Delta B)$, for all $(\Delta,\alpha)\in\bbD L$.
	Finally, understanding the above two group embeddings and the natural action of $\Aut_{\sf{VB}}(L)$ on $\Omega^2_{D,\text{cl}}(L)$, one gets that Courant--Jacobi automorphism group of $\bbD L$ is the following semidirect product
	\begin{equation}
		\label{eq:CJ-Aut:semidirect_product}
		\Aut_{\sf{CJ}}(\bbD L)\simeq \Aut_{\sf{VB}}(L)\ltimes\Omega^2_{D,\text{cl}}(L).
	\end{equation}

Let us now discuss infinitesimal Courant-Jacobi automorphisms of the omni-Lie algebroid.
The \emph{infinitesimal Courant--Jacobi automorphisms} of $\bbD L$ are the infinitesimal generators of Courant--Jacobi automorphisms of $\bbD L$, an so they form a Lie algebra denoted by $\mathfrak{aut}_{\sf CJ}(\bbD L)$.
By Definition~\ref{def:CJ-Aut}, the elements of $\mathfrak{aut}_{\sf CJ}(\bbD L)$ are pairs $(\square,\Delta)$ consisting of derivations $\square\in\calD(\bbD L)$ and $\Delta\in\calD L$ that have the same symbol, i.e.~$\sigma_\square=\sigma_\Delta\in\frakX(M)$, and are compatible with the structure maps of $\bbD L$, i.e.
\begin{equation*}
	\pr_D(\square u)=[\Delta,\pr_D u],\quad \Delta\ldab u,v\rdab=\ldab\square u, v\rangle+\langle u,\square v\rdab,\quad \square\ldsb u,v\rdsb=\ldsb\square u,v\rdsb+\ldsb u,\square v\rdsb,
\end{equation*}
for all $u,v\in\Gamma(\bbD L)$.
	Let us now describe the group structure of $\Aut_{\sf{CJ}}(\bbD L)$.
	Each derivation of $L$ gives rise to an infinitesimal Courant--Jacobi automorphism of $\bbD L$ as in the Lie algebra embedding 
	\begin{equation*}
		\calD L\hooklongrightarrow\mathfrak{aut}_{\sf{CJ}}(\bbD L),\ \Delta\longmapsto(\ldsb(\Delta,0),-\rdsb,\Delta).
	\end{equation*}
	Each closed $L$-valued Atiyah $2$-form gives rise to an infinitesimal Courant--Jacobi automorphism of $\bbD L$ as in the abelian Lie algebra embedding
	\begin{equation*}
		\Omega_{D,\text{cl}}^2(L)\hooklongrightarrow\mathfrak{aut}_{\sf CJ}(\bbD L),\ \rmd_D\alpha\longmapsto(\ldsb(0,\alpha),-\rdsb,0).
	\end{equation*}
	Finally, understanding the natural Lie algebra action $\calD(L)\times\Omega^2_{D,\text{cl}}(L)\to \Omega^2_{D,\text{cl}}(L),\ (\Delta,B)\mapsto\scrL_\Delta B$, and the above two Lie algebra embeddings, the group isomorphism~\eqref{eq:CJ-Aut:semidirect_product} induces the Lie algebra isomorphism
\begin{equation}
	\label{eq:CJ-inf_aut:semidirect_product}
	\mathfrak{aut}_{\sf{CJ}}(\bbD L)\simeq \calD L\ltimes\Omega^2_{D,\text{cl}}(L).
\end{equation}
Moreover, for future reference, let us notice that the 1-parameter group of local Courant--Jacobi automorphisms generated by $(\Delta,B)\in \calD L\ltimes\Omega^2_{D,\text{cl}}(L)\simeq\mathfrak{aut}_{\sf CJ}(\bbD L)$ is given by the pair
\begin{equation*}
	\left(\exp\left(\int_0^t(\Phi_{-\epsilon}^\ast B)\rmd\epsilon\right)\circ\bbD \Phi_t,\Phi_t\right),
\end{equation*}
where $\Phi_t:L\to L$ denotes 1-parameter group of local line bundle automorphims generated by $\Delta$

\begin{remark}
	\label{rem:infinitesimal_CJ_aut_LieAlgebra}
	Since the der-complex is acyclic (see~Equation~\eqref{eq:contracting_homotopy}), each infinitesimal Courant--Jacobi automorphism of $\bbD L$ arises from a section of $\bbD L\to M$ according to the following surjective map
	\begin{equation*}
		\Gamma(\bbD L)\longrightarrow\mathfrak{aut}_{\sf CJ}(\bbD L)\simeq \calD L\ltimes\Omega^2_{D,\text{cl}}(L),\ u\longmapsto(\ldsb u,-\rdsb,\pr_D u)\simeq(\pr_Du,-\rmd_D\alpha).
	\end{equation*}

\end{remark}

\subsection{Dirac-Jacobi Bundles}
\label{Sec:DirJac}
Let us now discuss the subbundles of the omni-Lie algebroid we are interested in: the so-called \emph{Dirac-Jacobi bundles}. They are, roughly speaking, Dirac-like subbundles of the omni-Lie algebroid and they were also introduced in \cite{CHEN2010799}. The first time Dirac-like structures appeared in order to model the Dirac analogue in Jacobi geometry 
was in \cite{Wade2000}, which were called $\mathcal{E}^1(M)$-Dirac structures. Note that these bundles are special cases 
of the Dirac-Jacobi bundles we will define, if the line bundle of the omni-Lie algebroid is trivial. 


\begin{definition}
	\label{def:DJ_structure}
	Let $L\to M$ be a line bundle.
	A \emph{Dirac--Jacobi structure} on $L$ is a subbundle $\calL\subset \bbD L$ satisfying the following two conditions
	\begin{enumerate}[label={\arabic*)}]
		\item $\mathcal L$ is involutive with respect to $\ldsb- ,-\rdsb$,
		\item $\mathcal L$ is maximally isotropic with respect to $\bla - ,-\bra$.
	\end{enumerate}
	Then, a \emph{Dirac--Jacobi bundle} $(L,\calL)$ over $M$ is a line bundle $L\to M$ equipped with a Dirac-Jacobi structure $\calL$, and a \emph{Dirac--Jacobi manifold} $(M,L,\calL)$ is a manifold $M$ with a Dirac--Jacobi structure over it. 
\end{definition}


\begin{remark}
	Since the pairing $\bla - ,-\bra$ has split signature, the maximal isotropy condition on $\calL\subset\bbD L$ in Definition~\ref{def:DJ_structure} makes sense and is equivalent to $\calL$ being a \emph{Lagrangian subbundle} of $(\bbD L,\ldab-,-\rdab)$, i.e.
	\begin{equation*}
		\calL=\calL^\perp,
	\end{equation*}
	where the superscript ${}^\perp$ denotes the orthogonal wrt $\ldab-,-\rdab$.
	Moreover, the involutivity of a Lagrangian subbundle $\mathcal{L}\subset\bbD L$ is equivalent to the vanishing of the \emph{Courant tensor} $\Upsilon_\calL\in\Gamma(\wedge^3\calL^\ast\otimes L)$ defined by
	\begin{equation}
		\label{Eq: Torsion}
		\Upsilon_{\calL}(X,Y,Z):=\ldab X,\ldsb Y,Z\rdsb\rdab.
	\end{equation}
\end{remark}

\begin{example}
	\label{ex:PrecontactStructures_DJ}
	A first class of examples is represented by precontact structures.
	Indeed, for any line bundle $L\to M$, mapping each $L$-valued Atiyah $2$-form $\varpi$ to its graph establishes a one-to-one correspondence
	\begin{equation*}
		\Omega^2_D(L)\overset{\sim}{\longrightarrow}\{\calL\subset\bbD L\ \textnormal{Lagrangian}\mid\calL\pitchfork J^1L\},\ \varpi\longmapsto\operatorname{Gr}(\varpi):=\{(\delta,\varpi^\flat\delta)\mid\delta\in DL\}
	\end{equation*}
	such that $\operatorname{Gr}\varpi\subset\bbD L$ is a Dirac--Jacobi structure iff $\varpi$ is a presymplectic Atiyah form, i.e.~$\rmd_D\varpi=0$.
	Hence, $L$-valued precontact structures identify with those Dirac--Jacobi structures $\calL\subset\bbD L$ such that $\calL\pitchfork J^1L$.
\end{example}

\begin{example}
	A second class of examples is represented by Jacobi structures.
	Indeed, for any line bundle $L\to M$, mapping each bi-derivation $J\in\calD^2L$ to its graph establishes a one-to-one correspondence
	\begin{equation*}
		\calD^2(L)\overset{\sim}{\longrightarrow}\{\calL\subset\bbD L\ \textnormal{Lagrangian}\mid\calL\pitchfork DL\},\ J\longmapsto\operatorname{Gr}(J):=\{(J^\sharp\alpha,\alpha)\mid\alpha\in J^1 L\}
	\end{equation*}
	such that $\operatorname{Gr}J\subset\bbD L$ is a Dirac--Jacobi structure iff $J$ is a Jacobi biderivation of $L$, i.e.~$[J,J]=0$.
	Hence, $L$-valued precontact structures identify with those Dirac--Jacobi structures $\calL\subset\bbD L$ such that $\calL\pitchfork DL$.
\end{example}

\begin{example}
	A third class of examples is represented by locally conformal presymplectic structures~\cite[Section 3]{DirJacBun}.
	Let us first recall that, on a line bundle $L\to M$, a \emph{locally conformal presymplectic (or plcs) structure} $(\nabla,\omega)$ consists of a flat connection $\nabla$ on $L$ and a $\rmd_\nabla$-closed $2$-form $\omega\omega^2(M;L)$.
	Then, each lcps structure $(\nabla,\omega)$ on $L$ determines a Dirac--Jacobi structure $\calL_{(\nabla,\omega)}\subset\bbD L$ given by
	\begin{equation*}
		\calL_{(\nabla,\omega)}=\{(\nabla_\xi,\sigma^\ast(\omega^\flat \xi)+\alpha)\mid (\xi,\alpha)\in TM\oplus(\image\nabla)^\circ\}.
	\end{equation*}
	Moreover, the lcps structure $(\nabla,\omega)$ is faithfully encoded into the associated Dirac--Jacobi structure $\calL_{(\nabla,\omega)}$.
\end{example}

Let us focus now on Dirac-Jacobi morphisms: as in the Dirac case, there are two different kinds.
\begin{definition}
	\label{def:backward_forward_DJ}
	Let $\Phi:L\to L^\prime$ be a regular line bundle morphism over a smooth map $\varphi:M\to M^\prime$.
	The \emph{forward transform along $\Phi$} of a Dirac--Jacobi structure $\calL\subset\bbD L$  is the 
	(non-necessarily smooth) Lagrangian family $\Phi_!\calL\subset\varphi^\ast\bbD L^\prime$ defined by
	\begin{equation*}
		(\Phi_!\mathcal{L})_x:=\{(D\Phi(\delta),\alpha^\prime)\in (\mathbb{D}L^\prime)_{\varphi(x)}\mid (\delta,(D_x\Phi)^\ast\alpha^\prime)\in \calL_x\}.
	\end{equation*}
	The \emph{backward transform along $\Phi$} of a Dirac--Jacobi structure $\calL^\prime\subset\bbD L^\prime$ is the 
	(non-necessarily smooth) Lagrangian family $\Phi^!\calL^\prime\subset\bbD L$ defined by
	\begin{equation*}
		(\Phi^!\mathcal{L}^\prime)_x:=\{(\delta,(D_x\Phi)^\ast\alpha^\prime)\in (\mathbb{D}L)_x\mid ((D\Phi)\delta,\alpha^\prime)\in \calL^\prime_{\varphi(x)}\}.
	\end{equation*}
\end{definition}

We refer to~\cite[Remarks~8.3 and 8.6]{DirJacBun} for a proof that $\Phi^!\calL^\prime$ and $\Phi_!\calL$ are actually Lagrangian families.

\begin{definition}
	Let $(M_i,L_i,\calL_i)$ be a Dirac--Jacobi manifold, for $i=1,2$, and let  $\Phi\colon L_1\to L_2$ be a regular line bundle morphism over a smooth map $\varphi:M_1\to M_2$.
	\begin{enumerate}[label={\arabic*)}]
		\item $\Phi:(M_1,L_1,\calL_1)\to (M_2,L_2,\calL_2)$ is a \emph{forward Dirac-Jacobi map} if $\Phi_!\calL_1=\calL_2|_{\varphi(M_1)}$.
		\item $\Phi:(M_1,L_1,\calL_1)\to (M_2,L_2,\calL_2)$ is a \emph{backward Dirac-Jacobi map} if $\calL_1=\Phi^!\calL_2$.
	\end{enumerate}
\end{definition}
It is worth mentioning that with both morphisms Dirac-Jacobi manifolds become a category.
We are going to discuss them in some detail.
For a regular LB morphism $\Phi\colon L_1\to L_2$, over a smooth map $\varphi:M_1\to M_2$, and a Dirac--Jacobi structure $\mathcal{L}_2\subseteq \mathbb{D}L_2$, if $\mathcal{L}_1:=\Phi^!\calL_2$ is a Dirac--Jacobi structure, then $\Phi$ is clearly a backward Dirac--Jacobi map $(M_1,L_1,\calL_1)\to (M_2,L_2,\calL_2)$.
But $\Phi^!\calL_2$ is not necessarily a Dirac-Jacobi structure: 

%

Let us examine under which circumstances the forward and backward transforms of a Dirac--Jacobi bundle are still subbundles. 
For this end, very useful tools are the following \emph{clean intersection} criteria.

\begin{theorem}{\cite[Prop.~8.11]{DirJacBun}}
	\label{thm:CleanInt_forward_DJ}
	Let $\Phi\colon L\to L^\prime$ be a regular LB morphism over $\varphi:M\to M^\prime$ and let $\calL\in \bbD L$ be a Dirac--Jacobi structure.
	If $\calL\cap\ker D\Phi$ has constant rank, then $\Phi_!\calL$ is a Dirac--Jacobi structure. 
\end{theorem}

\begin{theorem}{\cite[Prop.~8.4]{DirJacBun}}
	\label{Thm: CleanInt}
	Let $\Phi\colon L\to L^\prime$ be a regular LB morphism over $\varphi:M\to M^\prime$ and let $\mathcal{L}^\prime\in \mathbb{D}L^\prime $ be a Dirac--Jacobi structure.
	If $\ker (D\Phi)^\ast\cap \varphi^\ast \calL^\prime$ has constant rank, then $\Phi^!\calL^\prime$ is a Dirac--Jacobi structure. 
\end{theorem}
From this theorem one can deduce a very useful result (see Corollary~\ref{Cor: TrnsMap}) about the smoothness of the backward transform of a Dirac--Jacobi structure along a Dirac--Jacobi transversal.

\begin{definition}
	\label{def:DJ_transversal}
	Let $(M,L,\calL)$ be a Dirac--Jacobi manifold.
	A \emph{Dirac--Jacobi transversal} to $(M,L,\calL)$ is a regular LB morphism $\Phi:L_N\to L$, covering a smooth map $\varphi:N\to M$, such that, for all $x\in N$,
	\begin{equation}
		\label{eq:def:DJ_transversal}
		(D_x\Phi) (D_xL_N)+ \pr_D \calL_{\varphi(x)}=D_{\varphi(x)}L.
	\end{equation}
	If additionally $N\subset M$ is a submanifold, $L_N=L|_N$ and $\Phi$ is the inclusion, one says that $N$ is a \emph{Dirac--Jacobi transversal} to $(M,L,\calL)$.
\end{definition}

\begin{corollary}[{\cite[Prop.~4.2]{NorForJac}}]
	\label{Cor: TrnsMap}
	Let $(M,L,\mathcal{L})$ be a Dirac-Jacobi manifold and let $\Phi\colon L_N \to L$ be a regular LB morphism.
	If $\Phi$ is a Dirac--Jacobi transversal for $(M,L,\calL)$ then $\Phi^!\calL\subseteq\bbD L_N$ is a Dirac--Jacobi structure. 
\end{corollary}

\begin{proof}
	Let $(x,(0,\alpha))\in \ker D\Phi^\ast\cap \varphi^\ast\mathcal{L}$, thus $(0,\alpha)\in \mathcal{L}_{\varphi(x)}$ and $\alpha\in(\image(D_x\Phi))^\circ$.
	Since $\calL^\perp=\calL$ and so $\calL\cap J^1L=(\pr_D\mathcal{L})^\circ$, we see that $\alpha\in(\image(D_x\Phi)+ \pr_D\calL_{\varphi(x)})$.
	By the Dirac--Jacobi transversality condition~\eqref{eq:def:DJ_transversal} we conclude that $\alpha=0$.
	Hence $\ker D\Phi^\ast\cap\varphi^\ast\calL=0$ and the claim follows from Theorem~\ref{Thm: CleanInt}.
\end{proof}

Corollary \ref{Cor: TrnsMap} implies, in particular, that the backward transform along a regular LB morphism covering a submersion is always a Dirac-Jacobi structure.
Note that we have in special cases that backward and forward transforms are inverse to each other, which will be useful throughout this paper.

\begin{corollary}
	\label{cor:forward_backward}
	Let $\Phi:L\to L^\prime$ be a regular LB morphism covering a surjective submersion $\varphi:M\to M^\prime$.
	Then, for any Dirac--Jacobi structure $\calL^\prime\subset\bbD L^\prime$, one gets
	\begin{equation*}
		\Phi_!\Phi^!\calL^\prime=\calL^\prime.
	\end{equation*}
\end{corollary}  
\begin{remark}
	Since, by their very definition, Courant--Jacobi automorphisms of the omni-Lie algebroid preserve both the symmetric product and the Dorfman bracket, if $(F,\Phi)\in\Aut_{\sf{CJ}}(\bbD L)$ and $\calL\subset\bbD L$ is a Dirac--Jacobi structure, then $F\calL\subset\bbD L$ is again a Dirac--Jacobi structure.
	In particular, given an arbitrary LB automorphism $\Phi\in\Aut_{\sf{VB}}(L)$, for any Dirac--Jacobi structure $\calL\subset\bbD L$ one gets that
	\begin{align*}
		(\mathbb{D}\Phi)\mathcal{L}=(\Phi^{-1})^!\calL=\Phi_!\calL
	\end{align*}
	which is a Dirac-Jacobi structure again.
	Moreover, given an arbitrary $B\in \Omega^2_{D,\text{cl}}(L)$, for any Dirac--Jacobi structure $\calL\subset\bbD L$, there is a unique Dirac--Jacobi structure $\calL^B\subset\bbD L$, the \emph{B-field tranform of $\calL$}, such that 
	\begin{equation*}
		\mathcal{L}^B=\exp(B) \mathcal{L}.
	\end{equation*}
\end{remark}

\begin{remark}
	\label{rem:opposite_DJ}
	An additional operation on Dirac--Jacobi structures is the following.
	For any Dirac--Jacobi structure $\calL\subset\bbD L$, the \emph{opposite Dirac-Jacobi structure} $\calL^{\mathrm{opp}}\subset\bbD L$ is defined by 
	\begin{align*}
		\mathcal{L}^{\mathrm{opp}}:=\{(\delta,-\alpha)\in \bbD L\mid(\delta,\alpha)\in \mathcal{L}\}.
	\end{align*}
\end{remark}

\section{Products of Dirac-Jacobi Bundles}
\label{sec:products}

This Section introduces the notion of \emph{products of Dirac--Jacobi bundles} (see Theorem~\ref{theor:product_DJ_structure}).
Indeed, this notion will play a crucial role in studying weak dual pairs and defininig the most relevant operations on them.
As a necessary technical tool, we also recall the notion of products and fiber products of line bundles.

\subsection{Products of Line Bundles}
\label{subsec:LineBundlesProducts}

The category of line bundles should not be seen as a full subcategory of vector bundles, at least not for our purposes. 
The reason is that we want to shrink the  $\Hom$ sets in order to get a category admitting products and a reasonable amount of fiber products.
Precisely, we will consider the category $\mathfrak{Line}$ consisting of line bundles as objects and ``regular'' line bundle morphisms as arrows.

\begin{theorem}
	\label{Thm: ProdCatLine}
	The category $\mathfrak{Line}$ admits products.
\end{theorem}

\begin{proof}
	Fix arbitrary line bundles $L_i\to M_i$, for $i=1,2$.
	First, we define the smooth manifold
	\begin{align*}
		M_1\times^! M_2 =\{\Psi_{x,y}\colon L_{1,x}\to L_{2,y}\ \text{linear isomorphism}\ |\ (x,y)\in M_1\times M_2\},
	\end{align*}
	with the natural projections $p_i\colon M_1\times^! M_2\to M_i$, for $i=1,2$.
	Second, we construct the line bundle $L_1\times^!L_2\to M_1\times^!M_2$ by setting 
	\begin{align*}
		L_1\times^!L_2:=p_1^\ast L_1
	\end{align*}		 
	together with the regular line bundle morphisms $P_1\colon L_1\times^!L_2 \to L_1$ and $P_2\colon L_1\times^!L_2 \to L_2$ defined by 
	\begin{equation*}
		P_1\colon L_1\times^!L_2 \ni (\Psi_{x,y},\lambda_x)\mapsto \lambda_x\in L_1\qquad\text{and}\qquad P_2\colon L_1\times^!L_2 \ni (\Psi_{x,y},\lambda_x)\mapsto \Psi_{x,y}(\lambda_x)\in L_2.
	\end{equation*}
	Finally, let us consider an arbitrary line bundle $L\to M$ with regular LB morphisms $\Phi_i\colon L\to L_i$, for $i=1,2$.
	Then, it is easy to see that there is a unique regular LB morphism $\Phi_1\times^!\Phi_2\colon L\longrightarrow L_1\times^! L_2$ such that the following diagram
	\begin{center}
		\begin{tikzcd}
			&& L\arrow[lld, bend right=20, swap, "\Phi_1"]\arrow[rrd, bend left=20, "\Phi_2"]\arrow[d, dashed, "\Phi_1\times^!\Phi_2"] &&\\
			L_1 && L_1\times^!L_2\arrow[ll, "P_1"]\arrow[rr, swap, "P_2"] && L_2
		\end{tikzcd}
	\end{center}
	commute, and it is explicitly given by the following map
	\begin{align*}
		\Phi_1\times^!\Phi_2\colon L\longrightarrow L_1\times^!L_2,\ \ l_x\longmapsto (\Phi_{2,x}\circ\Phi_{1,x}^{-1},\Phi_{1,x}(l_x)). 
	\end{align*}  
	Hence, the line bundle $L_1\times^!L_2\to M_1\times^!M_2$, equipped with the regular LB morphisms $P_i\colon L_1\times^!L_2 \to L_i$, for $i=1,2$, has the universal property of the \emph{product of $L_1$ and $L_2$} in the category $\mathfrak{Line}$.
\end{proof}

\begin{remark}
	\label{rem:product_LB_morphisms}
	In the following, given regular LB morphisms $\Phi_1:L^\prime_i\to L_i$, with $i=1,2$, we will denote by $\Phi_1\times^!\Phi_2:L_1^\prime\times^!L_2^\prime\to L_1\times^!L_2$ the unique regular LB morphism s.t.~the following commute
	\begin{center}
		\begin{tikzcd}
			L_1^\prime\arrow[d, swap, "\Phi_1"] && L_1^\prime\times^!L_2^\prime\arrow[ll, swap, "P_1^\prime"]\arrow[rr, "P_2^\prime"]\arrow[d, dashed, "\Phi_1\times^!\Phi_2"] &&L_2^\prime\arrow[d, "\Phi_2"]\\
			L_1 && L_1\times^!L_2\arrow[ll, "P_1"]\arrow[rr, swap, "P_2"] && L_2
		\end{tikzcd}
	\end{center}
	where $P_i:L_1\times^!L_2\to L_i$ and $P_i^\prime:L_1^\prime\times^!L_2^\prime\to L_i^\prime$, for $i=1,2$, denote the natural projections of the products.
\end{remark}
\color{black}
Next, we show that $\mathfrak{Line}$ admits a reasonable amount of fiber products.
In the following we will denote by $\mathfrak{Man}$ the category of smooth manifolds with smooth maps as morphisms.

\begin{theorem}
	\label{Thm:FiberedProduct_LB}
	Let $\Phi_i:L_i\to L$ be regular LB morphisms, covering smooth maps $\varphi_i:M_i\to M$, for $i=1,2$.
	If, in the category of manifolds $\mathfrak{Man}$, there exists the fiber product
	\begin{center}
		\begin{tikzcd}
			M_1\times_M M_2 \arrow[r, "p_2"] \arrow[d, swap, "p_1"]& M_2\arrow[d, "\varphi_2"]\\
			M_1\arrow[r, swap, "\varphi_1"] & M
		\end{tikzcd}
	\end{center} 
	then, in the category of line bundles $\mathfrak{Line}$, there exists the fiber product of $\Phi_1\colon L_1\to L$ and $\Phi_2\colon L_2\to L$.
\end{theorem}

\begin{proof}
	Since $M_1\times_M M_2$ is a smooth manifold by assumption, we can consider the smooth map 
	\begin{align*}
		p\colon M_1\times_M M_2\longrightarrow M_1\times^!M_2,\ \ (x,y)\longmapsto \Phi_{2,y}^{-1}\circ \Phi^{-1}_{1,x}.
	\end{align*}
	Then, we can construct the line bundle $L_1\times_LL_2\to M_1\times_MM_2$ by setting 
	\begin{align*}
	L_1\times_L L_2:=p^\ast(L_1\times^!L_2)\simeq\{(\lambda_{1,x},\lambda_{2,y})\in L_1\times L_2\mid (x,y)\in M_1\times_MM_1\ \text{and}\ \Phi_1\lambda_{1,x}=\Phi_2\lambda_{2,y}\}
\end{align*}
	together with the regular LB morphisms $P_i\colon L_1\times_LL_2 \to L_i$, covering $p_i:M\times_MM_2\to M_i$, defined by 
	\begin{equation*}
		P_1\colon L_1\times_LL_2 \ni (\lambda_x,\lambda_y)\mapsto \lambda_x\in L_1\qquad\text{and}\qquad P_2\colon L_1\times_LL_2 \ni (\lambda_x,\lambda_y)\mapsto \lambda_y\in L_2.
	\end{equation*}
	Now it is easy to see that the latter fits in the following commutative diagram of regular LB morphisms
	\begin{equation}
	\label{Lem:Pullback_Proof}
		\begin{tikzcd}
			L_1\times_L L_2\arrow[r, "P_2"] \arrow[d, swap, "P_1"]& L_2\arrow[d, "\Phi_2"]\\
			L_1\arrow[r, swap, "\Phi_1"] & L
		\end{tikzcd}
	\end{equation}
	Finally, let us consider regular LB morphisms $\Psi_i:L^\prime\to L_i$, for $i=1,2$, such that $\Phi_1\circ\Psi_1=\Phi_2\circ\Psi_2$.
	Then there is a unique regular LB morphism $\Psi\colon L^\prime\rightarrow L_1\times_L L_2$ such that the following diagram commutes
	\begin{center}
		\begin{tikzcd}
			L^\prime\arrow[rrd, bend left=25, "\Psi_2"]\arrow[ddr, bend right=40, swap, "\Psi_1"]\arrow[rd, dashed, "\Psi"]&&\\
			&L_1\times_L L_2\arrow[r, "P_2"] \arrow[d, swap, "P_1"]& L_2\arrow[d, "\Phi_2"]\\
			&L_1\arrow[r, swap, "\Phi_1"] & L
		\end{tikzcd}
	\end{center}
 	As it is easy to see, the latter is explicitly given by the map $\Psi\colon L^\prime\longrightarrow L_1\times_LL_2,\ l_x\longmapsto(\Psi_1l_x,\Psi_2l_x)$.
 	So, the line bundle $L_1\times_LL_2\to M_1\times_MM_2$, with the regular LB morphisms $P_i\colon L_1\times_LL_2 \to L_i$, for $i=1,2$, has the universal property of the \emph{fiber product of $\Phi_1:L_1\to L$ and $\Phi_2:L_2\to L$} in the category $\mathfrak{Line}$.
\end{proof}

For later use we consider the Atiyah algebroid and the first jet bundle of a (fiber) product line bundle. 

\begin{lemma}
	\label{Lem: SplitDProd}
	Let $L_i\to M_i$ be line bundles for $i=1,2$.
	Then, for the product line bundle $L_1\times^!L_2\to M_1\times^!M_2$, with projections $P_i\colon L_1\times^!L_2 \to L_i$, for $i=1,2$, its Atiyah algebroid decomposes as follows 
	\begin{align*}
		D(L_1\times^!L_2) =\ker DP_1\oplus \ker DP_2.
	\end{align*}
\end{lemma}

\begin{proof}
	Note that the map $p\colon M_1\times^!M_2 \longrightarrow M_1\times M_2,\ \Phi_{x,y}\longmapsto(x,y)$ is a surjective submersion, with $\dim(M_1\times^!M_2)=\dim M_1+\dim M_2$, and hence $\rank\ker Tp=1$. 
	Moreover, $\ker Tp=\ker Tp_1\cap \ker Tp_2$ and it is generated by the fundamental vector field of the action of the 
	\begin{align*}
		\varphi\colon \mathbb{R}^\times \times (M_1\times^!M_2)\longrightarrow M_1\times^!M_2,\ (\alpha,\Psi_{x,y})\mapsto \alpha^{-1}\Psi_{x,y},
	\end{align*}
	where $\mathbb{R}^\times$ is the multiplicative group of non-zero real numbers.
	Further, the latter lifts to an action of $\mathbb{R}^\times$ on $L_1\times^!L_2\to M_1\times^!L_2$ by line bundle automorphisms
	\begin{align*}
		\Phi\colon \mathbb{R}^\times \times (L_1\times^!L_2)\longrightarrow L_1\times^!L_2,\ \ (\alpha,(\Psi_{x,y},\lambda_x))\longmapsto \Phi_\alpha(\Psi_{x,y},\lambda_x):=(\alpha^{-1}\Psi_{x,y},\lambda_x).
	\end{align*}	 
	By the construction of $L_1\times^!L_2$ (see the proof of Theorem~\ref{Thm: ProdCatLine}) we have that 
	\begin{align}\label{Eq: ProjProd}
		P_1\circ \Phi_\alpha=P_1 \text{ and } P_1\circ \Phi_\alpha =\alpha^{-1}P_1\text{ for all } \alpha\in \mathbb{R}^\times.
	\end{align}
	Let us denote by $\Delta$ the infinitesimal generator of the $\mathbb{R}^\times$-action on $L_1\times^!L_2\to M_1\times^!M_2$ by LB automorphisms, i.e.~the derivation $\Delta\in \calD(L_1\times^!L_2)$ given by 
	\begin{align*}
		\Delta(\lambda)=\frac{\D}{\D t}\At{t=0} \Phi_{\exp(t)}^\ast \lambda.
	\end{align*}
	So, in particular, $\ker Tp=\langle\sigma(\Delta)\rangle$.
	Moreover, using Equations \eqref{Eq: ProjProd}, we have 
	\begin{align*}
		DP_1(\Delta_{\Psi_{x,y}})=0 \text{ and } DP_2(\Delta_{\Psi_{x,y}})=\mathbbm{1}_y.
	\end{align*}
	Fix now $\square\in \ker DP_1\cap \ker DP_2$.
	Then $\sigma(\square)\in \ker Tp_1\cap \ker Tp_2$ and so $\square=a\mathbbm{1}+b\Delta$ with $a,b\in \mathbb{R}$.
	Consequently, one can compute
	\begin{equation*}
		0=DP_1(\square)=a\mathbbm{1}\qquad\text{and}\qquad 0=DP_2(\square)=(a+b)\mathbbm{1},
	\end{equation*}
	and so $a=b=0$.
	This means that $\ker DP_1\cap \ker DP_2=\{0\}$ and counting dimensions the claim follows. 
\end{proof}

Here are convenient expressions the Atiyah algebroid and the first jet bundle of a product line bundle. 

\begin{proposition}
	\label{prop:D-FunProd}
	Let $L_i\to M_i$ be line bundles for $i=1,2$.
	Then there exist canonical VB isomorphisms
	\begin{equation}
		\label{eq:prop:D-FunProd}
		D(L_1\times^!L_2) \cong p_1^\ast(DL_1)\oplus p_2^\ast(DL_2)\qquad\text{and}\qquad J^1(L_1\times^!L_2) \cong p_1^\ast(J^1L_1)\oplus p_2^\ast(J^1L_2)
	\end{equation}
	with the natural projections $p_1\colon M_1\times^! M_2\rightarrow M_1,\ \Psi_{x,y}\mapsto x$, and  $p_2\colon M_1\times^! M_2\rightarrow M_2,\ \Psi_{x,y}\mapsto y$.
\end{proposition}

\begin{proof}
	Let us simply write down the map
	\begin{align*}
		D(L_1\times^!L_2) \longrightarrow p_1^\ast(DL_1)\oplus p_2^\ast(DL_2),\quad \Delta_{\Psi_{x,y}} \longmapsto (\Psi_{x,y},(DP_1(\Delta_{\varphi_{x,y}}),DP_2(\Delta_{\varphi_{x,y}})))
	\end{align*}
	which is injective by Lemma \ref{Lem: SplitDProd}.
	Comparing the ranks of $p_1^\ast(DL_1)\oplus p_2^\ast(DL_2)$ and $D(L_1\times^!L_2)$, the first identification in Equation~\eqref{eq:prop:D-FunProd} follows.
	Dualizing one also obtains the second identification.
\end{proof}

As a final Corollary, we consider the Atiyah algebroid of the fiber product line bundle.

\begin{corollary}
	\label{Cor: SplittingPPDFun}
	Let $\Phi_i\colon L_i\to L$ be regular LB morphisms, for $i=1,2$.
	Under the same assumptions of Theorem~\ref{Thm:FiberedProduct_LB} there exist canonical VB isomorphisms
	\begin{equation}
		\label{eq:Cor: SplittingPPDFun}
		D(L_1\times_LL_2) \cong  DL_1\times_{DL}  DL_2\quad\text{and}\quad J^1(L_1\times_LL_2)\cong 
		J^1L_1\sqcup_{J^1L}J^1L_2.
	\end{equation}
\end{corollary}

\begin{proof}
The map 
	\begin{align*}
	D(L_1\times_LL_2)\ni \Delta \mapsto (DP_1(\Delta), DP_2(\Delta))\in DL_1\times_{DL}  DL_2.
	\end{align*}	 
is a linear ismorphism, which can be seen by counting dimension and by using the commutative diagram 
\eqref{Lem:Pullback_Proof}. The second statement follows by duality. 
\end{proof}

\subsection{Products of Dirac-Jacobi Bundles}
\label{subsec:ProdofDJ}

The section introduces the product of two Dirac--Jacobi structures and studies its first properties.
As well-known the product of two Dirac structures lives on the product manifold, similarly the product of two Dirac--Jacobi structures lives on the product line bundle constructed in Section~\ref{subsec:LineBundlesProducts}.
As a preliminary step, we start introducing the sum of two Dirac--Jacobi structures defined on the same line bundle.

\begin{definition}
	Let $L \to M$ be a line bundle and let $\mathcal L_1, \mathcal L_2 \subset \mathbb D L$ be Dirac-Jacobi structures on $L$.
	The \emph{sum of $\mathcal{L}_1$ and $\mathcal{L}_2$} is the Lagrangian family $\mathcal{L}_1\star\mathcal{L}_2\subset\mathbb{D}L$ defined by 
\begin{equation*}
	\mathcal{L}_1 \star \mathcal{L}_2 := \{ (\delta, \alpha_1 + \alpha_2) \mid (\delta, \alpha_i) \in \mathcal{L}_i,\ \text{for}\ i = 1,2\} \subset 
	\mathbb{D} L.
\end{equation*}
\end{definition}

Note that, as stated in definition, $\mathcal{L}_1\star\mathcal{L}_2\subset \mathbb D L$ is a Lagrangian family.
Indeed, by construction, its isotropy follows from the isotropy of $\mathcal{L}_1$ and $\mathcal{L}_2$.
Moreover, $\calL_1\star\calL_2$ fits in the (pointwise) exact sequence
\begin{equation}
	\label{Lem:Prod_Proof}
	\begin{tikzcd}
		0\arrow[r]&(\pr_\rmD \calL_1\cap \pr_\rmD\calL_2)^\circ\arrow[r, "\text{incl}"]&\calL_1\star\calL_2\arrow[r, "\pr_\rmD"]&\pr_\rmD\calL_1\cap \pr_\rmD\calL_2\arrow[r]&0,
	\end{tikzcd}
\end{equation}
where we use that $\pr_D(\calL_1\star\calL_2)=\pr_\rmD\calL_1\cap\pr_\rmD\calL_2$ and $J^1L\cap(\calL_1\star\calL_2)=(\pr_\rmD\calL_1\cap\pr_\rmD\calL_2)^\circ$.
From the latter one concludes that $\rank(\calL_1\star\calL_2)=\rank DL=\frac{1}{2}\rank\bbD L$ and hence $\calL_1\star\calL_2$ is maximal isotropic.

\begin{lemma}
	\label{Lem: Product}
	Let $\calL_1,\calL_2\subset\bbD L$ be Dirac--Jacobi structures.
	If $\calL_1\star\calL_2\subset\bbD L$ is smooth, then $\calL_1\star\calL_2$ is a Dirac-Jacobi structure on $L$.
\end{lemma}

\begin{proof}
	We only have to prove that the smooth Lagrangian subbundle $\calL_1\star\calL_2\subset\bbD L$ is involutive.
	
	Fix a point $p\in M$, and choose an arbitrary element $(\delta,\alpha_1+\alpha_2)$ of $(\calL_1\star\calL_2)_p$.
	Since $\mathcal{L}_1\star\mathcal{L}_2$ is smooth and $\delta_p\in\pr_\rmD\calL_1\cap\pr_\rmD\calL_2=\pr_D(\calL_1\star\calL_2)$, there exists in particular a derivation $\Delta\in\calD L$ taking values in $\pr_D(\calL_1\star\calL_2)$ such that, additionally, $\Delta_p=\delta$.
	Therefore, for $i=1,2$, $\Delta$ takes values in $\pr_\rmD\calL_i$ and so there exists a local section $\eta_i$ of $J^1L_i$, defined in a neighborhood of $p$ in $M$, such that $(\Delta,\eta_i)$ takes values in $\calL_i$ and $(\Delta_p,\eta_{i,p})=(\delta,\alpha_i)$. 
	So $(\Delta,\eta_1+ \eta_2)$ is a local section of $\calL_1\star\calL_2$, defined in a neighborhood of $p$ in $M$, whose value at $p$ coincides with $(\delta,\alpha_1+\alpha_2)$.
	 
	Now it is easy to check that the Courant tensor of $\calL_1\star\calL_2$ vanishes on this kind of sections, i.e.
	\begin{equation*}
		\Upsilon_{\calL_1\star\calL_2}((\Delta,\eta_1+\eta_2),(\Delta^\prime,\eta^\prime_1+\eta^\prime_2),(\Delta^{\prime\prime},\eta^{\prime\prime}_1+\eta^{\prime\prime}_2))=0
	\end{equation*}
	for all $(\Delta,\eta_i),(\Delta^\prime,\eta^\prime_i),(\Delta^{\prime\prime},\eta^{\prime\prime}_i)\in\Gamma(\calL_i)$, with $i=1,2$.
	So $\Upsilon_{\calL_1\star\calL_2}$ vanishes identically.
\end{proof}

\begin{proposition}
	\label{prop:smoothness_sum}
	Let $\calL_1,\calL_2\subset\bbD L$ be Dirac--Jacobi structures.
	If $\pr_\rmD\calL_1+\pr_\rmD\calL_2\subset DL$ has constant rank, 
	then $\calL_1\star\calL_2$ is a Dirac--Jacobi structure on $L$.
\end{proposition}

\begin{proof}
	Since $\pr_\rmD\calL_1+\pr_\rmD\calL_2\subset DL$ has constant rank, one can introduce the surjective VB morphism
	\begin{align*}
		K\colon \calL_1\oplus\calL_2\longrightarrow\pr_\rmD\calL_1+ \pr_\rmD\calL_2,\quad((\delta_1,\alpha_1),(\delta_2,\alpha_2))\longmapsto\delta_1-\delta_2.
	\end{align*}
	Consequently, $\ker K\subset\calL_1\oplus\calL_2$ is a smooth subbundle and we can consider the VB morphism
	\begin{align*}
		S\colon \ker K\longrightarrow\bbD L,\quad ((\delta,\alpha),(\delta,\beta))\longmapsto (\delta,\alpha+\beta). 
	\end{align*}
	Since the kernel of $S$ is isomorphic to $(\pr_D\calL_1+\pr_D\calL_2)^\circ$ and so it has contant rank, also its image $\image S=\calL_1\star \mathcal L_2$ is a smooth subbundle and so, by Lemma \ref{Lem: Product}, it is a Dirac--Jacobi structure.  
\end{proof} 

\begin{remark}
	\label{Rem: SoothnessProd}
	Actually, the last Proposition is just a special case of the following fact which can be checked by elementary techniques.
	We are referring to the fact that the (contant rank) subbundle $\mathcal L_1 \star \mathcal L_2\subset\bbD L$ is smooth if and only if the singular subbundle $\pr_D\mathcal{L}_1\cap \pr_D\mathcal{L}_2\subset DL$ is smooth in the sense that, for any $p\in M$, its fiber over $p$ is generated by the values at $p$ of its smooth sections.
	If this is the case then, clearly, the compactly supported smooth sections of $\pr_D\calL_1\cap\pr\calL_2=\pr_D(\calL_1\star\calL_2)$ form a singular Lie subalgebroid of $DL$ in the sense of~\cite{zambon2018singular}.
\end{remark}

After these necessary preliminaries, the next theorem defines the product of Dirac--Jacobi structures. 

\begin{theorem}
	\label{theor:product_DJ_structure}
	Let $(M_i,L_i,\mathcal{L}_i)$ be a Dirac--Jacobi manifold, with $i=1,2$.
	Consider the product line bundle $L_1\times^!L_2\to M_1\times^!M_2$ with projections $P_i:L_1\times^!L_2\to L_i$, for $i=1,2$.
	Then the following is a smooth (constant rank) subbundle
	\begin{equation*}
		\calL_1\times^! \calL_2:=(P_1^!\calL_1)\star(P_2^!\calL_2)\subseteq\bbD(L_1\times^!L_2)
	\end{equation*}
	and so it is a Dirac-Jacobi structure on $L_1\times^!L_2\to M_1\times^!M_2$, called the \emph{product of $\calL_1$ and $\calL_2$}.
	Moreover, for $i=1,2$, the projection $P_i:L_1\times^!L_2\to L_i$ gives rise to a forward Dirac--Jacobi map
	\begin{equation*}
		\begin{tikzcd}
			(M_1\times^!M_2,L_1\times^!L_2,\calL_1\times^!\calL_2)\arrow[rr, "P_i"]&&(M_i,L_i,\calL_i).
		\end{tikzcd}
	\end{equation*}  
\end{theorem}

\begin{proof}
	By the definition of the backward transforms, we have $\ker DP_i\subset P_i^!\calL_i$, for $i=1,2$ and hence also
	\begin{equation*}
		\pr_D(P_1^!\calL_1)+\pr_D(P_2^!\calL_2)\supset\ker DP_1+ \ker DP_2=D(L_1\times^!L_2),
	\end{equation*}
	where in the last step we used Lemma~\ref{Lem: SplitDProd}.
	Applying now Proposition~\ref{prop:smoothness_sum}, we see that  
	$(P_1^!\calL_1)\star (P_2^!\calL_2)$ is a Dirac--Jacobi structure on $L_1\times^!L_2\to M_1\times^!M_2$. 
	Now we want to prove that, for $i=1,2$,
	\begin{align*}
		(P_i)_!((P_1^!\calL_1)\star (P_2^!\calL_2))=\calL_i.
	\end{align*}		
	Fix arbitrary $x\in M_1\times^!M_2$ and $(\delta_1, \alpha_1)\in \calL_{1,p_1x}$.
	Using again Lemma~\ref{Lem: SplitDProd}, we can find $\widetilde\delta\in\ker D_xP_2$, such that $(DP_1)\widetilde\delta=\delta_1$ and therefore $(\widetilde\delta,(Dp_1)^\ast\alpha_1)\in (P_1^!\calL_1)_x$. 
	Since $\widetilde\delta\in\ker Dp_2\subset P_2^!\calL_2$, we also have that $(\widetilde\delta,(DP_1)^\ast\alpha_1)\in (P_1^!\calL_1)\star (P_1^!\calL_2)$ and therefore also $(\delta_1, \alpha_1)=((DP_1)\widetilde\delta,\alpha_{p_1(x)})\in(P_i)_!((P_1^!\calL_1)\star(P_2^!\calL_2))$.
	So, we have just proved the following inclusion
	\begin{align*}
		\calL_i\subset(P_i)_!(P_1^!(\mathcal{L}_1)\star P_2^!(\mathcal{L}_2)).
	\end{align*}	  	
	Further, since both sides are Lagrangian, the equality holds, i.e.~$\calL_i=(P_i)_!((P_1^!\calL_1)\star (P_2^!\calL_2))$.
\end{proof}

\begin{remark}
	Notice that the product Dirac--Jacobi structure $\calL_1\times^!\calL_2$, with the forward Dirac--Jacobi maps $P_i\colon(M_1\times^!M_2,L_1\times^!L_2,\calL_1\times^!\calL_2)\longrightarrow(M_i,L_i,\calL_i)$, for $i=1,2$, does not satisfy the universal property of the product of $\calL_1$ and $\calL_2$ in the category of Dirac--Jacobi manifolds with forward Dirac--Jacobi maps.
\end{remark}

\begin{corollary}
	\label{cor:product_contact_structures}
	Let $L_i\to M_i$ be a line bundle and let $\varpi_i\in\Omega_D^2(L_i)$ be a presymplectic Atiyah form, $i=1,2$.
	Then there is a unique presymplectic Atiyah form $\varpi_1\times^!\varpi_2\in\Omega_D^2(L_1\times^! L_2)$, the \emph{product of $\varpi_1$ and $\varpi_2$}, s.t.
	\begin{equation*}
		\operatorname{Gr}(\varpi_1\times^!\varpi_2)=\operatorname{Gr}(\varpi_1)\times^!\operatorname{Gr}(\varpi_2).
	\end{equation*}
\end{corollary}

\begin{proof}
	Using the fact that $P_i^!(\operatorname{Gr}\varpi_i)=\operatorname{Gr}(P_i^\ast\varpi_i)$, for $i=1,2$, one can easily compute
	\begin{align*}
		(\operatorname{Gr}\varpi_1\times^!\operatorname{Gr}\varpi_2)\cap J^1(L_1\times^!L_2)=(\pr_D(\operatorname{Gr}(P_1^\ast\varpi_1)\star\operatorname{Gr}(P_2^\ast\varpi_2)))^\circ=(D(L_1\times^!L_2))^\circ=0,
	\end{align*}
	So, by Example~\ref{ex:PrecontactStructures_DJ}, $\operatorname{Gr}\varpi_1\times^!\operatorname{Gr}\varpi_2$ is the graph of a unique closed $(L_1\times^!L_2)$-valued Atiyah $2$-form.
\end{proof}

%
%

%
Now we want to introduce fiber products of Dirac-Jacobi manifolds, note that fiber products do not always exist in the category of line bundles.

\begin{corollary}
	\label{Cor: PBDJ}
	Let $\Phi_i:(M_i,L_i,\calL_i)\to(M,L,\calL)$ be a forward Dirac--Jacobi map, with $i=1,2$.
	Under the same assumptions of Theorem~\ref{Thm:FiberedProduct_LB}, so that there exists the fiber product line bundle 
	\begin{center}
		\begin{tikzcd}
			L_1\times_LL_2 \arrow[r, "P_2"]\arrow[d, swap, "P_1"] & L_2\arrow[d, "\Phi_2"] \\
			L_1\arrow[r, swap, "\Phi_1"] & L
		\end{tikzcd},
	\end{center}
	$P_1^!\mathcal{L}_1$, $P_2^!\mathcal{L}_2$ and 
	$(P_1^!\calL_1)\star (P_2^!\calL_2) $ are Dirac-Jacobi structures with the $P_i$'s being forward Dirac-Jacobi maps.
\end{corollary}


%
\begin{proof}
The proof follows from Theorem~\ref{theor:product_DJ_structure} and Corollary \ref{Cor: SplittingPPDFun}.
\end{proof}
Let us now discuss some properties of products and fiber products wrt backward/forward Dirac--Jacobi maps that will be relevant for the aims of this paper.

\begin{lemma}
	\label{Lem: BackProd}
	Let $\Phi_i:(M_i^\prime,L_i^\prime,\calL_i^\prime)\longrightarrow(M_i,L_i,\calL_i)$ be a backward Dirac--Jacobi map, for $i=1$.
	Then the following regular LB morphism (cf.~Remark~\ref{rem:product_LB_morphisms}) is a backward Dirac--Jacobi map
	\begin{equation*}
		\begin{tikzcd}
			(M_1^\prime\times^!M_2^\prime,L_1^\prime\times^!L_2^\prime,\calL_1^\prime\times^!\calL_2^\prime)\arrow[rr, "\Phi_1\times^!\Phi_2"]&&(M_1\times^!M_2,L_1\times^!L_2,\calL_1\times^!\calL_2).
		\end{tikzcd}
	\end{equation*}
\end{lemma}

\begin{proof}
	The statement is just a special case of the following claim that we are actually going to prove.
	\paragraph{\textbf{Claim}}
	Let $\Phi\colon L^\prime\to L$ be a regular LB morphism and $\calL_i\subset\bbD L$ be Dirac--Jacobi structures, $i=1,2$.
	Then
	\begin{equation}
		\label{eq:BackPro}
		\Phi^!(\calL_1\star\calL_2)=(P^!\calL_1)\star (P^!\calL_2).
	\end{equation}
		
	\paragraph{\emph{Proof of the Claim}}
	Fix an arbitrary $(\delta^\prime,(D\Phi)^\ast\alpha)\in\Phi^!(\calL_1\star\calL_2)$ so that $((D\Phi)\delta^\prime,\alpha)\in\calL_1\star\calL_2$.
	Then there are $\alpha_i\in J^1L$, with $i=1,2$, such that $\alpha=\alpha_1+\alpha_2$ and $((D\Phi)\delta^\prime,\alpha_i)\in\calL_i$, for $i=1,2$.
	So, we get 
	\begin{equation*}
		(\delta^\prime,(D\Phi)^\ast\alpha)=(\delta^\prime,(D\Phi)^\ast\alpha_1+(D\Phi)^\ast\alpha_2)\in(\Phi^!\calL_1)\star(\Phi^!\calL_2).
	\end{equation*}
	This proves that $\Phi^!(\calL_1\star\calL_2)\subset 
	(\Phi^!\calL_1)\star(\Phi^!\calL_2)$.
	Both bundles are Lagrangian and thus equal.
\end{proof}

The next lemma shows the interplay between forward/backward Dirac--Jacobi maps and fiber products.

\begin{lemma}
	\label{lem:transverse_forward}
	Let $\Phi_1:(M_1,L_1,\calL_1)\to(M,L,\calL)$ be a forward Dirac--Jacobi map, covering $\varphi_1\colon M_1\to M$, and let 
	$\Phi_2\colon L_2\to L$ be a Dirac--Jacobi transversal to $(M,L,\calL)$, covering $\varphi_2\colon M_2\to M$.
	Then the fiber product of $\Phi_1:L_1\to L$ and $\Phi_2:L_2\to L$ in $\mathfrak{Line}$ exists 
	\begin{equation}
		\label{eq:lem:transverse_forward}
		\begin{tikzcd}
			L_1\times_LL_2 \arrow[r, "P_2"]\arrow[d, swap, "P_1"]& L_2 \arrow[d, "\Phi_2"]\\
			L_1 \arrow[r, swap, "\Phi_1"]& L
		\end{tikzcd}
	\end{equation}
	and $P_2:(M_1\times_MM_2,L_1\times_LL_2,P_1^!\calL_1)\longrightarrow(M_2,L_2,\Phi_2^!\calL)$ is a forward Dirac--Jacobi map. 
\end{lemma}

\begin{proof}
	The Dirac--Jacobi transversality condition $(DL)|_{\varphi_2(M_2)}=\pr_D\calL+(D\Phi_2)DL_2$ and the Dirac--Jacobi forward map condition $\Phi_{1!}\calL_1=\calL|_{\varphi_1(M_2)}$ allow to easily compute
	\begin{equation*}
		(TM)|_{\varphi_1(M_1)\cap\varphi_2(M_2)}=\sigma\left(\pr_D(\Phi_{1!}\calL_1)+(D\Phi_2)DL_2\right)=(T\varphi_1)(TM_1)+(T\varphi_2)TM_2.
	\end{equation*}
	The latter means that the maps $\varphi_1:M_1\to M$ and $\varphi_2:M_2\to M$ are transversal and this is sufficient for the existence of their fiber product in $\mathfrak{Man}$
	\begin{equation*}
		\begin{tikzcd}
			M_1\times_M M_2 \arrow[r, "p_2"]\arrow[d, swap, "p_1"]& M_2 \arrow[d, "\varphi_2"]\\
			M_1\arrow[r, swap, "\varphi_1"]& M
		\end{tikzcd}
	\end{equation*}
	Consequently, by Theorem~\ref{Thm:FiberedProduct_LB}, the fiber product~\eqref{eq:lem:transverse_forward} in $\mathfrak{Line}$ exists.
	
	As for the second part of the statement, the hypotesis and the canonical isomorphism $D(L_1\times_LL_2)\simeq DL_1\times_{DL}DL_2$ from Corollary~\ref{Cor: SplittingPPDFun} allow to check that $P_1\colon L_1\times_L L_2\to L_1$ is transverse to $(M_1,L_1,\calL_1)$, i.e.
	\begin{equation*}
		(DL_1)|_{p_1(M_1\times_MM_2)}=\pr_D\calL_1+(DP_1)(D(L_1\times_LL_2))
	\end{equation*}and hence $P_1^!\calL_1$ is a Dirac--Jacobi structure on $L_1\times_L L_2\to M_1\times_M M_2$ by Corollary~\ref{Cor: TrnsMap}.	
	Now fix arbitrary $x_1\in M_1$, $x_2\in M_2$, with $\varphi_1(x_1)=\varphi_2(x_2)=:x$.
	Unravelling the definitions, we get that $(P_{2!}P_1^!\calL_1)_{(x_1,x_2)}$ consists of those $(\delta_2,\alpha_2)\in D_{x_2}L_2\oplus J^1_{x_2}L_2$ such that%
	, for some unique $\alpha_1\in J^1_{x_1}L_1$ $\delta_1\in D_{x_1}L_1$,
	\begin{equation*}
		(D_{x_2}\Phi_2)\delta_2=(D_{x_1}\Phi_1)\delta_1=:\delta,\quad (D_{(x_1,x_2)}P_2)^\ast\alpha_2=(D_{(x_1,x_2)}P_1)^\ast\alpha_1,\quad (\delta_1,\alpha_1)\in(\calL_1)_{x_1},
	\end{equation*}
	where we have used Corollary~\ref{Cor: SplittingPPDFun}.
	For the same reasons, $(D_{(x_1,x_2)}P_2)^\ast\alpha_2=(D_{(x_1,x_2)}P_1)^\ast\alpha_1$ implies that $\alpha_i=(D_{x_i}\Phi_i)^\ast\alpha$, for $i=1,2$, and for some (unique) $\alpha\in J^1_xL$.
	Consequently, $(\delta_1,(D\Phi_1)^\ast\alpha)=(\delta_1,\alpha_1)\in(\calL_1)_{x_1}$ and $((D\Phi_1)\delta_1,\alpha)=(\delta,\alpha)=((D\Phi_2)\delta_2,\alpha)\in(\Phi_{1!}\calL_1)_{(x_1,x_2)}=\calL_x$ imply that $$(\delta_2,\alpha_2)=(\delta_2,(D_x\Phi_2)^\ast\alpha)\in(\Phi_2^!\calL)_{(x_1,x_2)}.$$
	This proves the inclusion $P_{2!}P_1^!\calL_1\subset p_2^\ast(\Phi_2^!\calL)$, and so also the equality $P_{2!}P_1^!\calL_1=p_2^\ast(\Phi_2^!\calL)$.
\end{proof}

\section{Weak Dual Pairs in Dirac--Jacobi Geometry}
\label{sec:weak_dual_pairs}

In this Section we introduce the notion of (weak) dual pairs in Dirac--Jacobi geometry (see Definition~\ref{def:weak_dual_pair}), prove their very first properties and show how they naturally pop up in connection with the theory of multiplicative precontact structures on Lie groupoids (see Theorem~\ref{theor:over-precontact_groupoid:WDP}). 
\subsection{Definition and First Properties}
Let us begin by simply giving the definition of a weak dual pair. 

\begin{definition}
	\label{def:weak_dual_pair}
	A \emph{weak dual pair} is a pair of forward Dirac--Jacobi maps on the same precontact manifold
\begin{equation}
	\label{eq:def:weak_dual_pair}
	\begin{tikzcd}
	(M_0,L_0,\calL_0)&(M,L,\operatorname{Gr}\varpi)\arrow[l, swap, "S"]\arrow[r, "T"]&(M_1,L_1,\calL_1^\text{opp})
	\end{tikzcd}
\end{equation}
	covering surjective submersions $\begin{tikzcd}M_0&M\arrow[l, swap, "s"]\arrow[r, "t"]&M_1\end{tikzcd}$ and satisfying the following two conditions:
	\begin{enumerate}[label={\arabic*)}]
		\item
		\label{enumitem:def:weak_dual_pair:1}
		$\ker DS$ and $\ker DT$ are orthogonal to each other wrt $\varpi$, i.e.$$\varpi(\ker DS, \ker DT)=0$$
		\item
		\label{enumitem:def:weak_dual_pair:2}
		$\ker DS\cap \ker\varpi^\flat \cap \ker DT$ has constant rank equal to $\dim(M)-\dim(M_0)-\dim(M_1)-1$, i.e.
		$$\rank(\ker DS\cap \ker\varpi^\flat \cap \ker DT )=\dim(M)-\dim(M_0)-\dim(M_1)-1$$
	\end{enumerate}	
	Additionally, if $\dim(M)=\dim(M_0)+\dim(M_1)+1$, we say that the diagram~\eqref{eq:def:weak_dual_pair} is a \emph{dual pair}.
\end{definition}

	We refer to the Dirac--Jacobi manifolds $(M_i,L_i,\calL_i)$, with $i=0,1$, as the \emph{legs} of the weak dual pair~~\eqref{eq:def:weak_dual_pair}.

\begin{remark}
	\label{rem:full_contact_dual_pairs}
	The notion of weak dual pair is related to the notion of contact dual pair introduced in~\cite{STV2019}.
	Indeed, it is easy to see that:
	\begin{itemize}
		\item If the legs of a weak dual pair~\eqref{eq:def:weak_dual_pair} are Jacobi manifolds, then $\varpi\in\Omega^2_D(L)$ is regular, with
		\begin{equation*}
			\label{eq:rem:full_contact_dual_pairs}
			\ker\varpi^\flat=\ker DS\cap\ker\varpi^\flat\cap\ker DT,
		\end{equation*}
		i.e.~$\ker\varpi^\flat\subset\ker DS\cap\ker DT$, and so with $\rank\varpi=2+\dim M_0+\dim M_1$.
		\item
		If the legs of a dual pair~\eqref{eq:def:weak_dual_pair} are Jacobi manifolds, then $\varpi\in\Omega^2_D(L)$ is non-degenerate, i.e.~$\varpi$ is a symplectic Atiyah form. 
		Hence $(M,L,\operatorname{Gr}\varpi)$ is a contact manifold, and $\ker DS=(\ker DT)^{\perp\varpi}$.
	\end{itemize}
	This proves that a dual pair~\eqref{eq:def:weak_dual_pair} with Jacobi legs is the same thing as a \emph{full contact dual pair} (see~\cite{STV2019}, in particular, Definition~3.1 and Proposition~4.4).
\end{remark}

A first consequence of the definition of weak dual pairs is the following.

\begin{proposition}
	\label{prop:isotropy_bundle}
	Let~\eqref{eq:def:weak_dual_pair} be a weak dual pair.
	Then $\ker DS\cap\ker\varpi^\flat\cap\ker DT\subset\bbD L$ is a smooth subbundle. 
\end{proposition}

\begin{proof}
	The VB morphism $K:\ker DS\oplus\ker DT\to\bbD L,\ (\delta,\delta^\prime)\mapsto(\delta-\delta^\prime,\varpi^\flat\delta^\prime)$ leads to the following short (pointwise) exact sequence
	\begin{equation*}
		\begin{tikzcd}
			0\arrow[r]&\ker DS\cap\ker\varpi^\flat\cap\ker DT\arrow[r]&\ker DS\oplus\ker DT\arrow[r, "K"]&\ker DS+(\ker DT)^\varpi\arrow[r]&0   
		\end{tikzcd}
	\end{equation*}
	where the second arrow is induced by the diagonal.
	Hence condition~\ref{enumitem:def:weak_dual_pair:2} in Definition~\ref{def:weak_dual_pair} means that the VB morphism $K:\ker DS\oplus\ker DT\to\bbD L$ has constant rank.
	Consequently, $\ker DS\cap\ker\varpi^\flat\cap\ker DT$ (and $\ker DS+(\ker DT)^\varpi$ as well) is a smooth subbundle of $\bbD L$.
\end{proof}

\begin{remark}
	\label{rem:isotropy_bundle}
	For future reference we point out here that, as it is easy to see, the following identities hold
	\begin{equation}
		\label{eq:rem:isotropy_bundle}
		\ker DS\cap\ker\varpi^\flat\cap\ker DT=\ker DS\cap(\ker DT)^\varpi=(\ker DS)^\varpi\cap\ker DT.
	\end{equation}
	Moreover, the latter has constant rank equal to $\dim M-\dim M_0-\dim M_1-1$ if and only if both $\ker DS+(\ker DT)^\varpi$ and $\ker DT+(\ker DS)^\varpi$ have constant rank equal to $\dim M+1$.
\end{remark}

\begin{remark}
	\label{rem:tiny_technical_remark}
	For future reference we also point out that the following identity holds
	\begin{equation}
		\label{eq:rem:tiny_technical_remark}
		\varpi^\flat(\ker DS\cap\ker DT)=(\ker DS + \ker DT)^\circ.
	\end{equation}
	Indeed, Condition~\ref{enumitem:def:weak_dual_pair:1} in Definition~\ref{def:weak_dual_pair} immediately implies that $\varpi^\flat(\ker DS\cap\ker DT)\subset(\ker DS + \ker DT)^\circ$.
	Further, the latter have the same rank, and so they are equal, as a consequence of Condition~\ref{enumitem:def:weak_dual_pair:2}.
\end{remark}

Note that not every two Dirac--Jacobi manifolds admit a (weak) dual pair connecting them.
Later on, in Section~\ref{sec:characteristic_leaf_correspondence}, we will find consequences of the existence of a weak dual pair between two Dirac--Jacobi manifolds which would make it easy to construct examples of Dirac--Jacobi manifolds which cannot fit in the same weak dual pair.
However, for now, we want to show that any two pre-contact manifols always fit in a dual pair, and hence a weak dual pair.  

\begin{example}
	\label{ex:product_precontact_manifolds}
	Let $(M_i,L_i,\operatorname{Gr}\varpi_i)$ be a precontact manifold, for $i=1,2$.
	Then their product,	in the sense of Corollary~\ref{cor:product_contact_structures}, and the projection maps form a dual pair (actually a full contact dual pair by Remark~\ref{rem:full_contact_dual_pairs})
	\begin{equation}
		\begin{tikzcd}
			(M_0,L_0,\operatorname{Gr}\varpi_0)&(M_0\times^!M_1,L_0\times^!L_1,\operatorname{Gr}(\varpi_0\times^!(-\varpi_1))\arrow[l, swap, "P_0"]\arrow[r, "P_1"])&(M_1,L_1,\operatorname{Gr}(-\varpi_1)).
		\end{tikzcd}
	\end{equation}
\end{example}

\subsection{(Weak) Dual Pairs from (Over-)Precontact Groupoids}
\label{sec:contact_groupoids}

The first main motivating examples of (weak) dual pairs in Dirac--Jacobi geometry (see Definition~\ref{def:weak_dual_pair}) appear in the context of (over-)precontact groupoids.
Let us recall, in particular, that the role of (pre)contact groupoids in (Dirac--)Jacobi geometry is analogous to the one played by (pre)symplectic groupoids in (Dirac)Poisson geometry.
Indeed, (pre)contact groupoids arise as ``desingularizations'' of (Dirac--)Jacobi manifolds, in the sense that, up to certain technical conditions, a (Dirac--)Jacobi manifold $M$ integrates to a (pre)contact groupoid (cf., e.g.,~\cite{ponte2006integration}).

Let $\calG\rightrightarrows\calG_0$ be a Lie groupoid with structure maps $s,t:\calG\to\calG_0$, $m:\calG^{(2)}:=\calG{}_s{\times}_t\calG\to\calG$, $u:\calG_0\to\calG$, and $i:\calG\to\calG$, and let $\Phi:\calG\times_{\calG_0} L_0\to L_0$ be a groupoid representation of $\calG\rightrightarrows\calG_0$ on a line bundle $L_0\to\calG_0$.
Notice that the associated action groupoid $L:=t^\ast L_0\rightrightarrows L_0$ has a natural structure of trivial-core LB groupoid with its structure maps, $S,T:L\to L_0$, $M:L^{(2)}:=L{}_S{\times}_TL\to L$, $U:L_0\to L$, and $I:L\to L$, being regular LB morphisms over the structure maps of $\calG\rightrightarrows\calG_0$.

As recalled in Proposition~\ref{prop:Atiyah_presymplectic_forms}, $L$-valued precontact forms $\theta$ on $\calG$ are in one-to-one correspondence with $L$-valued presymplectic Atiyah forms $\varpi$ on $\calG$ by means of the relation $\varpi=\theta\circ\sigma$.
Assume now to have equipped $\calG$ with the precontact structure equivalently given by the presymplectic Atiyah form $\varpi$ or the corresponding precontact form $\theta$.
Then the following conditions become equivalent:
\begin{itemize}
	\item the $L$-valued $1$-form $\theta$ is \emph{multiplicative}, i.e.~$M^\ast\theta=\operatorname{Pr}_1^\ast\theta+\operatorname{Pr}_2^\ast\theta$,
	\item the $L$-valued presymplectic form $\varpi$ is \emph{multiplicative}, i.e.~$M^\ast\varpi=\operatorname{Pr}_1^\ast\varpi+\operatorname{Pr}_2^\ast\varpi$.
\end{itemize}
Above $\operatorname{Pr}_1,\operatorname{Pr}_2:L^{(2)}:=L{}_S{\times}_TL\to L$ denote the standard projections.
We look now at when the multiplicative precontact structure on $\calG\rightrightarrows\calG_0$ can be pushed forward to a Dirac--Jacobi structure on $L_0\to\calG_0$ and so give rise to a weak dual pair.
A sufficient condition for this is found in the following.

\begin{proposition}
	\label{prop:robust_multiplicative_precontact}
	Assume that the multiplicative $L$-valued presymplectic Atyiah form $\varpi$ on $\calG$ is \emph{robust}, i.e.
	\begin{equation*}
		\rank(\ker DS\cap\varpi^\flat\cap\ker DT)|_{\calG_0}=\dim\calG-2\dim\calG_0-1,
	\end{equation*}
	Then there is a unique Dirac--Jacobi structure $\calL_0$ on $L_0\to\calG_0$ satisfying the following equivalent conditions:
	\begin{enumerate}
		\item
		\label{enumitem:prop:robust_multiplicative_precontact:1}
		$T:(\calG,L,\operatorname{Gr}\varpi)\longrightarrow (\calG_0,L_0,\calL_0)$ is a forward Dirac--Jacobi map, covering $t:\calG\to\calG_0$,
		\item
		\label{enumitem:prop:robust_multiplicative_precontact:2}
		$S:(\calG,L,\operatorname{Gr}\varpi)\longrightarrow (\calG_0,L_0,\calL_0^\text{opp})$ is a forward Dirac--Jacobi map, covering $s:\calG\to\calG_0$.
	\end{enumerate}
\end{proposition}

\begin{proof}
	Since $\varpi$ is multiplicative, one gets that $I^\ast\varpi=-\varpi$.
	This, together with $S\circ I= T$, imply the equivalence of Conditions~\ref{enumitem:prop:robust_multiplicative_precontact:1} and~\ref{enumitem:prop:robust_multiplicative_precontact:2}.
	Let us continue proving separately uniqueness and existence of $\calL_0$.
	
	\emph{Uniqueness.}
	If there exists a Dirac--Jacobi structure $\calL_0\subset\bbD L_0$ such that $T_!\operatorname{Gr}\varpi=t^\ast\calL_0$, then it will be necessarily given by the Lagrangian family $\calL_0:=(T_!\operatorname{Gr}\varpi)|_{\calG_0}$.
	For later use in this proof, let us point out that, using the multiplicativity of $\varpi$, the latter can also be expressed as
	\begin{equation}
		\label{eq:proof:prop:robust_multiplicative_precontact:0}
		\calL_0=\{(DT)\delta_1+\delta_2,\varpi^\flat\delta_1)\mid(\delta_1,\delta_2)\in(\ker DS)|_{\calG_0}\oplus((\ker\varpi^\flat)|_{\calG_0}\cap DL_0)\}.
	\end{equation}
	
	\emph{Existence.}
	We have to prove that $\calL_0$ is smooth and $T_!\operatorname{Gr}\varpi=t^\ast\calL_0$.
	As a preliminary step, let us point out that, since $\varpi$ is multiplicative, the following identity holds, for all $g\in\calG$, $\delta_{tg}\in \ker D_{tg}S$ and $\delta_g\in D_gL$,
	\begin{equation}
		\label{eq:proof:prop:robust_multiplicative_precontact:1}
		\varpi_g(D_{tg}R_g)\delta_{tg},\delta_g)=\varpi_{tg}(\delta_{tg},(D_gT)\delta_g).
	\end{equation}
	Let us also notice that the latter immediately implies that $\varpi(\ker DS,\ker DT)=0$.
	
	Let us start proving that $\calL_0\subset\bbD L_0$ is smooth subbundle.
	From Equation~\eqref{eq:proof:prop:robust_multiplicative_precontact:1} it follows that $(\ker DS)|_{\calG_0}\subset(\ker DT)|_{\calG_0}^\perp$ and so $\varpi^\flat(\ker DS)|_{\calG_0}\subset(\ker DT)|_{\calG_0}^\circ=((DT)^\ast(J^1L_0))|_{\calG_0}\simeq J^1L_0$.
	Therefore, we have the following well-defined VB morphism over $\id_{\calG_0}$
	\begin{equation*}
		F\colon\ker(DS)|_{\calG_0}\longrightarrow\bbD L_0,\ \delta\longmapsto((DT)\delta,\varpi^\flat\delta).
	\end{equation*}
	Since $\ker F=(\ker DS\cap\ker\varpi^\flat\cap\ker DT)|_{\calG_0}$ and $\rank(\ker DS)|_{\calG_0}=\dim\calG-\dim\calG_0$, the robustness assumption on $\varpi$ implies that $F\colon\ker(DS)|_{\calG_0}\longrightarrow\bbD L_0$ has constant rank equal to $\dim\calG_0+1$.
	So, $\image F\subset\bbD L_0$ is a smooth subbundle of rank $\dim\calG_0+1$.
	However, from Equation~\eqref{eq:proof:prop:robust_multiplicative_precontact:0} it follows that the Lagrangian family $\calL_0$ is actually included in $\image F$, and so $\calL_0=\image F$.
	This proves that $\calL_0$ is a smooth subbundle of $\bbD L_0$ and so a Dirac--Jacobi structure on $L_0\to\calG_0$.
	
	Let us continue proving that $T_!\operatorname{Gr}\varpi=t^\ast\calL_0$.
	From Equation~\eqref{eq:proof:prop:robust_multiplicative_precontact:1}, through a straigthforward computation, it follows that $(D_gT)(\ker D_gS\cap\ker\varpi^\flat_g)\subset\ker\varpi_{tg}^\flat\cap D_{tg}L_0$.
	Moreover, using again the multiplicativity of $\varpi$, one can easily compute, for all $g\in\calG$
	\begin{align*}
		\dim(\ker D_gS\cap\varpi^\flat_g)-\dim(\ker\varpi_{tg}^\flat\cap D_{tg}L_0)&=\dim(\ker D_{tg}S\cap\varpi^\flat_{tg})-\dim(\ker\varpi_{tg}^\flat\cap D_{tg}L_0)\\
		&=\frac{1}{2}\left(\dim\ker\varpi^\flat_{tg}-2\dim\calG_0+\dim\calG\right)\\
		&\phantom{=}-\frac{1}{2}\left(\dim\ker\varpi^\flat_{tg}+2\dim\calG_0-\dim\calG+1\right)\\
		&=\dim\calG-2\dim\calG_0-1.
	\end{align*}
	By this and the robustness of $\varpi$ one gets, for any $g\in\calG$, the following short exact sequence of linear maps
	\begin{equation}
		\label{eq:proof:prop:robust_multiplicative_precontact:2}
		\begin{tikzcd}
			0\arrow[r]&\ker D_gS\cap\ker\varpi^\flat_g\cap\ker D_gT\arrow[r, "\text{incl}"]&\ker D_gS\cap\ker\varpi^\flat_g\arrow[r, "D_gT"]&\ker\varpi^\flat_{tg}\cap D_{tg}L_0\arrow[r]&0.
		\end{tikzcd}
	\end{equation}
	We use the latter to prove $(T_!\operatorname{Gr}\varpi)_g=\calL_{0,tg}$.
	By Equation~\eqref{eq:proof:prop:robust_multiplicative_precontact:0}, an arbitrary $u\in\calL_{0,tg}$ can be written as 
	\begin{equation*}
		u=((D_{tg}T)\delta_1+\delta_2,\varpi_{tg}^\flat\delta_1),
	\end{equation*}
	for some $\delta_1\in\ker D_{tg}S$ and $\delta_2\in\ker\varpi^\flat_{tg}\cap D_{tg}L_0$.
	Further, also in view of of~\eqref{eq:proof:prop:robust_multiplicative_precontact:2}, there are $\widetilde\delta_1\in\ker D_gS$ and $\widetilde\delta_2\in\ker D_gS\cap\ker\varpi^\flat_g$ such that $\delta_1=(D_{tg}R_g)\widetilde\delta_1$ and $\delta_2=(D_gT)\widetilde\delta_2$.
	So, by Equation~\eqref{eq:proof:prop:robust_multiplicative_precontact:0}, one gets
	\begin{equation*}
		(D_gT)(\widetilde\delta_1+\widetilde\delta_2)=(D_gT)\widetilde\delta_1+\delta_2=(D_{tg}T)\delta_1+\delta_2\quad\text{and}\quad\varpi^\flat_g(\widetilde\delta_1+\widetilde\delta_2)=\varpi^\flat_g\widetilde\delta_1=(D_gT)^\ast(\varpi_{tg}^\flat\delta_1),
	\end{equation*}
	so that $(\widetilde\delta_1+\widetilde\delta_2,(D_gT)^\ast\varpi^\flat_{tg}\delta_1)\in(\operatorname{Gr}\varpi)_g$ and $u=((D_{tg}T)\delta_1+\delta_2,\varpi_{tg}^\flat\delta_1)\in (T_!\operatorname{Gr}\varpi)_g$.
	This proves the inclusion $T_!\operatorname{Gr}\varpi\subset t^\ast\calL_0$ and so the equality of these two Lagrangian families $T_!\operatorname{Gr}\varpi$ and $t^\ast\calL_0$.
\end{proof}

The last proposition represents the analogue for precontact/Dirac--Jacobi structures of the results first obtained in the presymplectic/Dirac setting~\cite{bursztyn2004integration} and, moreover, it motivates the next definition.
\begin{definition}
	\label{def:over-precontact_groupoid}
	An \emph{over-precontact groupoid} consists consists of a Lie groupoid $\calG\rightrightarrows\calG_0$, a trivial core LB groupoid $L\rightrightarrows L_0$ over $\calG\rightrightarrows\calG_0$, and a multiplicative presymplectic Atiyah form $\varpi\in\Omega_D^2(L)$ such that
	\begin{equation*}
		\rank(\ker DS\cap\varpi^\flat\cap\ker DT)|_{\calG_0}=\dim\calG-2\dim\calG_0-1.
	\end{equation*}
	If additionally $\dim\calG=2\dim\calG_0+1$, the latter is said to be a \emph{precontact groupoid}.
\end{definition}

Now the content of Proposition~\ref{prop:robust_multiplicative_precontact} can be rephrased as follows.
\begin{theorem}
	\label{theor:over-precontact_groupoid:WDP}
	Let $(\calG\rightrightarrows\calG_0,L\rightrightarrows L_0,\varpi)$ be an over-precontact groupoid.
	Denote by $\calL_0$ the Dirac--Jacobi structure induced on $L_0\to\calG_0$ as in Proposition~\ref{prop:robust_multiplicative_precontact}.
	Then the following is a weak dual pair
	\begin{equation}
		\label{eq:theor:Jacobi}
		\begin{tikzcd}
			(\calG_0,L_0,\calL_0^\text{opp})&(\calG,L,\operatorname{Gr}\varpi)\arrow[l, swap, "S"]\arrow[r, "T"]&(\calG_0,L_0,\calL_0).
		\end{tikzcd}
	\end{equation}
\end{theorem}

\section{Properties of Weak Dual Pairs}
\label{sec:properties_operations}

In this Section we study additional properties of weak dual pairs that will be crucial in discussing, in the next Sections~\ref{sec:self-dual_pair} and~\ref{sec:characteristic_leaf_correspondence}, the main results about them.
Specifically, we focus here on their possible equivalent characterizations (see Propositions~\ref{prop:characterizations_WDPs_I} and~\ref{prop:characterizations_WDPs_II}) and the most relevant operations we can perform on them, namely \emph{composition} (see Proposition~\ref{prop:transitivityWDP}) and \emph{pull-back along transversals} (see Proposition~\ref{prop:TransversePullBackWDPs}).

\subsection{Equivalent Characterizations of Weak Dual Pairs}
Let us now continue examining the structures of weak dual pairs and give some equivalent descriptions.
We start with a general statement about Dirac-Jacobi structures and maps covering surjective submersions.

\begin{lemma}
	\label{lem:SurjSub}
	Let $P\colon L\to L_0$ be a regular line bundle morphism covering a surjective submersion $p\colon M\to M_0$.
	Then, for any $L$-valued presymplectic Atiyah form $\varpi\in\Omega_D^2(L)$, one gets that
	\begin{equation}
		\label{eq:lem:SurjSub}
		P^!P_!\operatorname{Gr}\varpi= \ker DP + (\ker(DP)^{\perp_\varpi})^\varpi.
	\end{equation}	 
	Moreover, the latter is smooth if and only if $P_!\operatorname{Gr}\varpi$ is smooth.
\end{lemma}

\begin{proof}
	The proof of the first part is an easy computation
	\begin{align*}
		P^!P_!\operatorname{Gr}\varpi&=\{(\delta,(DP)^\ast\alpha)\mid(\delta^\prime,(DP)^\ast\alpha)\in\operatorname{Gr}\varpi\ \ \text{and}\ \ (DP)\delta=(DP)\delta^\prime\}\\
		&=\ker(DP)+\operatorname{Gr}\varpi\cap\pr_J^{-1}((\ker DP)^0)=\ker DP+((\ker DP)^{\perp\varpi})^\varpi.
	\end{align*}
	For the second part, assume first that $P_!\operatorname{Gr}\varpi$ is smooth.
	Since $P$ covers a submersion, $\ker(DP)^\ast=0$ and so $\ker(DP)^\ast\cap P_!\operatorname{Gr}\varpi=0$.
	Hence $P^!P_!\operatorname{Gr}\varpi$ is smooth by Theorem~\ref{Thm: CleanInt}.
	Conversely, assume that $P^!P_!\operatorname{Gr}\varpi$ is smooth.
	From Equation~\eqref{eq:lem:SurjSub} it follows that $\ker DP+(\ker(DP)^{\perp\varpi})$ has constant rank.
	So, also $\ker DP\cap(\ker(DP)^{\perp\varpi})=\ker(DP)\cap\operatorname{Gr}\varpi$ has constant rank.
	Hence $P_!P^!\operatorname{Gr}\varpi$ is smooth by Theorem~\ref{thm:CleanInt_forward_DJ}.
\end{proof}

The next Propositions~\ref{prop:characterizations_WDPs_I} and~\ref{prop:characterizations_WDPs_II} are useful tools for proving certain properties of (weak) dual pairs.
The analogue statements for weak dual pairs in Dirac geometry can be found in~\cite[Prop.~6]{FM2018}. 

\begin{proposition}[{\bf Characterization of Weak Dual Pairs I}]
	\label{prop:characterizations_WDPs_I}
	Let $S:L\to L_0$ and $T:L\to L_1$  be regular LB morphisms covering surjective submersion.
	For any presymplectic Atiyah form $\varpi\in\Omega^2_D(L)$ and Dirac--Jacobi structures $\calL_0\subset\bbD L_0$ and  $\calL_1\subset\bbD L_1$, we get that
	\begin{equation}
		\label{eq:prop:characterizations_WDPs_I}
		\begin{tikzcd}
			(M_0,L_0,\calL_0)&(M,L,\operatorname{Gr}\varpi)\arrow[l, swap, "S"]\arrow[r, "T"]&(M_1,L_1,\calL_1^\text{opp})
		\end{tikzcd}
	\end{equation}
	is a weak dual pair if and only if the following two conditions hold:
	\begin{enumerate}[label={\arabic*)}]
		\item
		\label{enumitem:prop:characterizations_WDPs_I:1}
		$S^!\calL_0$ is the B-field transform of $T^!\calL_1$ via $\varpi$, i.e.~$S^!\calL_0=(T^!\calL_1)^\varpi$,
		\item
		\label{enumitem:prop:characterizations_WDPs_I:2}
		$\ker DS\cap \ker\varpi^\flat \cap \ker DT$ has constant rank equal to $\dim(M)-\dim(M_0)-\dim(M_1)-1$.
	\end{enumerate}	

\end{proposition}

\begin{proof}
	Assume that~\eqref{eq:prop:characterizations_WDPs_I} is a weak dual pair.
	Since $S$ is a forward Dirac map, by Lemma \ref{lem:SurjSub} we have 
	\begin{equation*}
		S^!\mathcal{L}_0=S^!S_!\operatorname{Gr}\varpi=\ker DS+((\ker DS)^{\perp_\varpi})^\varpi.
	\end{equation*}
	Since $\ker DT\subset(\ker DS)^{\perp_\varpi}$ by Condition~\ref{enumitem:def:weak_dual_pair:1} in Definition~\ref{def:weak_dual_pair}, we conclude that
	\begin{equation*}
		S^!\mathcal{L}_0\supset \ker DS+ (\ker DT)^\varpi.
	\end{equation*}
	From the definition~\ref{def:backward_forward_DJ} and the proof of Proposition~\ref{prop:isotropy_bundle}, we deduce that both sides of the latter have rank $\dim M+1$ and so they are equal.
	This, and a similar argument with $(T,-\varpi)$ in place of $(S,\varpi)$, show that
	\begin{align*}
		S^!\mathcal{L}_0=\ker DS+ (\ker DT)^\varpi\qquad\text{and}\qquad T^!\mathcal{L}_1= \ker DT+ (\ker DS)^{-\varpi},
	\end{align*}
	and hence $S^!\mathcal{L}_0=(T^!\calL_1)^\varpi$ holds.
	
	Assume that Conditions~\ref{enumitem:prop:characterizations_WDPs_I:1} and~\ref{enumitem:prop:characterizations_WDPs_I:2} hold.
	From $\ker DS\subset S^!\calL_0$ and $\ker DT\subset T^!\calL_0$ it follows first that
	\begin{equation}
		\label{eq:proof:prop:characterizations_WDPs_I:1}
		\ker DS+(\ker DT)^\varpi\subset S^!\calL_0=(T^!\calL_1)^\varpi
	\end{equation}
	and then, by isotropy of $S^!\calL_0=(T^!\calL_1)^\varpi$, also that $\ker DS\subset(\ker(DT))^{\perp\varpi}$.
	Hence, one easily gets
	\begin{equation}
		\label{eq:proof:prop:characterizations_WDPs_I:2}
		\ker DS+(\ker DT)^\varpi\subset \ker DS+((\ker DS)^{\perp_\varpi})^\varpi=S^!S_!\operatorname{Gr}\varpi,
	\end{equation}
	where we used Lemma~\ref{lem:SurjSub}.
	From the rank condition and Remark~\ref{rem:isotropy_bundle} it follows that the rank of $\ker DS+(\ker DT)^\varpi$ is $\dim M+1$.
	Moreover, the RHS's of Equations~\eqref{eq:proof:prop:characterizations_WDPs_I:1} and~\eqref{eq:proof:prop:characterizations_WDPs_I:2} are isotropic.
	Hence $S^!\calL_0=S^!S_!\operatorname{Gr}\varpi$ and so, using Corollary~\ref{cor:forward_backward}, one can compute
	\begin{equation*}
		\mathcal{L}_0= S_!S^!\mathcal{L}_0=S_!S^!S_!\operatorname{Gr}\varpi=S_!\operatorname{Gr}\varpi.
	\end{equation*}
	The same arguments can be used to obtain that $T$ is also a forward map and so~\eqref{eq:prop:characterizations_WDPs_I} is a weak dual pair. 
\end{proof}

As a straightforward consequence of the proof of Proposition~\ref{prop:characterizations_WDPs_I} we get the following
\begin{corollary}
	\label{cor:prop:characterizations_WDPs_I}
	For any weak dual pair~\eqref{eq:def:weak_dual_pair}, the following identities of Dirac--Jacobi structures hold
	\begin{equation*}
		S^!\calL_0=\ker DS+(\ker DT)^\varpi\qquad\text{and}\qquad T^!\calL_1^\mathrm{opp}=\ker DT+(\ker DS)^\varpi.
	\end{equation*}
\end{corollary}

\begin{lemma}
	\label{lem:forward_DJ_product}
	Under the the same assumptions of Proposition~\ref{prop:characterizations_WDPs_I}, one gets that the following
	\begin{equation*}
		S\times^!T:(M,L,\operatorname{Gr}\varpi)\to (M_0\times^! M_1,L_0\times^!L_1,\calL_0\times^!\calL_1)
	\end{equation*}
	is a forward Dirac--Jacobi morphism if and only if the following two inclusions hold
	\begin{equation}
	\label{Eq: DPinclusions}
		S^!\calL_0\subset \ker DS + (\ker  DT)^\varpi \qquad \text{and}
	\qquad T^!\calL_1\subset \ker DT + (\ker DS)^{-\varpi}.
\end{equation}
\end{lemma}

\begin{proof}
	Recall that the product line bundle $L_0\times^!L_1$ was introduced in Theorem~\ref{Thm: ProdCatLine}, where also the product $S\times^!T:L\to L_0\times^!L_1$ was defined as the unique regular LB morphism s.t.~the following diagram commute 
	\begin{equation*}
		\begin{tikzcd}
			&& L\arrow[lld, bend right=20, swap, "S"]\arrow[rrd, bend left=20, "T"]\arrow[d, dashed, "S\times ^! T"] &&\\
			L_0 && L_0\times^!L_1\arrow[ll, "P_0"]\arrow[rr, swap, "P_1"] && L_1
		\end{tikzcd}
	\end{equation*}
	where the horizontal arrows denote the standard projections.
	Note also that, for any $x\in M$ and $\lambda_x\in L_x$, we have $(S\times^!T)\lambda_x=T_x\circ S_x^{-1}$ and the following canonical identifications hold (see Lemma~\ref{Lem: SplitDProd})
	\begin{equation*}
	(D(L_0\times^!L_1))_{T_xS_x^{-1}}\simeq(DL_0)_{sx}\oplus(DL_1)_{tx}\quad\textnormal{and}\quad(J^1(L_0\times^!L_1))_{T_xS_x^{-1}}\simeq(J^1L_0)_{sx}\oplus(J^1L_1)_{tx}.
	\end{equation*}
	Understanding these splitting, we obtain first the following expression for the fiber of $(S\times^!T)_!\operatorname{Gr}\varpi$
	\begin{equation}
	\label{eq:lem:forward_DJ_product:1}
		((S\times^!T)_!\operatorname{Gr}\varpi)_{T_xS_x^{-1}}=\{((D_xS)\delta_x+(D_xT)\delta_x,\alpha_{sx}+\alpha_{tx})\mid \varpi^\flat\delta_x=(D_xS)^\ast\alpha_{sx}+(D_xT)^\ast\alpha_{tx}\}.
	\end{equation}
	Moreover, we can also compute
	\begin{equation*}
		(P_0^!\calL_0)_{T_xS_x^{-1}}=(\calL_0)_{sx}\oplus(DL_1)_{tx}\quad\textnormal{and}\quad(P_1^!\calL_1)_{T_xS_x^{-1}}=(\calL_1)_{tx}\oplus(DL_1)_{tx},
	\end{equation*}
	and then we obtain the following expression for the fiber of $\calL_0\times^!\calL_1$ over $T_xS_x^{-1}$
	\begin{equation}
	\label{eq:lem:forward_DJ_product:2}
		(\calL_0\times^!\calL_1^\textnormal{opp})_{T_xS_x^{-1}}=(\calL_0)_{sx}\oplus(\calL_1^\textnormal{opp})_{tx}
	\end{equation}
	Now, from the comparison of Equations~\eqref{eq:lem:forward_DJ_product:1} and~\eqref{eq:lem:forward_DJ_product:2}, we get that $(\calL_0\times^!\calL_1^\textnormal{opp})_{T_xS_x^{-1}}\subset ((S\times^!T)_!\operatorname{Gr}\varpi)_{T_xS_x^{-1}}$ is equivalent to the following two inclusions
	\begin{equation}
		\label{eq:lem:forward_DJ_product:3}
		\begin{aligned}
			(\calL_0)_{sx}&\subset\{((D_xS)\delta_x,\alpha_{sx})\mid\delta_x\in\ker D_xT,\ \text{and}\ \varpi^\flat\delta_x=(D_xS)^\ast\alpha_{sx}\}\\
			(\calL_1)_{tx}&\subset\{((D_xT)\delta_x,\alpha_{tx})\mid\delta_x\in\ker D_xS,\ \text{and}\ \varpi^\flat\delta_x=(D_xT)^\ast\alpha_{tx}\}.
		\end{aligned}
	\end{equation}
	Finally, since $S:L\to L_0$ and $T:L\to L_1$ cover surjective submersions, in particular $D_xS:D_xL\to D_{sx}L_0$ and $D_xT:D_xL\to D_{tx}L_1$ are surjective and so the inclusions~\eqref{eq:lem:forward_DJ_product:3} are equivalent to the inclusions $(S^!\calL_0)_x\subset\ker D_xS+(\ker D_xT)^\varpi$ and $(T^!\calL_1)_{T_xS_x^{-1}}\subset\ker D_xT+(\ker D_xS)^\varpi$.
	This proves the claim.
\end{proof}

\begin{proposition}[{\bf Characterization of Weak Dual Pairs II}]
	\label{prop:characterizations_WDPs_II}
	Let $\Phi_i:L\to L_i$ be a regular LB morphism covering a surjective submersion $\varphi_i:M\to M_i$, for $i=0,1$.
	For any presymplectic Atiyah form $\varpi\in\Omega^2_D(L)$ and any Dirac--Jacobi structures $\calL_i\subset\bbD L_i$, with $i=0,1$,  we get that
	\begin{equation}
		\label{eq:prop:characterizations_WDPs_II}
		\begin{tikzcd}
			(M_0,L_0,\calL_0)&(M,L,\operatorname{Gr}\varpi)\arrow[l, swap, "{\Phi_0}"]\arrow[r, "{\Phi_1}"]&(M_1,L_1,\calL_1)
		\end{tikzcd}
	\end{equation}
	is a weak dual pair if and only if the following two conditions hold:
	\begin{enumerate}[label={\arabic*)}]
		\item
		\label{enumitem:prop:characterizations_WDPs_II:1}
		${\Phi_0}\times^! {\Phi_1}\colon (M,L,\operatorname{Gr}\varpi)\to (M_0\times^!M_1,L_0\times^!L_1,\calL_0\times^!\calL_1^{\mathrm{opp}})$ is a forward Dirac-Jacobi map and
		\item
		\label{enumitem:prop:characterizations_WDPs_II:2}
		one (and so all) of the following three properties are satisfied:
		\begin{enumerate}
			\item\label{enumitem:prop:characterizations_WDPs_II:2a}
			$\varpi(\ker D{\Phi_0},\ker D{\Phi_1})=0$
			\item\label{enumitem:prop:characterizations_WDPs_II:2b}
			$\rank(\ker D{\Phi_0} \cap \ker\varpi^\flat\cap \ker D{\Phi_1})=\dim(M)-\dim(M_0)-\dim(M_1)-1$,
			\item\label{enumitem:prop:characterizations_WDPs_II:2c}
			${\Phi_0}^!\calL_0=({\Phi_1}^!\calL_1)^\varpi$
		\end{enumerate}
	\end{enumerate}	 
\end{proposition}
	
\begin{proof}
	Let us start checking that, if Condition~\ref{enumitem:prop:characterizations_WDPs_II:2} holds, then the three conditions~\ref{enumitem:prop:characterizations_WDPs_II:2a}-\ref{enumitem:prop:characterizations_WDPs_II:2c} become equivalent.
	\\
	\ref{enumitem:prop:characterizations_WDPs_II:2a} $\Rightarrow$ \ref{enumitem:prop:characterizations_WDPs_II:2b}
	Since Condition~\ref{enumitem:prop:characterizations_WDPs_II:2} is equivalent to the inclusions~\eqref{Eq: DPinclusions}, from $\ker DS\subset(\ker DT)^{\perp\varpi}$ we get
	$$T^!\calL_1\subset\ker DT+(\ker DS)^\varpi\subset\ker DT+((\ker DT)^{\perp\varpi})^\varpi=T^!T_!\operatorname{Gr}\varpi,$$
	where we also used Lemma~\ref{lem:SurjSub}.
	Further, since $T^!\calL_1$ and $T^!T_!\operatorname{Gr}\varpi$ are both maximally isotropic, the latter tells that $\rank(\ker DT+(\ker DS)^\varpi)=\dim M+1$ and this  is equivalent to Condition~\ref{enumitem:prop:characterizations_WDPs_II:2b} by Remark~\ref{rem:isotropy_bundle}.
	\\
	\ref{enumitem:prop:characterizations_WDPs_II:2b} $\Rightarrow$ \ref{enumitem:prop:characterizations_WDPs_II:2c}
	Since Condition~\ref{enumitem:prop:characterizations_WDPs_II:2} is equivalent to the inclusions~\eqref{Eq: DPinclusions} by Lemma~\ref{lem:forward_DJ_product} and Condition~\ref{enumitem:prop:characterizations_WDPs_II:2b} is equivalent to $\rank(\ker DT+(\ker DS)^\varpi)=\dim M+1$ by Remark~\ref{rem:isotropy_bundle}, we get
	$$S^!\calL_1=\ker DS+(\ker DT)^\varpi\qquad\text{and}\qquad T^!\calL_1=\ker DT+(\ker DS)^\varpi.$$
	Hence we can compute $S^!\calL_0=(T^!\calL_1)^\varpi$.
	\\
	\ref{enumitem:prop:characterizations_WDPs_II:2b} $\Rightarrow$ \ref{enumitem:prop:characterizations_WDPs_II:2a}
	Since $\ker DS\subset S^!\calL_0$ and $\ker DT\subset T^!\calL_1$, from $S^!\calL_0=(T^!\calL_1)^\varpi$ we get first the inclusion
	$$\ker DS+\ker(DT)^\varpi\subset S^!\calL_0=(T^!\calL_1)^\varpi$$
	Then, by the isotropy of the RHS, we get $\varpi(\ker DS,\ker DT)=0$.
	
	Now let us assume Conditions~\ref{enumitem:prop:characterizations_WDPs_II:1} and~\ref{enumitem:prop:characterizations_WDPs_II:2}.
	This, and the previous step of the proof, implies that, in particular, Conditions~\ref{enumitem:prop:characterizations_WDPs_I:1} and~\ref{enumitem:prop:characterizations_WDPs_I:2} in Proposition~\ref{prop:characterizations_WDPs_I} hold.
	So~\eqref{eq:prop:characterizations_WDPs_II} is a weak dual pair.

	Now let us assume that~\eqref{eq:prop:characterizations_WDPs_II} is a weak dual pair.
	By Definition~\ref{def:weak_dual_pair}, we have $\varpi(\ker DS,\ker DT)=0$.
	Additionally, by Corollary~\ref{cor:prop:characterizations_WDPs_I}, we also have 
	\begin{equation*}
		S^!\calL_0= \ker DS +(\ker DT)^\varpi\quad \text{and}\quad T^!\calL_1= \ker DT +(\ker DS)^{-\varpi}
	\end{equation*}
	By Lemma~\ref{lem:forward_DJ_product}, the latter implies that $S\times^! T\colon (M,L,\operatorname{Gr}\varpi)\to (M_0\times^!M_1,L_0\times^!L_1,\calL_0\times^!\calL_1^{\mathrm{opp}})$ is a forward Dirac-Jacobi map.
	So, both Conditions~\ref{enumitem:prop:characterizations_WDPs_II:1} and~\ref{enumitem:prop:characterizations_WDPs_II:2} hold.
%
\end{proof}

\subsection{Operations on Weak Dual Pairs}
The relation ``being connected by a dual pair'' does not form an \emph{equivalence} in the category of Dirac-Jacobi manifolds because it fails to satisfy transitivity.
This is actually the reason to introduce weak dual pairs, since the relation ``being connected by a weak dual pair'' is an equivalence.
The idea and the proofs of this claim can be found in~\cite{FM2018} in the Dirac setting.
Since symmetry is obvious, we start here proving transitivity: in Section~\ref{sec:self-dual_pair}, we will discuss reflexivity.
We stress that, as we claimed before, transitivity does not work with dual pairs, but reflexivity does (see Theorem~\ref{thm:existence_self-dual_pair}).

\begin{proposition}[{\bf Transitivity}]
	\label{prop:transitivityWDP}
	Let $(M_0,L_0,\calL_0)$ and $(M_1,L_1,\calL_1)$ be connected by the weak dual pair
	\begin{equation}
		\label{eq:Lem: TransitivityDP:1}
		\begin{tikzcd}
			(M_0,L_0,\calL_0)&(M_{01},L_{01},\operatorname{Gr}\varpi_{01})\arrow[l, swap, "S_{01}"]\arrow[r, "T_{01}"]&(M_1,L_1,\calL_1^\mathrm{opp})
		\end{tikzcd}
	\end{equation}
	and let $(M_1,L_1,\calL_1)$ and $(M_2,L_2,\calL_2)$ be connected by the weak dual pair
	\begin{equation}
		\label{eq:Lem: TransitivityDP:2}
		\begin{tikzcd}
			(M_1,L_1,\calL_1)&(M_{12},L_{12},\operatorname{Gr}\varpi_{12})\arrow[l, swap, "S_{12}"]\arrow[r, "T_{12}"]&(M_2,L_2,\calL_2^\mathrm{opp}).
		\end{tikzcd}
	\end{equation}
	Then also the Dirac--Jacobi manifolds $(M_0,L_0,\calL_0)$ and $(M_1,L_1,\calL_1)$ are connected by a weak dual pair. 
\end{proposition}

\begin{proof}
	Since regular LB morphisms $T_{01}:L_{01}\to L_1$ and $S_{12}:L_{12}\to L_1$ cover surjective submersions, we can consider the fiber product line bundle $L_{02}=L_{01}\times_{L_1} L_{12}\to M_{02}=M_{01}\times_{M_1}M_{12}$ (see Theorem~\ref{Thm:FiberedProduct_LB})
	\begin{center}
		\begin{tikzcd}
			L_{02}\arrow[d, swap, "P_1"]\arrow[r, "P_2"] & L_{12}\arrow[d, "S_{12}"] \\
			L_{01}\arrow[r, swap, "T_{01}"] & L_1
		\end{tikzcd}
	\end{center}
	and, additionally, we get that also the projection maps $P_1$ and $P_2$ cover surjective submersions.
	Consequently, $S_{02} =S_{01}\circ P_1:L_{02}\to L_0$ and 
	$T_{02}=T_{12}\circ P_2:L_{02}\to L_2$ are regular LB morphisms covering surjective submersions and we can draw the following commutative diagram
	\begin{equation*}
		\begin{tikzcd}
			& & L_{02}\arrow[dl,"P_1"']\arrow[dr, "P_2"]\arrow[lldd, bend right=45, swap, "S_{02}"]\arrow[rrdd, bend left=45, "T_{02}"] & & \\
			& L_{01}\arrow[dl, "S_{01}"']\arrow[dr, "T_{01}"'] & & L_{12}\arrow[dl, "S_{12}"]\arrow[dr, "T_{12}"] &\\
			L_0& & L_1 & & L_2
		\end{tikzcd}
	\end{equation*}
	
	Let us define the closed $L_{02}$-valued Atiyah 2-form  $\varpi_{02}:=P_1^\ast\varpi_{01}-P_2^\ast\varpi_{12}$.
	Now we want to prove that the following is a weak dual pair
	\begin{equation}
		\label{eq:Lem: TransitivityDP:3}
		\begin{tikzcd}
			(M_0,L_0,\calL_0)&(M_{02},L_{02},\operatorname{Gr}\varpi_{02})\arrow[l, swap, "S_{02}"]\arrow[r, "T_{02}"]&(M_2,L_2,\calL_2^\mathrm{opp}).
		\end{tikzcd}
	\end{equation}
	First, we check that $S_{02}^!\mathcal{L}_0=(T_{02}^!\mathcal{L}_2)^{\omega_{02}}$ through the following straightforward computation
	\begin{align*}
		S_{02}^!\mathcal{L}_0&=P_1^!S_{01}^!\mathcal{L}_0=P_1^!((T_{01}^!\mathcal{L}_1)^{\omega_{01}})=(P_1^!T_{01}^!\mathcal{L}_1)^{P_1^\ast \omega_{01}}=(P_2^!S_{12}^!\mathcal{L}_1)^{P_1^\ast \omega_{01}},\\
		(T_{02}^!\mathcal{L}_2)^{\omega_{02}}&=(P_2^!T_{12}^!\mathcal{L}_2)^{P_1^\ast \omega_{01}+P_2^\ast\omega_{12}}=(P_2^!((T_{12}^!\mathcal{L}_2)^{\omega_{12}}))^{P_1^\ast \omega_{01}}=(P_2^!S_{12}^!\mathcal{L}_1)^{P_1^\ast \omega_{01}}.
	\end{align*}
	Then, in view of Corollary~\ref{Cor: SplittingPPDFun}, we get that $DL_{02}=DL_{01}\times_{DL_1}DL_{12}$.
	So, we can easily get also that
	\begin{equation*}
		\ker D S_{02}\cap\ker\varpi_{02}^\flat\cap\ker D T_{02}=(\ker D S_{01}\cap\ker\varpi_{01}^\flat)\times_{DL_1}(\ker\varpi_{01}^\flat\cap\ker D T_{12}),
	\end{equation*}
	where we use the definition $\varpi_{02}=P^\ast\varpi_{01}+P_2^\ast\varpi_{12}$.
	Consequently we can compute
	\begin{align*}
		\rank (\ker D S_{02}\cap\ker\varpi_{02}^\flat\cap\ker D T_{02})&=\rank((\ker D S_{01}\cap\ker\varpi_{01}^\flat\cap\ker DT_{01})+\rank DL_1\\
		&\phantom{=}+\rank(\ker DS_{12}\cap\ker\varpi_{01}^\flat\cap\ker D T_{12})\\
		&=\dim M_{01}-\dim M_0-\dim M_1-1+\dim M_1 +1\\
		&\phantom{=}+\dim M_{12}-\dim M_1-\dim M_2-1\\
		&=\dim M_{02}-\dim M_0-\dim M_2-1.
	\end{align*}
	So, by Lemma~\ref{prop:characterizations_WDPs_II}, we get that~\eqref{eq:Lem: TransitivityDP:3} is a weak dual pair and this proves the claim.
\end{proof}

\begin{remark}
	\label{rem:composition}
	The weak dual pair~\eqref{eq:Lem: TransitivityDP:3} is called the \emph{composition} of the weak dual pairs~\eqref{eq:Lem: TransitivityDP:1} and~\eqref{eq:Lem: TransitivityDP:2}.
\end{remark}

We discuss next operations with (weak) dual pairs. Again transversals play an important role, namely there is a notion of ``pulling back'' (weak) dual pairs along Dirac--Jacobi transversals.

\begin{proposition}
	\label{prop:TransversePullBackWDPs}
	For any (weak) dual pair 
	\begin{equation}
		\begin{tikzcd}
			\label{eq:prop:TransversePullBackWDPs:given_WDP}
			(M_0,L_0,\calL_0)&(M,L,\operatorname{Gr}\varpi)\arrow[l, swap, "S"]\arrow[r, "T"]&(M_1,L_1,\calL_1^\mathrm{opp})
		\end{tikzcd}
	\end{equation}
	and any regular LB morphism $\Phi_i\colon L_{N_i}\hookrightarrow L_i$ transversal to $\mathcal{L}_i$, with $i=0,1$, define: 
	\begin{enumerate}[label={\arabic*)}]
		\item line bundle $L_\Sigma:={L_{N_0}}\times_{L_0} {L}\times_{L_1} L_{N_1}\to\Sigma:=N_0\times_{M_0} M \times_{M_1} N_1$,
		\item regular LB morphisms $S_\Sigma\colon L_\Sigma\to L_{N_0},\ (\lambda_{x_0},\lambda_x,\lambda_{x_1})\mapsto \lambda_{x_0}$, $T_\Sigma\colon L_\Sigma\to L_{N_1},\ (\lambda_{x_0},\lambda_x,\lambda_{x_1})\mapsto \lambda_{x_1}$, and $P_2:L_\Sigma\to L,\ (\lambda_{x_0},\lambda_x,\lambda_{x_1})\mapsto\lambda_x$,
		\item $L_\Sigma$-valued Atiyah presymplectic form $\varpi_\Sigma:=P_2^\ast\varpi$, 
	\end{enumerate}
	Then the following is a (weak) dual pair, the \emph{transverse pull-back} of~\eqref{eq:prop:TransversePullBackWDPs:given_WDP},
	\begin{equation}
		\label{eq:prop:TransversePullBackWDPs:transverse_pullback}
		\begin{tikzcd}
			(N_0,L_{N_0},\Phi_0^!\calL_0)&(\Sigma,L_\Sigma,\operatorname{Gr}\varpi_\Sigma)\arrow[l, swap, "S_\Sigma"]\arrow[r, "T_\Sigma"]&(N_1,L_{N_1},(\Phi_1^!\calL_1)^\mathrm{opp})
		\end{tikzcd}
	\end{equation}
	if and only if both $S_\Sigma:L_\Sigma\to L_{N_0}$ and $T_\Sigma:L_\Sigma\to L_{N_1}$ cover surjective submersions. 
\end{proposition} 

\begin{proof}
	By Theorem~\ref{Thm: ProdCatLine}, $S\times^!T:L\to L_0\times^!L_1$ and $S_\Sigma\times^!T_\Sigma:L_\Sigma\to L_{N_0}\times^!L_{N_1}$ are the unique regular LB morphisms s.t.~the following diagram, with the  the natural projections as horizontal arrows, commute
	\begin{equation*}
		\begin{tikzcd}
			&& L\arrow[lld, bend right=20, swap, "S"]\arrow[rrd, bend left=20, "T"]\arrow[d, dashed, "S\times ^! T"] &&\\
			L_0 && L_0\times^!L_1\arrow[ll]\arrow[rr] && L_1
		\end{tikzcd}
		\qquad
		\begin{tikzcd}
			&& L\arrow[lld, bend right=20, swap, "S_\Sigma"]\arrow[rrd, bend left=20, "T_\Sigma"]\arrow[d, dashed, "S_\Sigma\times ^! T_\Sigma"] &&\\
			L_{N_0} && L_{N_0}\times^!L_{N_1}\arrow[ll]\arrow[rr] && L_{N_1}
		\end{tikzcd}
	\end{equation*}
	By Lemma~\ref{Lem: BackProd} there exists a unique regular LB morphism $\Phi_0\times^!\Phi_1:L_{N_0}\times^! L_{N_1}\to L_0\times^!L_1$ such that the following diagram, where the horizontal arrows are the natural projections, commutes
	\begin{equation*}
		\begin{tikzcd}
			L_{N_0}\arrow[d, swap, "\Phi_0"] && L_{N_0}\times^!L_{N_1}\arrow[ll]\arrow[rr]\arrow[d, dashed, "\Phi_0\times^!\Phi_1"] && L_{N_1}\arrow[d, "\Phi_1"]\\
			L_0 && L_0\times^!L_1\arrow[ll]\arrow[rr] && L_1
		\end{tikzcd}
	\end{equation*}
	and, additionally, $\Phi_0\times^!\Phi_1:(N_0\times^!N_1,L_{N_0}\times^!L_{N_1},(\Phi^!\calL_0)\times^!(\Phi_1^!\calL_1))\to(M_0\times^!M_1,L_0\times^!L_1,\calL_0\times^!\calL_1)$ is a backward Dirac-Jacobi map.
	Further, since $\Phi_i:L_{N_i}\to L_i$ is transversal to $\calL_i$, for $i=0,1$, using the splittings of $D(L_0\times^!L_1)$ and $D(L_{N_0}\times^!L_{N_1})$ from Lemma~\ref{Lem: SplitDProd} it is easy to see that $\Phi_0\times^!\Phi_1:L\to L_1\times^!L_2$ is transversal to $\calL_1\times^!\calL_2$.
	
	Note that, by construction, we have that the following commutative diagram of regular LB morphisms 
	\begin{equation}
		\label{Eq: PBDLine}
		\begin{tikzcd}
			L_\Sigma \arrow[rr, "S_\Sigma\times^!T_\Sigma"]\arrow[d, swap, "P_2"] &&  L_{N_0}\times^!L_{N_1} \arrow[d, "\Phi_0\times^! \Phi_1"]\\
			L \arrow[rr, swap, "S\times^! T"] && L_0\times^! L_1 
		\end{tikzcd}
	\end{equation}
	and, moreover, it is easy to see that the latter is a pull-back diagram in $\mathfrak{Line}$ as in Theorem~\ref{Thm:FiberedProduct_LB}. 
	So, Lemma~\ref{lem:transverse_forward} implies that the following is a forward Dirac--Jacobi map
	\begin{equation*}
		S_\Sigma\times^!T_\Sigma:(M_\Sigma,L_\Sigma,\operatorname{Gr}\varpi_\Sigma)\longrightarrow(N_0\times^!N_1,L_{N_0}\times^!L_{N_1},(\Phi_0^!\calL_0)\times^!(\Phi_1^!\calL_1)),
	\end{equation*}
	where we use the fact that $P_2:(\Sigma,L_\Sigma,\operatorname{Gr}\varpi_\Sigma)\to(M,L,\operatorname{Gr}\varpi)$ is a backward Dirac--Jacobi map by construction of $\varpi_\Sigma$.
	Note also that, by the very definition of $L_\Sigma$, we have the following commutative diagram of regular line bundle morphisms
	\begin{center}
		\begin{tikzcd}
			L_{N_0}\arrow[d, swap, "\Phi_0"] && L_\Sigma\arrow[ll, swap, "S_\Sigma"]\arrow[rr, "T_\Sigma"]\arrow[d, "P_2"] && L_{N_1}\arrow[d, "\Phi_1"]\\
			L_0 && L\arrow[ll, "S"]\arrow[rr, swap, "T"] && L_1
		\end{tikzcd}
	\end{center}
	and hence, using $S^!\calL_0=(T^!\calL_1)^\varpi$ and basic properties of the backward transform, we can compute 
	\begin{equation*}
		S_\Sigma^!\Phi_0^!\calL_0=P_2^!S^!\calL_0=P_2^!((T^!\calL_1)^\varpi)=(P_2^!T^!\calL_1)^{P_2^\ast\varpi}=(T_\Sigma^!\Phi_1^!\calL_1)^{\varpi_\Sigma}.
	\end{equation*}
By Proposition~\ref{prop:characterizations_WDPs_II}, this proves that~\eqref{eq:prop:TransversePullBackWDPs:transverse_pullback} is a weak dual pair iff $S_\Sigma:L_\Sigma\to L_{N_0}$ and $T_\Sigma:L_\Sigma\to L_{N_1}$ cover surjective submersions. 
\end{proof}

\section{Existence of Self Dual Pairs with Applications}
\label{sec:self-dual_pair}

In this Section we complete the proof of the fact that the property of fitting in a weak dual pair like~\eqref{eq:def:weak_dual_pair} defines an equivalence relation for Dirac--Jacobi manifolds $(M_1,L_1,\calL_1)$ and $(M_2,L_2,\calL_2)$.
Actually, the only part that still remains to prove, namely the reflexivity, follows from a slightly stronger result represented by the existence of self dual pairs (Theorem~\ref{thm:existence_self-dual_pair}).
In addition to prove the latter, we also apply it to get an alternative proof (Theorem~\ref{theor:NormalForm_DJtransversals}) of the Normal Norm Theorem around Dirac--Jacobi transversals~\cite{NorForJac}.

\subsection{Dirac--Jacobi Sprays and Proof of the Existence}
This section aims at proving that, for every Dirac-Jacobi manifold $(M_0,L_0,\calL_0)$, there is a line 
bundle $L\to M$ and a closed $L$-valued Atiyah form $\varpi\in\Omega_D^2(L)$, together with regular LB morphisms
$S,T\colon L\to L_0$, such that the following is a dual pair, which we will call a \emph{self dual pair} for $(M_0,L_0,\calL_0)$,
\begin{equation*}
	\begin{tikzcd}
		(M_0,L_0,\calL_0)&&(M,L,\operatorname{Gr}\varpi)\arrow[ll, swap, "S"]\arrow[rr, "T"]&&(M_0,L_0,\calL_0^\mathrm{opp}).
	\end{tikzcd}
\end{equation*}
This will represent the analogue in Dirac--Jacobi geometry of the result first obtained in the Dirac setting in~\cite{FM2018}.
Let us collect the necessary ingredients for proving this result starting with the Dirac--Jacobi sprays.

\begin{definition}
	\label{def:DJ_spray}
	Let $(M,L,\calL)$ be a Dirac-Jacobi manifold and let $p^\ast L:=\calL\times_ML\to\calL$ be the pull-back line bundle given by the diagram
	\begin{equation}
		\label{eq:def:DJ_spray}
		\begin{tikzcd}
			p^\ast L \arrow[r, "P"]\arrow[d] & L\arrow[d]\\
			\calL\arrow[r, swap, "p"]  & M
		\end{tikzcd}.
	\end{equation}
	A \emph{Dirac--Jacobi spray} for $(M,L,\calL)$ is a derivation $\Sigma\in \calD(p^\ast L)$ satisfying the following two properties 
	\begin{enumerate}[label={\arabic*)}]
		\item
		\label{enumitem:def:DJ_spray:1}
		$(DP)\Sigma_u=\pr_D u$, for all $u\in\calL$,
		\item
		\label{enumitem:def:DJ_spray:2}
		$M_t^\ast\Sigma=t\Sigma$, for all $t\in\bbR^\times$,
	\end{enumerate}
	where, for any $t\in\bbR$, we denote by $m_t\colon\calL\to\calL,\ u\mapsto tu,$ the fiberwise scalar multiplication and by $M_t\colon p^\ast L\to p^\ast L$ the line bundle morphism, covering $m_t\colon\calL\to\calL$, given by $M_t(u,\lambda)=(tu,\lambda)$.
\end{definition}

\begin{remark}
	By Condition~\ref{enumitem:def:DJ_spray:2}, the flow of $\Sigma$ is a $1$-parameter group $\Phi^\Sigma_\epsilon$ of local LB automorphisms s.t.
\begin{equation}
	\label{Eq: FlowSprayI}
	\Phi^\Sigma_\epsilon\circ M_t=M_t\circ\Phi^\Sigma_{t\epsilon},\quad \text{for all}\ t,\epsilon\in\bbR.
\end{equation}
Consequently, $\Phi^\Sigma_\epsilon$ covers a $1$-parameter group $\varphi^\Sigma_\epsilon$ of local diffeomorphisms, i.e.~the flow of $\sigma(\Sigma)$, satisfying
\begin{equation}
	\label{Eq: FlowSprayII}
	\varphi^\Sigma_\epsilon\circ m_t=m_t\circ\varphi^\Sigma_{t\epsilon},\quad \text{for all}\ t,\epsilon\in\bbR.
\end{equation}
Actually, Condition~\ref{enumitem:def:DJ_spray:2} is equivalent to~\eqref{Eq: FlowSprayI} and can also be written as $(DM_t)\Sigma=t^{-1}\Sigma$, for all $t\in\bbR^\times$.
\end{remark}

To prove the existence of self-dual pairs, we will need the following properties of Dirac--Jacobi sprays.

\begin{lemma}
	\label{rem:existence_DJ_spray}
	Each Dirac--Jacobi manifold admits a Dirac--Jacobi spray.
\end{lemma}

\begin{proof}
	Given a Dirac--Jacobi manifold $(M,L,\calL)$, an associated Dirac-Jacobi spray $\Sigma$ can be constructed as follows.
	Let us start considering the following short exact sequence of VB morphisms over $\id_\calL$
	\begin{equation}
		\label{eq:Ehresman_connection:splitting}
		\begin{tikzcd}
			0\arrow[rr]&&p^\ast\calL\arrow[rr, "\operatorname{vert}"]&&D(p^\ast L)\arrow[rr, "DP"]&&p^\ast(DL)\arrow[rr]&&0,
		\end{tikzcd}
	\end{equation}
	where $\operatorname{vert}\colon p^\ast\calL=\calL\times_M\calL\longrightarrow D(p^\ast L),\ (u,e)\longmapsto e^\uparrow_u$, is the canonical \emph{vertical lift} defined by
	\begin{equation}
		\label{eq:vertical_lift}
		e^\uparrow_u(fp^\ast\mu)=\left.\frac{\rmd}{\rmd t}\right|_{t=0}f(u+te)\mu_x,
	\end{equation}
	for all $x\in M$, $(u,e)\in\calL_x\times\calL_x=(p^\ast\calL)_u$, $f\in C^\infty(M)$ and $\mu\in\Gamma(L)$.
	Now, in this setting, an \emph{Ehresmann connection} is given by an \emph{horizontal lift}, i.e.~a VB-morphism $\operatorname{hor}:p^\ast(DL)\longrightarrow D(p^\ast L),\ (u,\delta)\longmapsto\operatorname{hor}_u\delta$, splitting~\eqref{eq:Ehresman_connection:splitting}.
	Further, an Ehresmann connection is called \emph{linear} if it satisfies $(DM_t)\circ\operatorname{hor}_u=\operatorname{hor}_{tu}$, for all $u\in\calL$ and $t\in\bbR$.
	By a partition of unity argument, one can easily prove the existence of linear Ehresmann connection.
	Finally, for any choice of a linear Ehresmann connection $\operatorname{hor}:p^\ast(DL)\longrightarrow D(p^\ast L)$, one can construct a Dirac--Jacobi spray $\Sigma$ for $(M,L,\calL)$ by setting $\Sigma_u=\operatorname{hor}_u(\pr_Du)$, for all $u\in\calL$.
\end{proof}

\begin{remark}
	Notice that, if restricted along $M$, the vertical lift gives rise to a VB morphism $\operatorname{vert}\colon\calL\to(D(p^\ast L))|_M,\ e_x\mapsto e_x^\uparrow$, characterized by
	\begin{equation}
		\label{eq:vertical_lift_on_M}
		e^\uparrow_x\lambda=\left.\frac{\rmd}{\rmd t}\right|_{t=0}M_0((M_t^\ast\lambda)_{e_x}).
	\end{equation}
	for all $x\in M$, $e_x\in\calL_x$ and $\lambda\in\Gamma(L)$.
	Similarly, if restricted along $M$, the short exact sequence~\eqref{eq:Ehresman_connection:splitting} reduces to the following short exact sequence of VB morphisms over $\id_M$
	\begin{equation}
		\label{eq:Ehre}
		\begin{tikzcd}
			0\arrow[rr]&&\calL\arrow[rr, "\operatorname{vert}"]&&(D(p^\ast L))|_M\arrow[rr, "DP"]&&DL\arrow[rr]&&0,
		\end{tikzcd}
	\end{equation}
	The latter has a canonical splitting given by $DI:DL\to D(p^\ast L)|_M$, where the regular LB morphism $I:L\to p^\ast L$ is the natural inclusion defined by $I(\lambda_x)=(x,\lambda_x)$, for all $x\in M$ and $\lambda_x\in L_x$.
	Equivalently, this canonical splitting is also given by the projection $\pr_\calL:D(p^\ast L)\to\calL$ defined by
	\begin{equation*}
		\psi(\pr_\calL\delta_x)\lambda_x=\delta_x(\psi p^\ast\lambda),
	\end{equation*}
	for all $x\in M$, $\delta_x\in D_x(p^\ast L)$, $\psi\in\Gamma(\calL^\ast)$ and $\lambda\in\Gamma(L)$, where the sections of the dual bundle $\calL^\ast$ identify with the fiber-wise linear functions on $\calL$.
	In the following, we will systematically understand the induced direct sum decomposition of vector bundles over $M$
	\begin{equation*}
		(D(p^\ast L))|_M=DL\oplus\calL.
	\end{equation*}
\end{remark}

\begin{lemma}
	\label{Lem: PropDJS}
	Let $(M,L,\calL)$ be a Dirac--Jacobi manifold with a Dirac--Jacobi spray $\Sigma\in\calD(p^\ast L)$.
	Then 
	\begin{enumerate}[label={\arabic*)}]
		\item
		\label{enumitem:Lem: PropDJS:1}
		$\Sigma|_M=0$, i.e.~the local LB automorphisms $\Phi_\epsilon^\Sigma:p^\ast L\longrightarrow p^\ast L$ is the identity map on $L=(p^\ast L)|_M$,
		\item
		\label{enumitem:Lem: PropDJS:2}
		$(D\Phi^{\Sigma}_\epsilon)|_M\colon (Dp^\ast L)\at{M}\to (Dp^\ast L)\at{M}$ is given by $(\delta_x,u_x)\mapsto
		(\delta_x+\epsilon\pr_D(u_x),u_x)$.
	\end{enumerate}
\end{lemma}

\begin{proof}
	Let us start proving point~\ref{enumitem:Lem: PropDJS:1}.
	Since the line bundle morphism $M_0:p^\ast L\to p^\ast L$ acts like the identity map on $L=(p^\ast L)|_M$, setting $t=0$ in Equation~\eqref{Eq: FlowSprayI} one immediately gets that $\Phi_\epsilon^\Sigma|_M=\id_L$. 
	
	Let us continue with point~\ref{enumitem:Lem: PropDJS:2}.
	In view of point~\ref{enumitem:Lem: PropDJS:1} above, $D\Phi_\epsilon^\Sigma$ is the identity on $DL\subset(D(p^\ast L))|_M$, i.e.
	\begin{equation*}
		(D\Phi_\epsilon^\Sigma)\delta_x=\delta_x,
	\end{equation*}
	for all $x\in M$ and $\delta_x\in D_xL$.
	Fix now arbitrary $x\in M$ and $e_x\in \calL_x$.
	For any $\lambda\in \Gamma(L)$, one can compute 
	\begin{align*}
		((DP)(D\Phi_\epsilon^\Sigma)e_x^\uparrow)\lambda&=P\Phi_\epsilon^\Sigma(e_x^\uparrow((\Phi_\epsilon^\Sigma)^\ast P^\ast\lambda)=\left.\frac{\rmd}{\rmd t}\right|_{t=0}P\Phi_\epsilon^\Sigma M_0((M_t^\ast(\Phi_\epsilon^\Sigma)^\ast P^\ast\lambda)_{e_x})\\
		&=\left.\frac{\rmd}{\rmd t}\right|_{t=0}P(((\Phi_{t\epsilon}^\Sigma)^\ast P^\ast\lambda)_{e_x})=\epsilon((D_{e_x}P)\Sigma)\lambda=\epsilon\pr_D(e_x)\lambda,
	\end{align*}
	where we used Equations~\eqref{Eq: FlowSprayI} and~\eqref{eq:vertical_lift_on_M}, the identity $P\circ M_t=P$, for all $t\in\bbR^\times$, as well as Condition~\ref{enumitem:def:DJ_spray:1} in Definition~\ref{def:DJ_spray}.
	This proves that, for all $x\in M$, $e_x\in\calL_x$, we have
	\begin{equation*}
		(DP)(D\Phi_\epsilon^\Sigma) e_x^\uparrow=\pr_D e_x.
	\end{equation*}
	Moreover, for any $\psi\in\Gamma(\calL^\ast)$ and $\lambda\in \Gamma(L)$, one can also easily compute
	\begin{align*}
		((D\Phi^\Sigma_\epsilon)e_x^\uparrow)(\psi p^\ast\lambda)&=\Phi_\epsilon^\Sigma(e_x^\uparrow(\Phi_\epsilon^\Sigma)^\ast(\psi p^\ast\lambda))=\left.\frac{\rmd}{\rmd t}\right|_{t=0}\Phi_\epsilon^\Sigma M_0\left((M_t^\ast(\Phi_\epsilon^\Sigma)^\ast(\psi p^\ast\lambda))_{e_x}\right)\\
		&=\left.\frac{\rmd}{\rmd t}\right|_{t=0}t M_0\left(((\Phi_{t\epsilon}^\Sigma)^\ast(\psi p^\ast\lambda))_{e_x}\right)=M_0\left((\psi p^\ast\lambda)_{e_x}\right)=\psi(e_x)\lambda_x,
	\end{align*}
	where we used the fiber-wise linearity of $\psi$, Equations~\eqref{Eq: FlowSprayI} and~\eqref{eq:vertical_lift_on_M}, as well as the identity $\Phi_\epsilon^\Sigma\circ M_0=M_0$, for all $\epsilon\in\bbR$, which follows from point~\ref{enumitem:Lem: PropDJS:1} above.
	This proves that we have
	\begin{equation*}
		\pr_\calL ((D\Phi^\Sigma_\epsilon)e_x^\uparrow)=e_x^\uparrow,
	\end{equation*}
	for all $x\in M$, $e_x\in\calL_x$, and so it concludes the proof.
\end{proof}

\begin{remark}
	To give a constructive proof of the existence of self-dual pairs, we also need to recall that, for any line bundle $L\to M$, the total space of its first jet bundle $J^1L\overset{\pi}{\to}M$ is endowed with a canonical contact structure.
	To this end let us consider the pull-back line bundle given by the following diagram
	\begin{equation}
		\label{eq:def:canonical_contact_structure}
		\begin{tikzcd}
			\pi^\ast L \arrow[r, "\Pi"]\arrow[d] & L\arrow[d]\\
			J^1L\arrow[r, swap, "\pi"]  & M
		\end{tikzcd}.
	\end{equation}
	Then the \emph{tautological $\pi^\ast L$-valued Atiyah $1$-form} on $J^1L$ is the unique $\lambda_{\text{taut}}\in\Omega^1_D(\pi^\ast L)=\Gamma(J^1(\pi^\ast L))$ which satisfies the tautological property:
	\begin{equation}
		\label{eq:tautological_property}
		\alpha^\ast\lambda_{\text{taut}}=\alpha,\ \text{for all}\ \alpha\in\Omega^1_D(L)=\Gamma(J^1L).
	\end{equation}
	In equivalent terms, the tautological $1$-form is also completely characterized by the following property
	\begin{equation*}
		\label{eq:tautological_1-form}
		\lambda_{\text{taut},\alpha_x}(\delta_{\alpha_x})=\alpha_x((D\Pi)\delta_{\alpha_x}),
	\end{equation*}
	for all $x\in M$, $\alpha_x\in J^1_xL$ and $\delta_{\alpha_x}\in D_{\alpha_x}(\pi^\ast L)$.
	Further, the \emph{Cartan contact structure} on $J^1L$ is encoded by the \emph{canonical $\pi^\ast L$-valued symplectic Atiyah form} $\varpi_{\text{can}}\in\Omega^2_D(\pi^\ast L)$ given by
	\begin{equation}
		\label{eq:canonical_symplectic_Atiyah_form}
		\varpi_\text{can}=-\rmd_D\lambda_\text{taut}.
	\end{equation}
\end{remark}

Fixing the notation for the following, we notice that~\eqref{eq:def:DJ_spray} and~\eqref{eq:def:canonical_contact_structure} fit in the larger commutative diagram:
\begin{center}
	\begin{tikzcd}
		p^\ast  L \arrow[r, swap, "P_J" ]\arrow[d]\arrow[rr, "P", bend left] & \pi^\ast  L \arrow [r, swap, "\Pi"] \arrow[d]& L\arrow[d] \\
		\mathcal{L} \arrow[r, "\pr_J"]\arrow[rr, swap, "p", bend right] & J^1L\arrow[r,"\pi"] & M
	\end{tikzcd}.
\end{center}

As stated, this section aims at proving the following existence theorem for self-dual pairs of Dirac--Jacobi manifolds, and now we have all the necessary tools for its proof.

\begin{theorem}
	\label{thm:existence_self-dual_pair}
	Let $(M,L,\calL)$ be a Dirac--Jacobi manifold.
	For any associated Dirac-Jacobi spray $\Sigma$, there is a neighborhood $U$ of $M$ in $\calL$ such that the following is a well-defined $(p^\ast L)|_U$-valued closed Atiyah $2$-form
	\begin{equation}
		\label{eq:thm:existence_self-dual_pair:2-form}
		\varpi=\int_0^1 \left.\Phi^{\Sigma}_{t}\right|_U^\ast P_J^\ast\varpi_\text{can} \D t
	\end{equation}
	and, together with $S:=P\colon (p^\ast L)|_U \to L$ and $T:=P\circ \Phi_{1}^{\Sigma}\colon (p^\ast L)|_U\to L$, they form the self-dual pair
	\begin{equation}
		\label{eq:thm:existence_self-dual_pair:SDP}
		\begin{tikzcd}
			(M,L,\calL)&(U,(p^\ast L)|_U,\operatorname{Gr}\varpi)\arrow[l, swap, "S"]\arrow[r, "T"]&(M,L,\calL^\mathrm{opp}).
		\end{tikzcd}
	\end{equation}
	Additionally, if $\calL=\operatorname{Gr}\calJ$ for some (unique) Jacobi structure $\calJ\in\calD^2L$, then the closed $2$-form~\eqref{eq:thm:existence_self-dual_pair:2-form} is non-degenerate and~\eqref{eq:thm:existence_self-dual_pair:SDP} is a full contact dual pair. 
\end{theorem}


\begin{proof}
	Since $\Sigma|_M=0$ by Lemma~\ref{Lem: PropDJS}, there is a neighbourhood $U$ of $M$ in $\calL$ such that $\Phi_\epsilon^\Sigma|_U:(p^\ast L)|_U\to p^\ast L$ is defined for $\epsilon\in[-1,1]$.
	So, Equation~\eqref{eq:thm:existence_self-dual_pair:2-form} gives a well-defined closed $(p^\ast L)|_U$-valued Atiyah $2$-form $\varpi$ and $S,T\colon(p^\ast L)|_U\to L$ are well-defined regular line bundle morphisms.
	Moreover, we claim that
	\begin{equation*}
		(\ker DS\cap \ker\varpi^\flat\cap \ker DT)|_M=0.
	\end{equation*}
	Indeed, for arbitrary $x\in M$, and $(\delta_1,e_1),(\delta_2,e_2)\in(Dp^\ast L)_x\simeq D_xL\oplus\calL_x$, a straightforward computation in local adapted coordinates shows first that
	\begin{align*}
		(P_J^\ast\varpi_\text{can})_x((\delta_1,e_1)), (\delta_2,e_2))=(\pr_Je_2)\delta_1-(\pr_Je_1)\delta_2,
	\end{align*}
	and then, using point~\ref{enumitem:Lem: PropDJS:2} in Lemma~\ref{Lem: PropDJS} and Equation~\eqref{eq:thm:existence_self-dual_pair:2-form}, one also gets that
	\begin{equation}
		\label{eq:proof:thm:existence_self-dual_pair:1a}
		\varpi_x((\delta_1,e_1),(\delta_2,e_2))= 
		(\pr_Je_2)(\delta_1+\frac{1}{2}\pr_De_1)-(\pr_Je_1)(\delta_2+\frac{1}{2}\pr_De_2).
	\end{equation}
	Moreover, using again point~\ref{enumitem:Lem: PropDJS:2} in Lemma~\ref{Lem: PropDJS} and the definition of $S$ and $T$, one can easily compute
	\begin{equation}
		\label{eq:proof:thm:existence_self-dual_pair:1b}
		(D_xS)(\delta_1,e_1)=(D_xP)(\delta_1,e_1)=\delta_1 \quad\text{ and }\quad (D_xT)(\delta_1,e_1)
		=(D_xP)(D_x\Phi^{\Sigma}_1)(\delta_1,e_1)=\delta_1+\pr_De_1.
	\end{equation}
	Hence, Equations~\eqref{eq:proof:thm:existence_self-dual_pair:1a} and~\eqref{eq:proof:thm:existence_self-dual_pair:1b} imply that if $(\delta_1,e_1)$ belongs to $\ker D_xS\cap\ker\varpi_x\cap\ker D_xT$ then $\delta_1=e_1=0$, and so our claim is proven.
	Consequently, since $\rank(\ker DS\cap \ker\varpi^\flat\cap \ker DT)_u$ is an upper semicontinuous function of $u\in U$, we can assume, up to shrink $U$ around $M$ if necessary, that
	\begin{equation}
		\label{eq:proof:thm:existence_self-dual_pair:1}
		\ker DS\cap \ker\varpi^\flat\cap \ker DT=0.
	\end{equation}

	From Lemma~\ref{Lem: PropDJS}, the definition~\eqref{eq:canonical_symplectic_Atiyah_form} of $\varpi_\text{can}$ and the tautological property~\eqref{eq:tautological_property} of $\lambda_\text{taut}$, we get that
	\begin{align*}
	(D_uP)(\Sigma_u)=\pr_D u\qquad \text{and}\qquad(P_J^\ast\lambda_\text{taut})_u=(D_uP)^\ast(\pr_Ju),
	\end{align*}
	for any $u\in\calL$.
	This shows that $(\Sigma,P_J^\ast\lambda_\text{taut})$ is a section of $P^!\calL\to\calL$ so that, in particular,
	\begin{equation*}
		(\Sigma,P_J^\ast\lambda_\text{taut})|_U\in\Gamma(S^!\calL),
	\end{equation*}
	Consequently, by Remark~\ref{rem:infinitesimal_CJ_aut_LieAlgebra}, the infinitesimal automorphism $\ldsb(\Sigma,P_J^\ast\lambda_\text{taut})|_U,-\rdsb$ generates a 1-parameter group of local Courant--Jacobi automorphisms of $\bbD(p^\ast L)|_U$ which is given by
	\begin{equation}
		\label{eq:proof:thm:existence_self-dual_pair:2_pre}
		\exp\left(\int_0^\epsilon(\Phi_{-t}^\Sigma)^\ast \rmd_D P_J^\ast\lambda_\text{taut}\rmd t\right)\circ \bbD\Phi_\epsilon^\Sigma
	\end{equation}
	and preserves $S^!\calL$ whenever it exists.
	Further, since by assumption $\Phi_\epsilon^\Sigma|_U:(p^\ast L)|_U\to p^\ast L$ is defined for $\epsilon\in[-1,1]$, the map~\eqref{eq:proof:thm:existence_self-dual_pair:2_pre} induces a global automorphism of $S^!\calL$ for all $\epsilon\in[-1,1]$, and so, in particular, for $\epsilon=-1$.
	A straightforward computation shows that 
	\begin{equation*}
		\int_0^{-1}(\Phi_{-\tau}^\Sigma)^\ast \rmd_DP_J^\ast\lambda_\text{taut}\rmd\tau=-\int_0^1(\Phi_{\tau}^\Sigma)^\ast \rmd_DP_J^\ast\lambda_\text{taut}\rmd\tau=\int_0^1(\Phi_{\tau}^\Sigma)^\ast P_J^\ast \varpi_\text{can}\rmd\tau=\varpi,
	\end{equation*}
	where we used the naturality of $\rmd_D$.
	So, the map~\eqref{eq:proof:thm:existence_self-dual_pair:2_pre} preserving $S^!\calL$ for $\epsilon=-1$ can be rewritten as
\begin{equation}
	\label{eq:proof:thm:existence_self-dual_pair:2}
	S^!\calL=(\mathbb D\Phi_{-1}^\Sigma(S^!\calL))^\varpi	=((\Phi_{1}^\Sigma)^!S^!\calL)^\varpi=(T^!\calL)^\varpi.
\end{equation}
	Finally, since $\dim U=\dim\calL=2\dim M+1$, Equations~\eqref{eq:proof:thm:existence_self-dual_pair:1} and~\eqref{eq:proof:thm:existence_self-dual_pair:2} imply that~\eqref{eq:thm:existence_self-dual_pair:SDP} is a dual pair by Proposition~\ref{prop:characterizations_WDPs_I}.
	Clearly, in view of Remark~\ref{rem:full_contact_dual_pairs}, if $\calL=\operatorname{Gr}\calJ$ for some (unique) Jacobi structure $\calJ\in\calD^2L$, then the closed $2$-form~\eqref{eq:thm:existence_self-dual_pair:2-form} is non-degenerate and~\eqref{eq:thm:existence_self-dual_pair:SDP} is a full contact dual pair. 
\end{proof}

%

\subsection{Application: Normal Forms around Dirac--Jacobi Transversals}

This Section makes use of the existence of self-dual pairs (Theorem~\ref{thm:existence_self-dual_pair}) to prove the normal form theorem around Dirac--Jacobi transversals.
Notice that this theorem was proved by the first author in~\cite{NorForJac}, with different tools, generalizing to the Dirac--Jacobi manifolds a similar result for Dirac manifolds~\cite{FM2017}.

We start recalling how to construct the \emph{local normal form} of a Dirac--Jacobi manifold around a Dirac--Jacobi transversal.
Let $(M,L,\calL)$ be a Dirac-Jacobi manifold and $I\colon L_N\to L$ be a transversal for $(M,L,\calL)$ covering an embedding $\iota\colon N\hookrightarrow M$.
Let us consider $q\colon\nu_N=(TM)|_N/TN\longrightarrow N$, the normal bundle to $N$ in $M$, and let $L_{\nu_N}:=q^\ast L_N\equiv\nu_N\times_NL_N\to\nu_N$ be the pull-back line bundle given by the diagram
\begin{equation}
	\begin{tikzcd}
		L_{\nu_N} \arrow[r, "Q"]\arrow[d] & L_N\arrow[d]\\
		\nu_N\arrow[r, swap, "q"]  & N
	\end{tikzcd}.
\end{equation}
Now, since the regular LB morphism $Q$ covers the surjective submersion $q$ and $I\colon L_N\to L$ is transversal to $(M,L,\calL)$, Corollary~\ref{Cor: TrnsMap} assures that the Lagrangian family $Q^!I^!\calL\subset\bbD L_{\nu_N}$ is smooth and so a Dirac--Jacobi structure.
This justifies the following definition.

\begin{definition}
	\label{def:linear_model_around_DJ_transversal}
	The \emph{local normal form} of a Dirac--Jacobi manifold $(M,L,\calL)$ around a Dirac--Jacobi transversal $I\colon L_N\to L$ is the Dirac--Jacobi manifold given by
	\begin{equation*}
		(\nu_N,L_{\nu_N},\calL_{\nu_N}:=Q^!I^!\calL).
	\end{equation*}
\end{definition}

The Normal Form Theorem~\cite{NorForJac} states that, around a Dirac--Jacobi transversal, a Dirac--Jacobi manifold is isomorphic to its local normal form.
The proof of its following variation relies on the existence of self-dual pairs (Theorem~\ref{thm:existence_self-dual_pair}) and the possibility of pulling back weak dual pairs along transversals (Proposition~\ref{prop:TransversePullBackWDPs}).

\begin{theorem}
	\label{theor:NormalForm_DJtransversals}
	Let $(M,L,\calL)$ be a Dirac--Jacobi manifold and $I\colon L_N\to L$ be a transversal for $(M,L,\calL)$ covering an embedding $\iota\colon N\hookrightarrow M$.
	Then there exist a neighborhood $U$ of $N$ in $\nu_N$, a regular line bundle morphism  $\Psi\colon L_{\nu_N}|_U\to L $, covering a local diffeomorphism $\psi:U\to M$ acting like the identity on $N$, and a closed $L_{\nu_N}$-valued Atiyah 2-form $B\in \Omega^2_D(L_{\nu_N})$, such that 
	\begin{align*}
		\calL_{\nu_N}|_U=(\Psi^!\calL)^B.
	\end{align*}
\end{theorem}

\begin{proof}
	Fix a Dirac--Jacobi spray $\Sigma$ of the Dirac--Jacobi manifold $(M,L,\calL)$.
	By Theorem~\ref{thm:existence_self-dual_pair}, the choice of $\Sigma$ determines a self-dual pair for $(M,L,\calL)$ 
	\begin{equation*}
		\begin{tikzcd}
			(M,L,\calL)&&(P,\ell,\operatorname{Gr}\varpi)\arrow[ll, swap, "S"]\arrow[rr, "T"]&&(M,L,\calL^\mathrm{opp}).
		\end{tikzcd}
	\end{equation*} 
	The Dirac--Jacobi transversality condition allows to pull-back the latter along $\Phi_0=I:L_N\to L$ and $\Phi_1=\id_L:L\to L$, as in Proposition~\ref{prop:TransversePullBackWDPs}, obtaining the following dual pair 
	\begin{equation}
		\begin{tikzcd}
			\label{eq:proof:theor:NormalForm_DJtransversals:a}
			(N,L_N,I^!\calL)&&(P_N,\ell_N,\operatorname{Gr}\varpi_N)\arrow[ll, swap, "S_N"]\arrow[rr, "T_N"]&&(M,L,\calL^\text{opp}),
		\end{tikzcd}
	\end{equation}
	where 
	\begin{align*}
		P_N=s^{-1}(N), \ \ell_N=\ell|_{P_N}, \ \varpi_N=\varpi|_{P_N}, \ S_N=S|_{P_N}\ \text{and}\ T_N=T|_{P_N}.
	\end{align*}
	Note that, by construction of the self-dual pair, $s_N$ is a surjective submersion while $t_N$ is a submersion whose image contains $N$.
	So, up to replace $M$ in~\eqref{eq:proof:theor:NormalForm_DJtransversals:a} by the neighborhood $t_N(P_N)$ of $N$ in $M$, we get that~\eqref{eq:proof:theor:NormalForm_DJtransversals:a} is a dual pair by Proposition~\ref{prop:TransversePullBackWDPs}. 
	
	The Dirac--Jacobi transversality allows to construct the short exact sequence of VB morphisms over $\id_N$
	\begin{equation}
		\label{eq:proof:theor:NormalForm_DJtransversals:b}
		\begin{tikzcd}
			0\arrow[r] & i^! \calL\arrow[r]& \calL|_N\arrow[r] & \nu_N\arrow[r] & 0
		\end{tikzcd}	
	\end{equation}
	where $i^!\calL\longrightarrow\calL|_N,\ (\delta,(DI)^\ast\alpha)\longmapsto((DI)\delta,\alpha),$ and $\calL|_N\longrightarrow\nu_N,\ u\longmapsto\sigma\pr_Du+TN$.
	The latter lifts canonically to the following commutative diagram 
	\begin{equation}
		\label{eq:proof:theor:NormalForm_DJtransversals:c}
		\begin{tikzcd}
			K \arrow[r]\arrow[d]& (p^\ast L)|_{\calL|_N} \arrow[r]\arrow[d] & L_{\nu_N}\arrow[d]\\
			i^\ast \calL\arrow[r]& \calL|_N\arrow[r] & \nu_N
		\end{tikzcd}
	\end{equation}
	where $K$ is the suitable pull-back line bundle and  the upper row consists of regular LB morphisms. 
	Let us choose a VB morphism $\varphi\colon \nu_N\to \calL|_N$ over $\id_N$ which splits the short exact sequence~\eqref{eq:proof:theor:NormalForm_DJtransversals:b}.
	Then the latter admits a canonical lift to a regular LB morphism $\Phi\colon L_{\nu_N} \to (p^\ast L)|_{\calL|_N}$ and we immediately get that
	\begin{equation*}
		S_N\circ\Phi=Q.
	\end{equation*}
	Note that $\Phi$ is an instance of a \emph{fat tubular neighborhood} of $L_N\to N$ in $L\to M$ (cf.~\cite[Definition~3.10]{le2018deformations}).
	
	By Lemma~\ref{Lem: PropDJS} and noticing that the canonical splitting $(D(p^\ast L))|_M\simeq\calL\oplus DL$ induces the canonical splitting $(D\ell_N)|_N=(D(p^\ast L_N))|_N\simeq\calL|_N\oplus DL_N$, we get that
	\begin{align*}
		(DT_N)|_N\colon DL_N\oplus \calL|_N\longrightarrow DL_N,\quad (\delta, e)\longmapsto \delta+\pr_D(e).
	\end{align*}
	So, there exists a neighborhood $U$ of $N$ in $\nu_N$ such that the following is a well-defined regular LB morphism
	\begin{align*}
		\Psi:=T_N\circ\Phi|_U\colon L_{\nu_N}|_U\longrightarrow L
	\end{align*}
	extending $I\colon L_N\to L$ and covering an embedding $\psi:=t_N\circ\varphi|_U\colon U \rightarrow M$.
	Now we can compute
	\begin{equation*}
		\calL_{\nu_N}|_U=Q|_U^!I^!\calL=\Phi|_U^!S_N^!I^!\calL=\Phi|_U^!((T_N^!\calL)^{\varpi_N})=(\Psi^!\calL)^{\Phi^\ast \varpi_N},
	\end{equation*}
	which proves that $\calL$ is locally isomorphic, up to a gauge transform, to its linear model $\calL_{\nu_N}$ around $N$.
\end{proof}

\begin{remark}
Note that this Normal Form Theorem is not directly applicable to Jacobi manifolds.
Indeed, because of the presence of gauge transformations~\eqref{eq:gauge_transformation}, not every isomorphism of Dirac--Jacobi manifolds is also an isomorphism of Jacobi manifolds.
Nevertheless, it can be modified in such a way that it will give a Normal Form Theorem for Jacobi manifolds.
This is carried out in~\cite[Theorems 5.4 and 5.9]{SplittingThmJac}. 
\end{remark}

\section{Characteristic Leaf Correspondence}
\label{sec:characteristic_leaf_correspondence}

After having proved in the previous sections that the property of fitting in a weak dual pair defines an equivalence relation for Dirac--Jacobi manifolds, it is natural to ask what, if any, are the implications on the geometry of two Dirac--Jacobi manifolds that are due to the existence of a weak dual pair connecting them.
Indeed, for any weak dual pair, the local structure of its legs is very closely related through the \emph{characteristic leaf correspondence}, as we are going to show in this section.

\subsection{The Characteristic Foliation of Dirac-Jacobi Manifolds}

As a necessary preliminary for discussing the characteristic leaf correspondence, we give first a glimpse at the local structure of Dirac--Jacobi manifolds~\cite{SplittingThmJac}.
In doing this, we mainly follow~\cite{DirJacBun}.

Let $(M,L,\calL)$ be a Dirac--Jacobi manifold $\calL\subset\bbD L$.
The vector bundle $\calL\to M$ is endowed with the Lie algebroid structure $([-,-]_\calL,\rho_\calL)$ and the Lie algebroid representation $\nabla^\calL$ on $L$ which are defined by
\begin{equation*}
	[u,v]_\calL=\ldsb u,v\rdsb,\quad\rho_\calL(u)=(\sigma\circ\pr_D)(u),\quad\nabla^\calL_u\lambda=\pr_D(u)\lambda,
\end{equation*}
for all $u,v\in\Gamma(\calL)$ and $\lambda\in\Gamma(L)$.
As a consequence, the image of the anchor map $\rho_\calL$
\begin{align*}
	K_\mathcal{L}=\sigma(\pr_D\calL)\subset TM
\end{align*}
gives a singular distribution on $M$ which is integrable á la Stefan--Sussman.
The singular foliation integrating $\calK_\calL$, denoted by $\calF_\calL$, is called the \emph{characteristic foliation of the Dirac--Jacobi manifold $(M,L,\calL)$}, and its leaves are called the \emph{characteristic leaves of $(M,L,\calL)$} (cf.~\cite[Sec.~5]{DirJacBun}).
In particular, if the associated algebroid is transitive, i.e.~$\calK_\calL=TM$, one says that the Dirac--Jacobi structure $\calL$ is \emph{transitive}.

Given a Dirac-Jacobi manifold $(M,L,\calL)$, for any characteristic leaf $\iota:S\hookrightarrow M$, we define the pull-back line bundle $L_S:=\iota^\ast L=S\times_ML\to S$ and the regular LB morphism $I:L_S\to L,\ (x,\lambda_{\iota{x}})\mapsto\lambda_{\iota{x}}$
The latter allows us to consider the backward transform
of the Dirac-Jacobi structure $\mathcal{L}$.
\begin{lemma}
	\label{Lem: SmoCharFol}
	Let $(M,L,\calL)$ be a Dirac--Jacobi manifold and let $\iota\colon S\hookrightarrow M $ be one of its characteristic leaves.
	Then $\calL_S:=I^!\calL$ is the (unique) transitive Dirac--Jacobi structure on $L_S\to S$ such that
	\begin{equation*}
		I_S:(S,L_S,\calL_S)\to (M,L,\calL)
	\end{equation*}
	is a backward Dirac--Jacobi map.
	Moreover $DI:DL_S\to DL$ identifies $\pr_D\calL_S$ with $pr_D\calL|_S$.
\end{lemma}  

\begin{proof}
	We have that the $\ker(DI^\ast)\subset(\image(DI))^\circ$.
	Moreover, since $\calL=\calL^\perp$, we get $J^1L\cap \calL=(\pr_D(\calL))^\circ$.
	Hence we can compute
	\begin{align*}
		\ker(DI^\ast)\cap \calL\at{S}=(\image(DI))^\circ\cap (\pr_D(\mathcal{L}_S))^\perp=(\image(DI)+\pr_D\mathcal{L}\at{S})^\circ=(\image(DI))^\circ,
	\end{align*}
	where, in the last step, we used the fact that $\sigma(\pr_D\mathcal{L}\at{S})=T\iota(TS)$ and so $\pr_D\calL\at{S}\subset\image(DI)$.
	Therefore, $\ker(DI^\ast)\cap \calL\at{S}$ has constant rank and $I^!\calL\subset\bbD L_S$ is a Dirac--Jacobi structure by Theorem~\ref{Thm: CleanInt}. 
\end{proof}
Since the symbol map has a one dimensional kernel, we can distinguish two kinds of leaves: 

\begin{definition}
	Let $(M,L,\calL)$ be a Dirac--Jacobi manifold and $\iota:S\hookrightarrow M$ be a leaf.
	Then 
	$S$ is said to be 
	\begin{enumerate}[label={\arabic*)}]
		\item \emph{pre-contact}, if $\rank(\pr_D(I^!\calL))=\dim S+1$.
		\item \emph{locally conformal pre-symplectic} (or \emph{lcps}), if $\rank(\pr_D(I^!\calL))=\dim S$.
	\end{enumerate}
\end{definition}

In Dirac geometry, the characteristic leaves have an induced pre-symplectic form, which is induced via the backward transform of the Dirac structure along the inclusion.
In the case of Dirac-Jacobi bundles it is a bit different, since they admit two different kind of characteristic leaves which have different induced structures.

\begin{lemma}
	Let $(M,L,\calL)$ be a Dirac--Jacobi manifold and let $\iota\colon S\hookrightarrow M $ be a characteristic leaf.
	Then the rank of $(\pr_D\calL)$ is constant along $S$.
	Moreover, there are only two possible cases for $S$:
	\begin{itemize}
		\item either $(DI)DL_S=\langle\mathbbm{1}\rangle|_S\oplus\pr_D\calL|_S$ and, in this case, we will say that $S$ is a \emph{plcs leaf},
		\item or $(DI)DL_S=\pr_D\calL|_S$ and, in this case, we will say that $S$ is a \emph{precontact leaf}.
	\end{itemize}
\end{lemma}

\begin{proof}
	The proof of the first part can be found in~\cite[Lemma~5.1]{DirJacBun}, then the second part is a straightforward consequence in view of the fact that $\langle\mathbbm{1}\rangle$ is the rank $1$ kernel of $\sigma:DL\to TM$.
\end{proof}

Note that this distinction is the first and probably one of the most significant conceptual differences between Dirac-Jacobi structures and Dirac structures.
Below we explain the names of the different leaves.
Let us start with precontact leaves, which are very similar to presymplectic leaves in Dirac geometry.


\begin{lemma}[{\cite[Prop.~5.4(1)]{DirJacBun}}]
	\label{Lem: PreCon}
	For any line bundle $L\to M$, the relation $\calL=\operatorname{Gr}(\varpi)$ establishes a canonical one-to-one correspondence between:
	\begin{enumerate}
		\item (transitive) Dirac--Jacobi structures $\calL\subset\bbD L$ such that $\pr_D\calL=DL$, and
		\item $L$-valued presymplectic Atiyah forms $\varpi$.
	\end{enumerate}
\end{lemma}


We can apply  Lemma \ref{Lem: PreCon} directly  to the case of pre-contact leaves, since (by their very definition) they are equipped with Dirac-Jacobi structures of this kind.
So they inherit a precontact structure for the ambient Dirac structure.
Let us now turn to locally conformal pre-symplectic leaves.

\begin{lemma}[{\cite[Prop.~5.4(2)]{DirJacBun}}]
	\label{Lem: PreSymLea}
	For any line bundle $L\to M$, the following relation
	\begin{equation*}
		\calL=\{(\nabla_\xi,\sigma^\ast(\omega^\flat \xi)+\alpha)\mid (\xi,\alpha)\in TM\oplus(\image\nabla)^\circ\}
	\end{equation*}
	establishes a canonical one-to-one correspondence between:
	\begin{enumerate}
		\item (transitive) Dirac--Jacobi structures $\calL\subset\bbD L$ such that $\langle\mathbbm{1}\rangle\oplus\pr_D\calL=DL$, and
		\item locally conformal presymplectic structures $(\omega,\nabla)$ on $L\to M$.
	\end{enumerate}
\end{lemma}

A lcps leaf is clearly equipped with one of these Dirac--Jacobi structures and so it inherits a lcps structure from the ambient Dirac--Jacobi structure.
Let us summarize the previous discussion in the following

\begin{corollary}
	\label{Cor: CharFolDJ}
	Let $(M,L,\calL)$ be a Dirac--Jacobi manifold and let $\iota\colon S\hookrightarrow M $ be a characteristic leaf. 
	\begin{enumerate}[label={\arabic*)}]
		\item If $S$ is a precontact leaf, there is a unique  presymplectic Atiyah form $\varpi\in \Omega^2_D(L_S)$, s.t.~$I^!\calL=\operatorname{Gr}(\varpi)$.
		\item
		If $S$ is an lcps leaf, there exists a unique lcps structure $(\nabla,\omega)$ on $L_S$ such that $I^!\calL=\calL_{(\nabla,\omega)}$.
	\end{enumerate}	 
\end{corollary}

\subsection{The Characteristic Leaf Correspondence Theorem(s)}

We subdivide the Characteristic Leaf Correspondence Theorme in three separate parts.
First, for any weak dual pair with connected fibers of the underlying maps,  we prove that there exists a one-to-one correspondence between the characteristic leaves of its legs that indentifies their leaf spaces and preserves the type of the leaves (Theorem~\ref{theor:CharacteristicLeafCorrespondence_I}).
Second, we discuss the relation between the transitive Dirac--Jacobi structures (i.e.~precontact or lcps) inherited by corresponding characteristic leaves (Theorem~\ref{theor:CharacteristicLeafCorrespondence_II}).
Finally, we show that corresponding characteristic leaves have the same transverse Dirac--Jacobi structure (Theorem~\ref{theor:LeafCorrespondence:transverseDJ}).

Let us start with a first lemma concerning the relation between the characteristic foliations of a Dirac--Jacobi structure and its backward transform.

\begin{lemma}
	\label{lem:LeavesCorrespondence_I}
	Let $\Phi\colon (M,L,\calL)\to (M^\prime,L^\prime,\calL^\prime)$ be a backward Dirac--Jacobi map covering a surjective submersion $\varphi\colon M\to M^\prime$ with connected fibers.
	Then the relation $\calS=\varphi^{-1}(\calS^\prime)$ establishes a 1-1 correspondence
	\begin{equation*}
		\text{characteristic leaves $\calS$ of $(M,L,\calL)$}\rightleftharpoons\text{characteristic leaves $\calS^\prime$ of $(M^\prime,L^\prime,\calL^\prime)$}.
	\end{equation*}
	Moreover, this correspondence respects the type of the leaves, i.e.~the following conditions are equaivalent
	\begin{itemize}
		\item $\calS$ is a pre-contact (resp.~locally conformal pre-symplectic) leaf of $(M,L,\calL)$,
		\item $\calS^\prime$ is a pre-contact (resp.~locally conformal pre-symplectic) leaf of $(M^\prime,L^\prime,\calL^\prime)$,.
	\end{itemize}  
\end{lemma}

\begin{proof}
	Since $\Phi$ is a backward Dirac--Jacobi map, i.e.~$\calL=\Phi^!\calL^\prime$, and covers a submersion $\varphi:M\to M^\prime$, so that the VB morphism $D\Phi:DL\to DL^\prime$ is fiberwise surjective, one easily compute that, for all $x\in M$,
	\begin{equation}
		\label{eq:proof:lem:LeavesCorrespondence_I:a}
		\pr_D\calL_x=\pr_D\{(\delta,(D_x\Phi)\alpha^\prime)\mid((D_x\Phi)\delta,\alpha^\prime)\in\calL^\prime_{\varphi(x)}\}=(D_x\Phi)^{-1}(\pr_D\calL^\prime_{\varphi(x)}).
	\end{equation}
	Since $(D_x\Phi)^{-1}\mathbbm{1}_{\varphi(x)}=\mathbbm{1}_x$, as a first consequence of the latter one gets that, for all $x\in M$,
	\begin{equation}
		\label{eq:proof:lem:LeavesCorrespondence_I:b}
		\mathbbm{1}_x\in\pr_D\calL_x\Longleftrightarrow\mathbbm{1}_{\varphi(x)}\in\pr_D\calL^\prime_{\varphi(x)}
	\end{equation}
	Denote by $\calF$ and $\calF^\prime$ the characteristic foliations of respectively $(M,L,\calL)$ and $(M^\prime,L^\prime,\calL^\prime)$.
	Then, applying the symbol map to Equation~\eqref{eq:proof:lem:LeavesCorrespondence_I:a}, one gets that $T\calF$ is the \emph{pull-back} of $T\calF^\prime$ along $\varphi$, i.e.~for all $x\in M$
	\begin{equation*}
		T_x\calF=(T_x\varphi)^{-1}T_{\varphi(x)}\calF.
	\end{equation*}
	The submersion $\varphi:M\to M^\prime$ is also surjective with connected fibers, and so Theorem E.7 and Corollary E.8 in~\cite{blaom2001geometric} imply that the map $\calS^\prime\longmapsto\varphi^{-1}(\calS^\prime)$ defines a bijection from the leaf space $M^\prime/\calF^\prime$ to the leaf space $M/\calF$.
	Finally, Equation~\eqref{eq:proof:lem:LeavesCorrespondence_I:b} assures that this bijection preserves the type of the characteristic leaves.
\end{proof}

\begin{lemma}
	\label{lem:LeavesCorrespondence_II}
	Let us consider the following weak dual pair, with underlying maps $\begin{tikzcd}
		M_0&M\arrow[l, swap, "\varphi_0"]\arrow[r, "\varphi_1"]&M_1
	\end{tikzcd}$,
	\begin{equation*}
		\begin{tikzcd}
			(M_0,L_0,\calL_0)&(M,L,\operatorname{Gr}\varpi)\arrow[l, swap, "\Phi_0"]\arrow[r, "\Phi_1"]&(M_1,L_1,\calL_1^\mathrm{opp}).
		\end{tikzcd}
	\end{equation*}
	Then the following Lagrangian families are all smooth, and so Dirac--Jacobi structures on $L\to M$,
	\begin{equation*}
		\Phi_0^!\calL_0,\quad\Phi_1^!\calL_1^\mathrm{opp},\quad(\Phi_0^!\calL_0)\star(\Phi_1^!\calL_1^\mathrm{opp}).
	\end{equation*}
	Moreover, their anchor maps have as image the same involutive subbundle of $DL$ given by
	\begin{equation}
		\label{eq:lem:LeavesCorrespondence_II}
		\pr_D(\Phi_0^!\calL_0)=\pr_D(\Phi_1^!\calL_1^\mathrm{opp})=\pr_D((\Phi_0^!\calL_0)\star(\Phi_1^!\calL_1^\mathrm{opp}))=\ker(D\Phi_0)+\ker(D\Phi_1).
	\end{equation}
	Consequently, they determine the same characteristic foliation $\calF$ of $M$ given by $T\calF=\ker T\varphi_0+\ker T\varphi_1$.
\end{lemma}

\begin{proof}
	The regular LB morphism $\Phi_i:L\to L_i$ covers a submersion, so $\ker(D\Phi_i)^\ast=0$, and $\Phi_i^!\calL_i\subset\bbD L_i$ is a Dirac--Jacobi structure, for $i=0,1$, by Theorem~\ref{Thm: CleanInt}.
	Further, from Corollary~\ref{cor:prop:characterizations_WDPs_I} we know that
	\begin{equation}
		\label{eq:proof:lem:LeavesCorrespondence_II:a}
		\Phi_0^!\calL_0=\ker D\Phi_0+(\ker D\Phi_1)^\varpi\ \text{and}\ \Phi_1^!\calL_1^\text{opp}=\ker D\Phi_1+(\ker D\Phi_0)^\varpi,
	\end{equation}
	So, using Remark~\ref{rem:tiny_technical_remark}, one can prove the following
	\begin{equation}
		\label{eq:proof:lem:LeavesCorrespondence_II:b}
		(\Phi_0^!\calL_0)\star(\Phi_1^!\calL_1^\mathrm{opp})=\left((\ker D\Phi_0+\ker D\Phi_1)\oplus(\ker D\Phi_0+\ker D\Phi_1)^\circ\right)^\varpi.
	\end{equation}	
	This shows that $(\Phi_0^!\calL_0)\star(\Phi_1^!\calL_1^\mathrm{opp})$ is smooth, and so a Dirac--Jacobi structure, by Lemma~\ref{Lem: Product}.
	Now the remaining part of the statement follows immediately from Equations~\ref{eq:proof:lem:LeavesCorrespondence_II:a} and~\ref{eq:proof:lem:LeavesCorrespondence_II:b}.
\end{proof}

Using Lemmas~\ref{lem:LeavesCorrespondence_I} and~\ref{lem:LeavesCorrespondence_II} we can prove the characteristic leaf correspondence theorem. 

\begin{theorem}[{\bf Characteristic Leaf Correspondence I}]
	\label{theor:CharacteristicLeafCorrespondence_I}
	Consider a weak dual pair
	\begin{equation}
		\label{eq:theor:LeafCorrespondence}
		\begin{tikzcd}
			(M_0,L_0,\calL_0)&(M,L,\operatorname{Gr}\varpi)\arrow[l, swap, "\Phi_0"]\arrow[r, "\Phi_1"]&(M_1,L_1,\calL_1^\mathrm{opp}),
		\end{tikzcd}
	\end{equation}
	whose underlying maps $\begin{tikzcd}
		M_0&M\arrow[l, swap, "\varphi_0"]\arrow[r, "\varphi_1"]&M_1
	\end{tikzcd}$ have connected fibers.
	Then the relation $\varphi_1^{-1}\calS_1=\varphi_0^{-1}\calS_0$ establishes a 1-1 correspondence
	\begin{equation*}
		\text{characteristic leaves $\calS_0$ of $(M_0,L_0,\calL_0)$}\rightleftharpoons\text{characteristic leaves $\calS_1$ of $(M_1,L_1,\calL_1)$}.
	\end{equation*}
	Moreover, this correspondence respects the type of the leaves, i.e.~the following conditions are equivalent
	\begin{itemize}
		\item $\calS_0$ is a pre-contact (resp.~locally conformal pre-symplectic) leaf of $(M_0,L_0,\calL_0)$,
		\item $\calS_1$ is a pre-contact (resp.~locally conformal pre-symplectic) leaf of $(M_1,L_1,\calL_1)$.
	\end{itemize}
\end{theorem}

\begin{proof}
	Let us first apply Lemma~\ref{lem:LeavesCorrespondence_I} to the backward Dirac--Jacobi map $\Phi_i:(M,L,\Phi_i^!\calL_i)\to(M_i,L_i,\calL_i)$, for $i=0,1$.
	So, one gets that the relation $\calS=\varphi_i^{-1}\calS_i$ establishes a canonical 1-1 correspondence between the characteristic leaves $\calS_i$ of $(M_i,L_i,\calL_i)$ and the characteristic leaves $\calS$ of $(M,L,\Phi_i^!\calL_i)$.
	Additionally, this correspondence preserves the type of the characteristic leaves.
	Further, by Lemma~\ref{lem:LeavesCorrespondence_II}, one gets that the Dirac--Jacobi structures $\Phi_i^!\calL_i$, with $i=0,1$, give rise to the same characteristic foliation of $M$ as $(\Phi_0^!\calL_0)\star(\Phi_1^!\calL_1)$, and this completes the proof.
\end{proof}

\begin{remark}
	\label{rem:LeafCorrespondence}
	From the proof of Theorem \ref{theor:CharacteristicLeafCorrespondence_I}, we can extract that, given a weak dual pair as the following
	\begin{equation*}
	\begin{tikzcd}
		(M_0,L_0,\calL_0)&(M,L,\operatorname{Gr}\varpi)\arrow[l, swap, "\Phi_0"]\arrow[r, "\Phi_1"]&(M_1,L_1,\calL_1^\mathrm{opp}),
	\end{tikzcd}
	\end{equation*}
	for $i=0,1$, the relation $\calS=\varphi_i^{-1}\calS_i$ establishes a 1-1 correspondence between the characteristic leaves $\calS$ of $(M,L,(\Phi_0^!\calL_0)\star(\Phi_1^!\calL_1^\mathrm{opp}))$ and the characteristic leaves $\calS_i$ of $(M_i,L_i,\calL_i)$.
	Exactly from the composition of these two bijections one obtains the characteristic leaf correspondence described in Theorem~\ref{theor:CharacteristicLeafCorrespondence_I}.
\end{remark}

In the following, we are going to show that, given the weak dual pair~\eqref{eq:theor:LeafCorrespondence} and two corresponding characteristic leaves, one can also relate their inherited transitive Dirac--Jacobi structures and their transverse geometries.
Let us start setting some notation.
Fix corresponding characteristic leaves $\calS_1$ of $(M_1,L_1,\calL_1)$ and $\calS_2$ of $(M_2,L_2,\calL_2)$.
Then $\calS:=\varphi_0^{-1}\calS_0=\varphi_1^{-1}\calS_1$ is a characteristic leaf of $(M,L,(\Phi_0^!\calL_0)\star(\Phi_1^1\calL_1))$ of the same type of $\calS_i$, for $i=0,1$.
Introduce the restricted line bundles $\ell:=L|_\calS\to\calS$ and $\ell_i:=L_i|_{\calS_i}\to\calS_i$, for $i=0,1$, and the regular line bundle morphisms $I:\ell\to L$ and $I_i:\ell_i\to L$ covering respectively the immersions $\iota:\calS\to M$ and $\iota_i:\calS_i\to\calS$, for $i=0,1$.
Further, by Theorem E.7 in~\cite{blaom2001geometric}, we get that $\Phi_i:L\to L_i$ induces a regular LB morphism $\Phi_i|_\calS:\ell\to\ell_i$ covering a surjective submersion $\varphi_i|_\calS:\calS\to\calS_i$, for $i=0,1$.

\begin{theorem}[{\bf Characteristic Leaf Correspondence II}]
	\label{theor:CharacteristicLeafCorrespondence_II}
	Consider a weak dual pair
	\begin{equation}
		\begin{tikzcd}
			(M_0,L_0,\calL_0)&(M,L,\operatorname{Gr}\varpi)\arrow[l, swap, "\Phi_0"]\arrow[r, "\Phi_1"]&(M_1,L_1,\calL_1^\mathrm{opp}),
		\end{tikzcd}
	\end{equation}
	whose underlying maps $\begin{tikzcd}
		M_0&M\arrow[l, swap, "\varphi_0"]\arrow[r, "\varphi_1"]&M_1
	\end{tikzcd}$ have connected fibers.
	Let $\calS_0$ and $\calS_1$ be corresponding characteristic leaves as in Theorem~\ref{theor:CharacteristicLeafCorrespondence_I}.
	Then there are only two possible cases.
	\begin{enumerate}[label={\arabic*)}]
		\item
		\label{enumitem:prop:CharLeafCorresp:1}
		$\calS_i$ is a precontact leaf with presymplectic Atiyah form $\varpi_i\in \Omega^2_D(\ell_i)$, for $i=0,1$, and
		\begin{equation}
			\label{eq:prop:CharLeafCorresp:1}
			I^\ast\varpi=\Phi_0|_{\calS_0}^\ast\varpi_0 - \Phi_1|_{\calS_1}^\ast\varpi_1\in\Omega^2_D(\ell).
		\end{equation}
		\item
		\label{enumitem:prop:CharLeafCorresp:2}
		$\calS_i$ is a lcps leaf with lcps structure $(\omega_i,\nabla^i)$ on $\ell_i\to\calS_i$, for $i=0,1$, then the following holds
		\begin{equation}
			\label{eq:prop:CharLeafCorresp:2}
			\rmd_\nabla(I^\ast\theta)=\Phi_0|_{\calS_0}^\ast\varpi_0-\Phi_1|_{\calS_1}^\ast\varpi_1
		\end{equation}
		where $\theta\in\Omega^1(M;L)$ is the $L$-valued $1$-form on $M$ given by $\theta\circ\sigma=\iota_{\mathbbm{1}}\varpi$, and $\nabla:T\calS\to D\ell$ is the unique $T\calS$-connection on $\ell$ such that $(D\Phi_i)\circ\nabla=\nabla^i\circ (T\varphi_i)$, for $=0,1$.
	\end{enumerate}
\end{theorem}

\begin{proof}
	First of all, by Theorem~\ref{theor:CharacteristicLeafCorrespondence_I} we already know that there are only two possible cases:
	\begin{enumerate}[label={$\arabic*^\prime$)}]
		\item
		\label{enumitem:proof:prop:CharLeafCorresp:1}
		$\calS_i$ is a pre-contact leaf with presymplectic Atiyah form $\varpi_i\in \Omega^2_D(\ell_i)$, for $i=0,1$,
		\item
		\label{enumitem:proof:prop:CharLeafCorresp:2}
		$\calS_i$ is a lcps leaf with lcps structure $(\omega_i,\nabla^i)$ on $\ell_i\to\calS_i$, for $i=0,1$.
	\end{enumerate}
	Let us assume~\ref{enumitem:proof:prop:CharLeafCorresp:1}.
	By definition of a pre-contact leaf we have that $I_i^!\calL_i=\operatorname{Gr}\varpi_i$, for $i=0,1$.
	Therefore, using $\Phi_i\circ I=I_i\circ\Phi_i|_\calS$, for $i=0,1$, we can compute
	\begin{equation}
		\label{eq:proof:prop:CharLeafCorresp:1a}
		\begin{aligned}
			I^!((\Phi_0^!\calL_0)\star(\Phi_1^!\calL_1^\textrm{opp}))&=(I^!\Phi_0^!\calL_0)\star(I^!\Phi_1^!\calL_1^\textrm{opp})
			\\
			&=\operatorname{Gr}(\Phi_0|_\calS^\ast\varpi_0)\star\operatorname{Gr}(\Phi_1|_\calS^\ast\varpi_1)^\textrm{opp}=\operatorname{Gr}(\Phi_0|_\calS^\ast\varpi_0-\Phi_1|_\calS^\ast\varpi_1).
		\end{aligned}
	\end{equation}
	Moreover, $\calS=\varphi_i^{-1}\calS_i$ is a pre-contact leaf of $(M,L,(\Phi_0^!\calL_0)\star(\Phi_1^!\calL_1)^\textrm{opp})$, and so by Equation~\ref{eq:lem:LeavesCorrespondence_II} we get
	\begin{equation*}
		D\ell= \pr_D\left.\left((\Phi_0^!\calL_0)\oplus(\Phi_1^!\calL_1)\right)\right|_\calS=(\ker D\Phi_0+\ker D\Phi_1)|_\calS.
	\end{equation*}
	Consequently, using also Equation~\eqref{eq:proof:lem:LeavesCorrespondence_II:b}, we can easily compute
	\begin{equation}
		\label{eq:proof:prop:CharLeafCorresp:1b}
		I^!((\Phi_0^!\calL_0)\star(\Phi_1^!\calL_1^\textrm{opp}))=I^!\left(\left((\ker D\Phi_0+\ker D\Phi_1)\oplus(\ker D\Phi_0+\ker D\Phi_1)^\circ\right)^\omega\right)=\operatorname{Gr}(I^\ast\varpi).
	\end{equation}
	Now, from the comparison of Equations~\eqref{eq:proof:prop:CharLeafCorresp:1b} and~\eqref{eq:proof:prop:CharLeafCorresp:1a}, we obtain Equation~\eqref{eq:prop:CharLeafCorresp:1}.
	
	Let us assume~\ref{enumitem:proof:prop:CharLeafCorresp:2}.
	By definition of a lcps leaf, the $T\calS_i$-connection $\nabla^i$ on $\ell_i\to\calS_i$ is determined by
	\begin{equation}
		\label{eq:proof:prop:CharLeafCorresp:2a}
		\image\{\nabla^i:T\calS_i\to D\ell_i\}=\left.(\pr_D\calL_i)\right|_\calS\subset D\ell_i,\ \ \text{for}\ i=0,1,
	\end{equation}
	so that, additionally, by Lemma~\ref{Lem: PreSymLea} the $\rmd_{\nabla^i}$-closed $\ell_i$-valued $2$-form $\omega_i$ is uniquely determined by
	\begin{equation}
		\label{eq:proof:prop:CharLeafCorresp:2b}
		I_i^!\calL_i=\{(\nabla^i_\xi,\sigma^\ast(\omega_i^\flat \xi)+\alpha)\mid (\xi,\alpha)\in T\calS_i\oplus(\image\nabla^i)^\circ\},\ \ \text{for}\ i=0,1.
	\end{equation}
	Moreover, by Remark~\ref{rem:LeafCorrespondence}, $\calS=\varphi_i^{-1}\calS_i$ is a lcps leaf of $(M,L,(\Phi_0^!\calL_0)\star(\Phi_1^!\calL_1)^\textrm{opp})$, with inherited lcps structure $(\omega,\nabla)$ on $\ell\to\calS$.
	
	First, let us recall that the $T\calS$-connection $\nabla$ on $\ell\to\calS$ is determined by
	\begin{equation}
		\label{eq:proof:prop:CharLeafCorresp:2c}
		\image\{\nabla:T\calS\to D\ell\}=\left.\left(\pr_D((\Phi_0^!\calL_0)\star(\Phi_1^!\calL_1)^\textrm{opp})\right)\right|_\calS\overset{\eqref{eq:proof:lem:LeavesCorrespondence_II:b}}{=}\left.\left(\ker D\Phi_0+\ker D\Phi_1\right)\right|_\calS\subset D\ell,
	\end{equation}
	where we also used Equation~\eqref{eq:proof:lem:LeavesCorrespondence_II:b}.
	So, by definition of a lcps leaf, one gets the direct sum decomposition
	\begin{equation}
		\label{eq:proof:prop:CharLeafCorresp:2d}
		D\ell= \langle\mathbbm{1}\rangle\oplus\pr_D\left.\left((\Phi_0^!\calL_0)\oplus(\Phi_1^!\calL_1)\right)\right|_\calS=\langle\mathbbm{1}\rangle\oplus(\ker D\Phi_0+\ker D\Phi_1)|_\calS.
	\end{equation}
	In view of Equations~\eqref{eq:proof:prop:CharLeafCorresp:2a} and~\eqref{eq:proof:prop:CharLeafCorresp:2c}, $D\Phi_i$ maps $\image\{\nabla:T\calS\to D\ell\}$ into $\image\{\nabla^i:T\calS_i\to D\ell_i\}$, for $i=0,1$.
	So, by the direct sum decomposition above, $\nabla$ is also equivalently characterized by
	\begin{equation}
		\label{eq:proof:prop:CharLeafCorresp:2e}
		(D\Phi_i)\nabla_\xi=\nabla^i_{(T\varphi_i)\xi},\ \text{for all}\ \xi\in T\calS\ \text{and}\ i=0,1.
	\end{equation}

	By Lemma~\ref{Lem: PreSymLea} and Equations~\eqref{eq:proof:prop:CharLeafCorresp:2b}--\eqref{eq:proof:prop:CharLeafCorresp:2e}, the $\rmd_\nabla$-closed $\ell$-valued $2$-form $\omega$ is uniquely determined by
	\begin{align*}
		\{(\nabla_\xi,\sigma^\ast(\omega^\flat \xi)+\alpha)\mid (\xi,\alpha)\in T\calS\oplus(\image\nabla)^\circ\}&=I^!((\Phi_0^!\calL_0)\star(\Phi_1^!\calL_1)^\textrm{opp})\\
		&=(\Phi_0|_\calS^!I_0^!\calL_0)\star(\Phi_1|_\calS^!I_1^!\calL_1^\textrm{opp})\\
		&=\{(\nabla_\xi,\sigma^\ast((\Phi_0|_{\calS_0}^\ast\omega_0-\Phi_1|_{\calS_1}^\ast\omega_1)^\flat \xi)+\alpha)\mid (\xi,\alpha)\in T\calS\oplus(\image\nabla)^\circ\},
	\end{align*}
	On the other hand, using directly Equations~\eqref{eq:proof:lem:LeavesCorrespondence_II:b} and~\eqref{eq:proof:prop:CharLeafCorresp:2c}, a straightforward computation shows that
	\begin{align*}
			I^!((\Phi_0^!\calL_0)\star(\Phi_1^!\calL_1)^\textrm{opp})
			&=I^!\left(\left((\ker D\Phi_0+\ker D\Phi_1)\oplus(\ker D\Phi_0+\ker D\Phi_1)^\circ\right)^\varpi\right)\\
			&=\{(\nabla_\xi,(DI)^\ast(\varpi^\flat\nabla_\xi)+\alpha)\mid (\xi,\alpha)\in T\calS\oplus(\image\nabla)^\circ\},
	\end{align*}
	Since $\theta\circ\sigma=\iota_\mathbbm{1}\varpi$ implies $\varpi=\rmd_D(\theta\circ\sigma)$, comparing the last two equations one gets that
	\begin{equation*}
		(\Phi_0|_{\calS_0}^\ast\omega_0-\Phi_1|_{\calS_1}^\ast\omega_1)(\xi,\zeta)=\varpi(\nabla_\xi,\nabla_\zeta)
		=\nabla_\xi(\theta(\zeta))-\nabla_\zeta(\theta(\xi))-\theta([\xi,\zeta])=\rmd_\nabla(I^\ast\theta)(\xi,\zeta),
	\end{equation*}
	for all $x\in\calS$ and $\xi,\zeta\in T_x\calS$.
	The latter shows that Equation~\eqref{eq:prop:CharLeafCorresp:2} and so it completes the proof.
\end{proof}

Let us now pass to the transverse geometry, i.e.~the structures on transversals, in fact we have seen that there is a leaf correspondence, but leaves need not to be isomorphic.
Minimal transversals on the other hand are locally isomorphic as Dirac--Jacobi manifolds.

\begin{theorem}[{\bf Characteristic Leaf Correspondence III}]
	\label{theor:LeafCorrespondence:transverseDJ}
	Consider a weak dual pair
	\begin{equation}
		\label{eq:theor:LeafCorrespondence:transverseDJ}
		\begin{tikzcd}
			(M_0,L_0,\calL_0)&(M,L,\operatorname{Gr}\varpi)\arrow[l, swap, "\Phi_0"]\arrow[r, "\Phi_1"]&(M_1,L_1,\calL_1^\mathrm{opp}),
		\end{tikzcd}
	\end{equation}
	whose underlying maps $\begin{tikzcd}
		M_0&M\arrow[l, swap, "\varphi_0"]\arrow[r, "\varphi_1"]&M_1
	\end{tikzcd}$ have connected fibers.
	Let $\calS_0$ and $\calS_1$ be corresponding characteristic leaves as in Theorem~\ref{theor:CharacteristicLeafCorrespondence_I}.
	Then the Dirac--Jacobi structure transverse to $\calS_0$ in $(M_0,L_0,\calL_0)$ is isomorphic to the Dirac--Jacobi structure transverse to $\calS_1$ in $(M_1,L_1,\calL_1)$.
\end{theorem}

\begin{proof}
Set $\calS=\varphi_1^{-1}(\calS_1)=\varphi_2^{-1}(\calS_2)$.
Fix a point $x\in\calS$ and set $x_0:=\varphi_0(x)\in\calS_0$ and $x_1:=\varphi_1(x)\in\calS_1$.
Further, choose a \emph{minimal transversal} $\calQ$ to $\calS$ in $M$ at $x$, i.e.~a submanifold $\calQ\subset M$, with $x\in\calQ$, such that \begin{equation}
	\label{eq:proof:theor:LeafCorrespondence:transverseDJ:a}
	T_xM=T_x\calS\oplus T_x\calQ.
\end{equation}
Recall from Lemma~\ref{lem:LeavesCorrespondence_II} that $(T_x\varphi_i)^{-1}(T_{x_i}\calS_i)=T_x\calS$, for $i=0,1$.
So, up to replacing $M$ and $M_i$ with suitable neighbourhoods of $x$ and $x_i$ respectively, one can assume that, for $i=1,2$,
\begin{itemize}
	\item $\calQ_i:=\varphi_i(\calQ)$ is \emph{minimal transverse} to $\calS_i$ in $M_i$ at $x_i$, i.e.~a submanifold $\calQ_i\subset M_i$, with $x_i\in\calQ_i$, s.t.
	\begin{equation}
		\label{eq:proof:theor:LeafCorrespondence:transverseDJ:b}
		T_{x_i}M_i=T_{x_i}\calS_i\oplus T_{x_i}\calQ_i.
	\end{equation}
\item the regular line bundle morphism $\Phi_i:L\to L_i$, covering $\varphi:M\to M_i$, induces a line bundle isomorphism $\Phi_{i\perp}:L|_\calQ\to L_i|_{\calQ_i}$, covering a diffeomorphism $\varphi_{i\perp}:\calQ\to \calQ_i$, such that
\begin{equation}
	\label{eq:proof:theor:LeafCorrespondence:transverseDJ:c}
	\Phi_i\circ I_Q=I_{Q_i}\circ\Phi_{i\perp},
\end{equation}
where $I_\calQ:L|_\calQ\to L$ and $I_{\calQ_i}:L_i|_{\calQ_i}\to L_i$ are the regular LB morphisms given by the inclusions.
\end{itemize}
Now, using Equation~\eqref{eq:proof:theor:LeafCorrespondence:transverseDJ:c} and the identity $\Phi_0^!\calL_0=(\Phi_1^!\calL_1)^\varpi$, one can easily compute
\begin{equation*}
	\Phi_{0\perp}^!I_{\calQ_0}^!\calL_0=I_\calQ^!\Phi_0^!\calL_0=I_\calQ^!((\Phi_1^!\calL_1)^\varpi)=(I_\calQ^!\Phi_1^!\calL_1)^{I_\calQ^\ast\varpi}=(\Phi_{1\perp}^!I_{\calQ_1}^!\calL_1)^{I_\calQ^\ast\varpi}.
\end{equation*}
This immediately implies that the transverse Dirac--Jacobi structures $I_{\calQ_0}^!\calL_0$ and $I_{\calQ_1}^!\calL_1$ are isomorphic, i.e.
\begin{equation}
	\label{eq:isomorphism_transverseDJ}
	I_{\calQ_0}^!\calL_0=(\Psi^!(I_{\calQ_1}^!\calL_1))^B
\end{equation}
with LB iso $\Psi:=\Phi_{1\perp}\circ\Phi_{0\perp}^{-1}\colon L_0|_{\calQ_0}\to L_1|_{\calQ_1}$ and closed Atiyah form $B:=(I_\calQ\circ\Phi_{0\perp}^{-1})^\ast\varpi\in\Omega^2_D(L_0|_{\calQ_0})$.
\end{proof}

As the next corollary shows, when the two legs of the  weak dual pair~\eqref{eq:theor:LeafCorrespondence:transverseDJ} are Jacobi manifolds, the isomorphism~\eqref{eq:isomorphism_transverseDJ} between the transverse Dirac--Jacobi structures (that, in this case, are Jacobi structures) holds without the involvement of the $B$-field transformation (up to wisely choose the minimal transversals).

\begin{corollary}
	\label{cor:LeafCorrespondence:transverseDJ}
	Consider a weak dual pair connecting two Jacobi manifolds
	\begin{equation}
		\label{eq:cor:LeafCorrespondence:transverseDJ}
		\begin{tikzcd}
			(M_0,L_0,\operatorname{Gr}J_0)&(M,L,\operatorname{Gr}\varpi)\arrow[l, swap, "\Phi_0"]\arrow[r, "\Phi_1"]&(M_1,L_1,\operatorname{Gr}(-J_1)),
		\end{tikzcd}
	\end{equation}
	and assume also that the underlying maps $\begin{tikzcd}
		M_0&M\arrow[l, swap, "\varphi_0"]\arrow[r, "\varphi_1"]&M_1
	\end{tikzcd}$ have connected fibers.
	Let $\calS_0$ and $\calS_1$ be corresponding characteristic leaves as in Theorem~\ref{theor:CharacteristicLeafCorrespondence_I}.
	Then there exist minimal transversals $\calQ_i$ to $\calS_i$, for $i=0,1$, and a line bundle isomorphism $\Psi:L_0|_{\calQ_0}\to L_1|_{\calQ_1}$ such that
	\begin{equation*}
		I_{\calQ_0}^!{\operatorname{Gr}J_0}=\Psi^!(I_{\calQ_1}^!{\operatorname{Gr}J_1}),
	\end{equation*}
	where $I_{\calQ_0}:L_0|_{\calQ_0}\to L_0$ and $I_{\calQ_1}:L_1|_{\calQ_1}\to L_1$ denote the regular LB morphisms given by the inclusions.
\end{corollary}

\begin{proof}
	Let us start by setting $\calS=\varphi_1^{-1}(\calS_1)=\varphi_2^{-1}(\calS_2)$.
	Further, let us fix a point $x\in\calS$, and set $x_0:=\varphi_0(x)\in\calS_0$ and $x_1:=\varphi_1(x)\in\calS_1$.
	
	Let us notice that, in view of Lemma~\ref{lem:LeavesCorrespondence_II} and the proof of Theorem~\ref{theor:CharacteristicLeafCorrespondence_I}, one gets that $\calS$ is a leaf of the singular foliation $\calF$ on $M$ given by $T\calF=\ker T\varphi_0+\ker T\varphi_1=\sigma(\ker D\Phi_0+\ker D\Phi_0)$.
	In other words
	\begin{equation}
		\label{eq:proof:cor:LeafCorrespondence:transverseDJ:1}
		TS=(T\calF)|_\calS.
	\end{equation}


	Since the legs of the weak dual pairs~\eqref{eq:cor:LeafCorrespondence:transverseDJ} are Jacobi manifolds, from Remark~\ref{rem:full_contact_dual_pairs} it follows that the presymplectic Atiyah form $\varpi$ is regular, with $\ker\varpi^\flat\subset\ker D\Phi_0\cap\ker D\Phi_1$.
	So, in particular, $\mathbbm{1}\notin\ker\varpi^\flat$ and $\langle\mathbbm{1}\rangle^{\perp\varpi}$ is a corank $1$ subbundle of $DL$.
	From this and $\ker D\Phi_0\subset(\ker D\Phi_1)^{\perp\varpi}$ it follows that
	\begin{equation*}
		DL=\langle\mathbbm{1}\rangle^{\perp\varpi}+\ker D\Phi_0=\langle\mathbbm{1}\rangle^{\perp\varpi}+\ker D\Phi_1=\langle\mathbbm{1}\rangle^{\perp\varpi}+\ker D\Phi_0+\ker D\Phi_1,
	\end{equation*}
	i.e.~$\mathbbm{1}\notin(\ker D\Phi_i)^{\perp\varpi}$, with $i=0,1$.
	Therefore, we also obtain the following
	\begin{equation}
		\label{eq:proof:cor:LeafCorrespondence:transverseDJ:2}
		TM=\sigma(\langle\mathbbm{1}\rangle^{\perp\varpi})+\sigma(\ker D\Phi_0+\ker D\Phi_1)=\ker\theta+T\calF,
	\end{equation}
	where $\theta\in\Omega^1(M;L)$ denotes the regular precontact form on $M$ such that $\iota_{\mathbbm{1}}\varpi=\sigma^\ast\theta$.
	
	Now, as a consequence of Equations~\eqref{eq:proof:cor:LeafCorrespondence:transverseDJ:1} and~\eqref{eq:proof:cor:LeafCorrespondence:transverseDJ:2}, we can choose a minimal transversal $\calQ$ to $\calS$ in $M$ at $x$ such that $T\calQ\subset(\ker\theta)|_\calQ=\sigma(\ker(\sigma^\ast\theta))|_\calQ$.
	Consequently, one gets $I_\calQ^\ast\varpi=0$ as follows
	\begin{equation}
		\label{eq:proof:cor:LeafCorrespondence:transverseDJ:3}
		I_\calQ^\ast\varpi=I_\calQ^\ast(\rmd_D(\iota_{\mathbbm{1}}\varpi)=I_\calQ^\ast(\rmd_D(\sigma^\ast\theta)=\rmd_D(\sigma^\ast(I_\calQ^\ast\theta))=0,
	\end{equation}
	where $I_\calQ:L|_\calQ\to L$ denotes the regular line bundle morphism given by the inclusion.
	Further, up to restrict to working in a small neighborhood of $x$ in $M$, we also obtain, for $i=0,1$,
	\begin{itemize}
		\item a minimal transversal $\calQ_i:=\varphi_i(\calQ)$ to $\calS_i$ in $M_i$ at $x_i$, and
		\item a LB isomorphism $\Phi_{i\perp}:L|_\calQ\to L_i|_{\calQ_i}$, over a diffeomorphism $\varphi_{i\perp}:\calQ\to \calQ_i$, satisfying~\eqref{eq:proof:theor:LeafCorrespondence:transverseDJ:c}.
	\end{itemize}
	Then, the proof of Theorem~\ref{theor:LeafCorrespondence:transverseDJ} shows that the transverse Dirac--Jacobi structures are related as follows
	\begin{equation}
		\label{eq:proof:cor:LeafCorrespondence:transverseDJ}
		I_{\calQ_0}^!{\operatorname{Gr}J_0}=(\Psi^!(I_{\calQ_1}^!{\operatorname{Gr}J_1}))^B
	\end{equation}
	with LB iso $\Psi:=\Phi_{1\perp}\circ\Phi_{0\perp}^{-1}\colon L_0|_{\calQ_0}\to L_1|_{\calQ_1}$ and closed Atiyah form $B:=(I_\calQ\circ\Phi_{0\perp}^{-1})^\ast\varpi\in\Omega^2_D(L_0|_{\calQ_0})$.
	From Equation~\eqref{eq:proof:cor:LeafCorrespondence:transverseDJ:2} it follows that $B=0$ and this completes the proof.
\end{proof}

\appendix

\section{Presymplectic Leaf Correspondence for Weak Dual Pairs in Dirac Geometry}
\label{sec:presymplectic_leaf_correspondence}

The arguments used in Section~\ref{sec:characteristic_leaf_correspondence} can be easily adapted to get an analogue characteristic leaf correspondence theorem for weak dual pairs in Dirac geometry (in the sense of~\cite{FM2018}).
Therefore, for future reference and the reader's convenience, this Appendix describes the Presymplectic Leaf Correspondence Theorem for weak dual pairs in Dirac geometry (not yet present in literature).
Since the proofs are immediately obtained, mutatis mutandis, from the ones of the analogue results in Section~\ref{sec:characteristic_leaf_correspondence}, we omit them.

Consider a weak dual pair of forward Dirac maps with connected fibers
\begin{equation}
	\label{eq:weak_dual_pair:Dirac}
	\begin{tikzcd}
		(M_0,\calL_0)&(M,\operatorname{Gr}\omega)\arrow[l, swap, "\varphi_0"]\arrow[r, "\varphi_1"]&(M_1,\calL_1^\mathrm{opp}).
	\end{tikzcd}
\end{equation}
Then one can easily prove the following three theorems describing the presymplectic leaf correspondence.
\begin{theorem}
	\label{theor:CharacteristicLeafCorrespondence_I:Dirac}
	The relation $\varphi_1^{-1}\calS_1=\varphi_0^{-1}\calS_0$ establishes a 1-1 correspondence
	\begin{equation*}
		\text{presymplectic leaves $\calS_0$ of $(M_0,\calL_0)$}\rightleftharpoons\text{presymplectic leaves $\calS_1$ of $(M_1,\calL_1)$}.
	\end{equation*}
\end{theorem}

\begin{theorem}
	\label{theor:CharacteristicLeafCorrespondence_II:Dirac}
	Let $\calS_0$ and $\calS_1$ be corresponding presymplectic leaves.
	Then the presymplectic structures $\omega_i$ inherited by $\calS_i$, with $i=0,1$, are related as follows 
	\begin{equation*}
		i_\calS^\ast\omega=\varphi_0|_{\calS_0}^\ast\omega_0 - \varphi_1|_{\calS_1}^\ast\omega_1\in\Omega^2(\calS).
	\end{equation*}
\end{theorem}

\begin{theorem}
	Let $\calS_0$ and $\calS_1$ be corresponding presymplectic leaves.
	Then the Dirac structure transverse to $\calS_0$ in $(M_0,\calL_0)$ is isomorphic to the Dirac structure transverse to $\calS_1$ in $(M_1,\calL_1)$, i.e., for any minimal transversals $\calQ_i$ to $\calS_i$, with $i=0,1$, there exist a diffeomorphism $\psi:\calQ_0\to\calQ_1$ and a closed $2$-form $B\in\Omega^2(\calQ_0)$ such that
	\begin{equation}
		\label{eq:isomorphism_transverse:Dirac}
		i_{\calQ_0}^!{\operatorname{Gr}\pi_0}=(\psi^!(i_{\calQ_1}^!{\operatorname{Gr}\pi_1}))^B,
	\end{equation}
	where $i_{\calQ_0}:\calQ_0\to M_0$ and $i_{\calQ_1}:\calQ_1\to M_1$ denote the inclusion maps.
\end{theorem}

As in the next corollary, when the two legs of the  weak dual pair~\eqref{eq:weak_dual_pair:Dirac} are Poisson manifolds, the isomorphism~\eqref{eq:isomorphism_transverse:Dirac} between the transverse Dirac structures (that, in this case, are Poisson structures) holds without the involvement of the $B$-field transformation (up to wisely choose the minimal transversals).

\begin{corollary}
	\label{cor:LeafCorrespondence:transverse:Dirac}
	Assume that the weak dual pair~\eqref{eq:weak_dual_pair:Dirac} connects two Poisson manifolds, i.e.~$\calL_i=\operatorname{Gr}\pi_i$ with $\pi_i\in\frakX^2(M_i)$, for $i=0,1$.
	Let $\calS_0$ and $\calS_1$ be corresponding presymplectic leaves.
	Then there exist minimal transversals $\calQ_i$ to $\calS_i$, with $i=0,1$, and a diffeomorphism $\psi:\calQ_0\to \calQ_1$ such that
	\begin{equation*}
		i_{\calQ_0}^!{\operatorname{Gr}\pi_0}=\psi^!(i_{\calQ_1}^!{\operatorname{Gr}\pi_1}).
	\end{equation*}
\end{corollary}

\small
\addtocontents{toc}{\SkipTocEntry}
\section*{Acknowledgements}

We are grateful to María Amelia Salazar, Luca Vitagliano and Cornelia Vizman for useful discussions and helpful suggestions.
The first author is supported by the DFG research training group ``gk1821: Cohomological Methods in Geometry''.
The second author is supported by an FWO postdoctoral fellowship, further he is member of the National Group for Algebraic and Geometric Structures, and their Applications (GNSAGA – INdAM).
The authors also acknowledge the partial support by the FWO research project G083118N (Belgium) and the hospitality by the DipMat at the University of Salerno they received during the early phases of the preparation of this paper.

\bibliographystyle{plain}
\bibliography{references}

\end{document}